\documentclass[%
 aip,
 amsmath,amssymb,
 reprint,
]{article}

\usepackage{amsthm,amsmath,stmaryrd,bbm,hyperref,geometry,color}
\usepackage[utf8]{inputenc}
\usepackage[english]{babel}
\usepackage{graphicx}
\usepackage{amsfonts,amssymb}
\usepackage{verbatim}
\usepackage{enumitem}
\usepackage{algorithm}
\usepackage{authblk}
\usepackage{algpseudocode}
\usepackage{multirow}
\usepackage{xcolor}

\usepackage{caption}
\usepackage{subcaption}

\newcommand{\po}{\left(}
\newcommand{\pf}{\right)}
\newcommand{\co}{\left[}
\newcommand{\cf}{\right]}
\newcommand{\cco}{\llbracket}
\newcommand{\ccf}{\rrbracket}
\newcommand{\R}{\mathbb R} 
\newcommand{\T}{\mathbb T} 
\newcommand{\Z}{\mathbb Z}
\newcommand{\PP}{\mathbb P}
\newcommand{\E}{\mathbb E} 
 
\newcommand{\N}{\mathbb N} 
\newcommand{\dd}{\text{d}}
\newcommand{\na}{\nabla}
\newcommand{\1}{\mathbbm{1}}
\newcommand{\veps}{\varepsilon}

\newtheorem{thm}{Theorem}
\newtheorem{assu}{Assumption}
\newtheorem{lem}[thm]{Lemma}

\newtheorem{prop}[thm]{Proposition}
\newtheorem{corollary}[thm]{Corollary}

\newcommand{\new}[1]{#1}

\hypersetup{
    unicode=false,          
    pdftoolbar=true,        
    pdfmenubar=true,      
    pdffitwindow=false,    
    pdfstartview={FitH},
    citecolor=blue,
    colorlinks=true,       
    linkcolor=blue,
    filecolor=blue,      
    urlcolor=blue          
}

\title{The velocity jump Langevin process and its splitting scheme: long time convergence and numerical accuracy}

\author[1,2,3]{Nicolaï Gouraud}
\author[4]{Lucas Journel}
\author[1,2,5]{Pierre Monmarch{\'e}}
\affil[1]{Sorbonne Université, LCT, UMR 7616 CNRS, Paris, France}
\affil[2]{Sorbonne Université, LJLL, UMR 7598 CNRS, Paris, France}
\affil[3]{Qubit Pharmaceuticals, Advanced Research Department, Paris, France}
\affil[4]{Institut de Mathématiques, Université de Neuchâtel, Switzerland}
\affil[5]{Institut Universitaire de France, Paris, France}

\begin{document}

\maketitle

\begin{abstract}
    The Langevin dynamics is a diffusion process  extensively used, in particular in molecular dynamics simulations, to sample Gibbs measures. Some alternatives based on (piecewise deterministic) kinetic velocity jump processes have gained interest over the last decade. One interest of the latter is the possibility to split forces (at the continuous-time level), reducing the numerical cost for sampling the trajectory. Motivated by this, a numerical scheme based on hybrid dynamics combining velocity jumps and Langevin diffusion, numerically more efficient than their classical Langevin counterparts, has been introduced for computational chemistry in \cite{Weisman}. The present work is devoted to the numerical analysis of this scheme. Our main results are, first, the exponential ergodicity of the continuous-time velocity jump Langevin process, second, a Talay-Tubaro expansion of the invariant measure of the numerical scheme \new{on the torus}, showing in particular that the scheme is of weak order $2$ in the step-size and, third, a bound on the quadratic risk of the corresponding practical MCMC estimator (possibly with Richardson extrapolation). With respect to previous works on the Langevin diffusion, new difficulties arise from the jump operator, which is  non-local.
\end{abstract}

\section{Motivations and overview}

Molecular dynamics (MD) is a popular numerical tool to infer macroscopic properties of matter from simulations at the microscopic level. In the framework of statistical physics, thermodynamical quantities are seen as average values of certain functions called observables with respect to a probability measure. The \new{primary objective} of MD is to efficiently sample these measures, \new{typically} by computing long trajectories of stochastic processes. Here, we focus on the so-called canonical ensemble, in which the number of particles, the volume of the system
and its temperature are fixed.

Let us consider a system of $N$ interacting particles. \new{Denote by $x\in\mathcal{X}$ their positions and $v\in\R^{3N}$ their velocities, with $\mathcal X = \R^{3N}$ or $\mathcal{X} = \T^{3N}=\R^{3N}/\Z^{3N}$ if periodic boundary conditions are enforced.} Let $M = \text{diag}(m_1 I_3,\dots, m_N I_3)$ be the mass matrix of the particles, $\beta = 1/(k_B T)$ the inverse temperature of the system (where $k_B$ is the Boltzmann constant), $U:\R^{3N}\rightarrow \R$ the potential energy function encoding the interactions between particles and finally $H(x,v) =  U(x) + \frac{1}{2}v^T M v$ the Hamiltonian of the system, corresponding to its total energy. 
In the canonical ensemble, \new{the statistical distribution of microscopic configurations is described by} the Boltzmann-Gibbs measure, defined by
\begin{equation}\label{eq:Gibbs-measure}
    \dd\mu(x,v) = \mathcal{Z}_\mu^{-1}\exp(-\beta H(x,v))\dd x\dd v\,,
\end{equation}
where $\mathcal{Z}_\mu = \int_{\R^{6N}} \exp(-\beta H(x,v))\dd x\dd v$. \new{In most cases, computing macroscopic quantities of the form $\E_\mu[\varphi]$ analytically or with deterministic numerical quadrature is out of reach, as they are} high-dimensional integrals involving an unknown constant $\mathcal{Z}_\mu$. \new{Instead, they can be approximated} by simulating a Markov process $(X_t,V_t)_{t\geqslant 0}$ that is ergodic with respect to $\mu$, which means that for any $\varphi$ (in a certain class of functions), almost surely:
\[\frac{1}{t}\int_0^t \varphi(X_s,V_s)\dd s \underset{t\rightarrow\infty}\longrightarrow \int_{\R^{6N}} \varphi(x,v)\dd \mu(x,v) = \E_\mu[\varphi].\]
A popular process that has this property (under mild conditions on the potential $U$) is the kinetic Langevin diffusion, defined as the solution of the following SDE:
\begin{equation}\label{langevin}
\left\{\begin{array}{rcl}
\dd X_t & = &  V_t \dd t \\
\dd V_t & = & -M^{-1}\nabla U (X_t) \dd t - \gamma M^{-1} V_t \dd t - M^{-1}\sqrt{2\gamma\beta^{-1}}\dd \new{B_t} \,,
\end{array}\right.
\end{equation}
where $\new{(B_t)_{t\geqslant 0}}$ is a standard Brownian motion in $\R^{3N}$ and $\gamma > 0$ is a friction parameter. \new{In practice, the Langevin diffusion cannot be simulated exactly, leading to discretization splitting schemes such as  BAOAB~\cite{BouRabee,baoab1,baoab2,stoltzsplitting} where the parts of the dynamics corresponding  to the free transport, the forces and the friction/dissipation  are  simulated separately.}

When simulating the dynamics, the most expensive part is the computation of the forces $\nabla U$. A common approach to reduce the cost is to substitute this gradient with a stochastic version, as in the stochastic gradient descent in optimization (ubiquitous in machine learning \cite{gradient-sto}, see also examples in MD in \cite{rbm,ewald-randombatch,poier-gradsto}), or to treat various parts of the potentials with different time steps (multi-time-step methods), see~\cite{multitimestep}. We are interested here in an alternative approach that involves replacing \eqref{langevin} with a hybrid model combining a classical Langevin diffusion and a piecewise deterministic Markov process. This hybrid model still samples from $\mu$ but can be simulated using a numerical splitting scheme that requires fewer gradient computations per time step. 

The idea involves decomposing $\na U$ as a sum of vector fields $(F_i)_{0\leqslant i\leqslant K}$ for some $K\in \N$, where $F_0$ denotes the computationally inexpensive component (typically in MD the short-range forces exhibiting fast variation), while the $F_i$ terms, for $i\geqslant 1$, represent long-range forces, which are computationally more demanding than $F_0$ (as each atom interacts with all others through these forces, unlike the short-range ones). Then, we define the velocity jump Langevin process $(X_t,V_t)_{t\geqslant 0}$ as the Markov process which follows the Langevin diffusion process associated with the force $F_0$, with the velocity $V_t$ undergoing additional random jumps at rate $\lambda_i(x,v)$ following a jump kernel $q_i(x,v,\dd v')$ (both defined below in~\eqref{eq:lambdaiqi}), that depend on $F_i$ in a way that ensures that the equilibrium measure of the process is indeed the Gibbs measure $\mu$. 

This process can be simulated with a splitting scheme similar to BAOAB, but with an additional part arising from the velocity jumps. The simulation \new{of the jump times} is based on the thinning method, see~\cite{poisson1,poisson2}. Suppose that $\lambda_i(x,v)\leqslant\overline{\lambda_i}$ for $1\leqslant i \leqslant K$, and denote $\overline{\lambda} = \sum_i \overline{\lambda_i}$. A direct computation on the generator shows that the jumps can be exactly simulated this way: starting from $(x,v)$,
\begin{enumerate}
    \item Draw $\mathcal{E}$ a standard exponential random variable, and let $T = \mathcal{E}/\overline{\lambda}$ be the next jump time proposal.
    \item Draw $I$ in $\cco 1,K\ccf$ such that $\PP (I=i) = \overline{\lambda_i}/\overline{\lambda}$. Propose a jump of type $I$ at time $T$.
    \item Accept the jump with probability $\lambda_I(x,v)/\overline{\lambda_I}$, in which case the velocity is resampled at time $T$ according to $q_I(x,v,\dd v')$, otherwise the velocity at time $T$ is simply $v$ (there is no jump).
\end{enumerate}
Unlike a traditional numerical scheme for the Langevin diffusion (like BAOAB), where the gradient $\nabla U$ is computed at every time step, here we only need to evaluate $F_i$ when a jump of type $i$ is proposed (at step 3 above), which does not occur at every time step for every $i$. However, similarly to BAOAB and contrary to multi-time-step splitting methods, there is still a unique step-size in the discretization and no additional parameters to tune: the frequency at which $F_i$ is evaluated is random and is  adapted to each force through the bound $\overline{\lambda_i}$. By contrast, when using piecewise deterministic Markov processes such as the Bouncy Particle (BPS)~\cite{bps,M24,peters2012rejection} or Zig-Zag \cite{ZigZag} samplers (which can be seen as particular cases of the velocity jump Langevin diffusion where $F_0=0$ and $\gamma=0$, so that, between velocity jumps, the deterministic motion is simply the free transport $\dot x= v$), thinning is usually employed to sample exactly the full continuous-time process, without any time-discretization. However, when $F_0\neq 0$, this is only possible when it is a linear function (as in the so-called boomerang sampler \cite{pmlr-v119-bierkens20a}), which is not the case we are interested in. Besides, \new{note} that the use of discretization schemes  even for piecewise deterministic Markov processes have recently gained interest  \cite{bertazzi2021approximations,M45,corbella_automatic,cotter2020nuzz}.

\new{
The idea of the hybrid velocity jump Langevin approach was first introduced in~\cite{Weisman}, in the particular case where the jump mechanism is the one of the BPS, with a decomposition of the forces based on long-range pairwise interactions. Later, in~\cite{GouraudNumeric}, a generalized version was introduced in which the jump mechanism, originally defined in~\cite{MRZ}, can be seen as an interpolation between the BPS and the Hamiltonian dynamics, which is shown to be numerically more stable, while preserving dynamical properties, also of interest in MD.

While~\cite{GouraudNumeric} presents an implementation of the Langevin velocity jump approach on a molecular dynamics code and numerical experiments, the present work is devoted to its mathematical study.} We prove a second order weak error expansion \textit{à la Talay-Tubaro} \cite{talay1990expansion} of the discretization bias on the invariant measure, see Theorem~\ref{thm:TalayTubaro}. The proof is based on  classical weak backward error analysis arguments \cite{kopec2015weak,stoltzsplitting,MST,talay1990expansion} and, in particular, crucially relies on the ergodicity properties of the continuous-time process, which is of interest by itself and is in our case the topic of Theorem~\ref{thm:estimates}.

By comparison with previous works, the main new difficulties arise from the combination of, first,  the degenerate diffusion feature of the kinetic Langevin diffusion, second, the non-local  jump operator and, third, the necessity to get pointwise bounds on derivatives of the semi-group. Indeed, the exponential $L^2$ convergence  estimates obtained for kinetic piecewise deterministic Markov processes in recent works  \cite{Andrieuetal,MRZ} with the Dolbeault-Mouhot-Schmeiser modified norm  \cite{dolbeault2015hypocoercivity} are not sufficient to conduct the rest of the argument, and neither are the results in $V$-norm based on Harris theorem \cite{bierkens2019ergodicity,Calvez,BPS_ergo_Doucet,BPSDGM,M24} or in Wasserstein distances using coupling \cite{deligiannidis2021randomized}. We thus have to work with the $H^1$ modified norm of Villani \cite{villani}, or rather with its extension to higher order Sobolev spaces \cite{Journel1,Journel2,Zhang}. Thus, one of our main contributions is that we successfully implemented this method in a case with a non-local operator (i.e. a non-diffusive Markov process), leading to Theorem~\ref{thm:conv-Hk} (which is an important step in the proof of Theorem~\ref{thm:estimates}). Although the $H^1$ modified norm method has mostly been used for diffusion processes, it has  already been successfully  applied to non-diffusive dynamics in \cite{deligiannidis2021randomized,evans2021hypocoercivity,M14,M21}; however, let us emphasize that, contrary to our case, in  all these works the dynamics is a contraction of the Wasserstein 2 distance, from which the exponential decay of the $H^1$ norm is clear, and in fact the computations to prove the decay of the Wasserstein distance and of the gradient part of the $H^1$ norm are essentially the same (as discussed in \cite{M14,M27,M44}). In particular, in these cases, to get the decay of the $H^1$ norm,  the dissipation of the $L^2$ part of the norm is not exploited, which  in the general case (i.e. without  a Wasserstein 2 contraction) is a crucial feature of Villani's method (as in Lemma~\ref{lem:initialisation} below). Note that this is possible in our case because, although the process undergoes non-local jumps, it also has a diffusive part, from which the $L^2$ dissipation is bounded below by a first-order term which would be the dissipation of the usual Langevin diffusion (cf. the \emph{carré du champ} operator~\eqref{eq:Gamma}).

A second important difficulty is the design of a suitable Lyapunov function for the velocity jump Langevin diffusion, given its combination of diffusive, non-local and kinetic features (see Propositions~\ref{prop:deflyapunov} for the continuous-time process and Proposition~\ref{prop:lyapunovscheme} for the numerical scheme). The last important ingredient for establishing our main result is a finite-time numerical analysis of the discretization bias, Theorem~\ref{thm:finitetimeerror}, whose proof is close  to the proof of  Theorem~2.6 of the recent \cite{M45} (which is a similar result for piecewise deterministic samplers). 

\medskip

This work is organized as follows. \new{In Section~\ref{s-sec:maths-set}}, the process and its splitting scheme are \new{defined}, our main results are stated, \new{and some comments on the method from the applicable angle are given}. Section~\ref{sec:ergodicity} is devoted to the proof of Theorem~\ref{thm:estimates} (the long-time convergence of the continuous-time process). The weak error result, Theorem~\ref{thm:TalayTubaro}, is proven in Section~\ref{sec:weakerror}.

\new{\section{Setting and main results}\label{s-sec:maths-set}}

\subsection{Notations}

In all the following, $|\cdot|$ denotes the Euclidean norm and $\cdot$ the standard dot product. In the case of multi-indices $\alpha\in\N^{d}$, we also denote $|\alpha| = \sum_{i=1}^d\alpha_i$. The set of smooth functions with compact support from $\R^{2d}$ to $\R$ is denoted $\mathcal C_c^{\infty}(\R^{2d})$. For any operators $A$ and $B$, we denote $[A,B] = AB-BA$ their commutator. In the remainder of this work, we will denote by $C$ various constants that may change from line to line.

\new{\subsection{Definitions and assumptions}}

Fix some smooth potential $U\in\mathcal C^\infty(\mathcal X , \R)$ where either $\mathcal X=\R^d$ or $\mathcal X=\T^d$ for some $d\geqslant 1$. In the case of particle systems, $d=3N$.  For simplicity, in the remainder of this work, we fix $\beta = 1$, $M = I_d$ and consider the following decomposition of the forces. Assuming that $U=U_0+U_1$ with $U_0,U_1$ smooth, \new{let}
\begin{equation}\label{eq:decomposition}
F_0 = \nabla U_0,\quad F_i =\partial_i U_1 e_i,\,i\in\cco 1,d\ccf\,,
\end{equation}
where $(e_i)_{1\leqslant i\leqslant d}$ stands for the canonical basis of $\R^d$.
In other words, we decompose $\nabla U_1$ in a similar fashion as in the Zig-Zag process: 
\begin{equation*}\label{zigzag}
\nabla U_1(x) = \sum_{i=1}^d  \partial_i U_1(x) e_i\,.
\end{equation*}
Fix a friction parameter $\gamma>0$, a jump parameter $\rho\in[-1,1)$ and an activation function $\Psi\in\mathcal C^\infty(\R,\R_+)$. The velocity jump Langevin process is defined as the Markov process on $\mathcal X\times\R^d$ with infinitesimal generator given by 
\begin{equation}
    \label{eq:generator}
    \mathcal{L} = \mathcal{L}_H+\mathcal{L}_D+\mathcal{L}_J\,,
\end{equation}
where for any function $h\in\mathcal C^{\infty}_c\po\mathcal X\times \R^{d} \pf$:
\[
\begin{array}{rcll}
\mathcal{L}_H h(x,v) & = & v\cdot \nabla_x h(x,v) - \nabla U_0(x)\cdot\nabla_v h(x,v) & \text{ (Hamiltonian dynamics)}\\& & \\
\mathcal{L}_D h(x,v) & = & - \gamma v\cdot\nabla_v h(x,v) + \gamma \Delta_v h(x,v) & \text{ (friction/dissipation)}\\ & &\\
\mathcal{L}_J h(x,v) & = &\sum_{i=1}^d \mathcal{L}_J^i h(x,v) & \text{ (velocity jumps)}
\end{array}\]
with
\[
\mathcal{L}_J^i h(x,v) = \lambda_i(x,v) \int_{\R^d} \po h(x,v')-h(x,v)\pf q_i(x,v,\dd v'),
\]
where $\lambda_i,q_i$ satisfy
\begin{equation}
    \label{eq:lambdaiqi}
    \lambda_i(x,v)\int_{\R^d} h(x,v')q_i(x,v,\dd v') = \frac{2}{1-\rho} \E \co h\po x,V^i\pf \Psi\po\frac{\partial_i U_1(x)}{2}\po v_i-V^i_i\pf\pf\cf,
\end{equation}
and where
\[V_i^i = \rho v_i + \sqrt{1-\rho^2}\new{\xi}\,,\qquad V^i_j = v_j\quad\forall j\neq i,\]
and $\new{\xi}\sim\mathcal{N}(0,1)$ is a one-dimensional standard Gaussian variable. In the definition above, $q_i(x,v,\cdot)$ is a probability measure for all $(x,v)\in\mathcal X\times\R^{d}$, in other words $\lambda_i$, for all $i\in \cco 1,d \ccf$, is obtained by taking $h=1$ in the formula above. 

We will work under the following set of assumptions on $U$ and $\Psi$.

\begin{assu}\label{assu}
\begin{enumerate}
    \item the Gibbs measure is well defined, i.e. $\int_{\mathcal X}e^{-U(x)} \dd x <\infty$.
    \item The derivatives of $U_1$ of all order are bounded.
    \item $\na^2 U_0$ as well as its derivatives of all order are bounded, and there exist $\kappa>0$ and a compact set $\mathcal K\subset \mathcal X $ such that for all $x\notin \mathcal K$
\[-x\cdot \na U_0(x) \leqslant  -\kappa |x|^2. \]
    \item \new{$\Psi$ has all its derivatives bounded}, and for all $s\in\R$,
\begin{equation}\label{eq:Psi}
\Psi(s) - \Psi(-s) = s.
\end{equation}
\end{enumerate}
\end{assu}

Under Assumption~\ref{assu}, it is not difficult to see that the process is well-defined, see~\cite{MRZ}. Let $(X_t,V_t)$ be a Markov process with generator $\mathcal L$, and let $(P_t)_{t\geqslant 0}$ be the associated semi-group,
\[
P_tf(x,v) = \E_{(x,v)}\co f(X_t,V_t) \cf,
\]
defined for suitable $f:\mathcal X\times\R^{d}\to \R$. Define as well the set of admissible functions:
\[
 \mathcal A :=  \left\{ f\in\mathcal C^{\infty}(\mathcal X\times \R^d,\R)|\forall \alpha\in\N^{2d},\,\exists\, C>0,\,0<c<1,\, |\partial^{\alpha} f|\leqslant Ce^{cH}\right\}.
\]

Let us now introduce a splitting scheme for the process.  Starting from~\eqref{eq:generator}, we further split the Hamiltonian part as $\mathcal L_H = \mathcal L_A + \mathcal L_B$ where $\mathcal L_A = v\cdot\na_x$ stands for the free transport and $\mathcal L_B=-\na U_0\cdot\na_v$ for the acceleration. Motivated by the Trotter/Strang formula
\begin{equation}
    \label{eq:TrotterStrang}
    e^{t\mathcal L} = e^{\frac t2\mathcal L_B} e^{\frac t2\mathcal L_J}  e^{\frac t2\mathcal L_A} e^{t\mathcal L_O} e^{\frac t2\mathcal L_A} e^{\frac t2\mathcal L_J} e^{\frac t2\mathcal L_B} + \underset{t\rightarrow 0}{\mathcal O}(t^3)\,,
\end{equation}
given a step-size $\delta>0$, we call BJAOAJB the Markov chain $(\overline X_n,\overline V_n)$ on $\mathcal X\times\R^d$ with transition kernel
\begin{equation}\label{eq:transition-scheme}
Q  =  e^{\frac \delta 2\mathcal L_B} e^{\frac \delta  2\mathcal L_J}  e^{\frac \delta  2\mathcal L_A} e^{\delta  \mathcal L_O} e^{\frac \delta  2\mathcal L_A} e^{\frac \delta 2\mathcal L_J} e^{\frac \delta 2\mathcal L_B}\,. 
\end{equation}
From the definition of $\mathcal L_A$, $\mathcal L_B$, $\mathcal L_J$ and $\mathcal L_O$, we have, for $t>0$,
\begin{align*}
    e^{t \mathcal L_B} f\po x,v\pf &= f\po x,v-t\nabla U_0(x)\pf, \\ 
    e^{t \mathcal L_A} f\po x,v\pf &= f\po x+tv,v\pf, \\
    e^{t \mathcal L_O} f\po x,v\pf &= \mathbb E \co f(x,e^{-\gamma t}v+\sqrt{1-e^{-2\gamma t}}\new{\xi})\cf\qquad \text{with }\new{\xi}\sim\mathcal N(0,I_d),  \\
    e^{t \mathcal L_J} f\po x,v\pf &= \mathbb E \co f(x,W_t)\cf, 
\end{align*} 
where $(W_s)_{s\geqslant 0}$ is a continuous-time Markov chain on $\R^d$ initialized at $W_0=v$ and with jump rate $\lambda(x,\cdot)$ and jump kernel $q(x,\cdot,\cdot)$ given by 
\[\lambda(x,w) = \sum_{i=1}^{d} \lambda_i(x,w)\,,\qquad \lambda(x,w) q(x,w,\dd w') = \sum_{i=1}^d  \lambda_i(x,w)  q_i\po x,w,\dd w'\pf  \,. \]
Hence, one transition of the chain $(\overline X_n,\overline V_n)$ is given by the succession of steps BJAOAJB, where:
\begin{align*}
    (B)&: v\gets v-\frac{\delta}2\nabla U_0(x) \\
    (J)&: v \gets W_{\delta/2}\\
    (A)&: x\gets x+\frac{\delta}2 v \\ 
    (O)&: v\gets e^{-\gamma\delta}v+\sqrt{1-e^{-2\gamma\delta}}\new{\xi} \,.
\end{align*}

\new{\subsection{Main results}}

Our first main result, established in Section~\ref{sec:ergodicity}, gives point-wise estimates for the exponential long-time convergence of the semi-group, together with all its derivatives.

\begin{thm}\label{thm:estimates}
Under Assumption~\ref{assu}, for all $f\in\mathcal A$ and all multi-index $\alpha\in\N^{2d}$, there exist $C,q>0$, $b\in (0,1)$, such that for all $t>0$, $(x,v)\in\R^{2d}$, 
\begin{equation}
    \label{eq:thm_ergodicite}
    |\partial^{\alpha}(P_tf-\mu(f))(x,v)| \leqslant Ce^{-qt}e^{bH(x,v)}.
\end{equation}
\end{thm}

\new{
Note that although $e^{bH(x,v)}$ grows fast as $x$ and $v$ go to infinity, this term is an optimal prefactor depending on the initial conditions of the process. When $t=0$, the inequality reads $|\partial^\alpha f(x,v)-\mu(f)|\leqslant Ce^{bH(x,v)}$, which corresponds to the assumption $f\in\mathcal{A}$, i.e. that $f$ and its derivatives grow at most exponentially. In particular, the case $\alpha=0$ in \eqref{eq:thm_ergodicite} corresponds to the weak convergence of the process for all test functions in $\mathcal A$. In addition, since $b<1$, a crucial property of the right-hand side of~\eqref{eq:thm_ergodicite} is that it is $\mu$-integrable, which will be used to control the remainder terms in the expansion of the weak error of the numerical scheme. In addition, ergodicity of the continuous-time process is necessary to deduce the expansion of the invariant measure of the numerical scheme from the finite-time error estimates, by letting $t$ go to infinity.}

\medskip

Formally, Equation~\eqref{eq:TrotterStrang} yields that this palindromic form gives rise to a second-order scheme. Our second main result is that, indeed, the discretization bias on the invariant measure is of order $\delta^2$. Here the analysis is restricted to the compact torus, since this is the main case of interest in practice in MD and otherwise the construction of a Lyapunov function is more intricate for the numerical scheme than the continuous-time process (as in other cases such as \cite{Journel2}).

\begin{thm}
\label{thm:TalayTubaro}
In the case $\mathcal X=\T^d$, under Assumption~\ref{assu}, there exists $\delta_0>0$ such that for all $\delta\in(0,\delta_0]$, the BJAOAJB chain with transition kernel $Q$ admits a unique invariant measure $\mu_\delta$. Moreover, for all $f\in\mathcal A$, there exists $c_f\in\R$ such that
\[\mu_\delta(f) = \mu(f) + c_f \delta^2 + \underset{\delta\rightarrow 0}{\mathcal O}(\delta^3)\,.\]
\end{thm}
The proof of this theorem is the topic of Section~\ref{sec:weakerror}.
Along the proof of Theorem~\ref{thm:TalayTubaro}, we have to prove that the Markov chain with transition $Q$ is ergodic. In fact we prove a more precise statement, namely that the BJAOAJB chain converges to its equilibrium in $V$-norm exponentially fast with a rate which, as should be expected, is linear in the step-size, as we state now. Given $V:\T^d\times \R^d\rightarrow [1,\infty)$, for $f:\T^d\times\R^d\rightarrow \R$ and a signed measure $\mu$ over $\T^d\times\R^d$ we write $\|f\|_V=\|f/V\|_\infty$ and $\|\mu\|_V = \sup\{\mu(f),\ \|f\|_V\leqslant 1\}$. The next statement is established in Section~\ref{sec:ergodicityscheme}.

\begin{thm}
\label{thm:BJAOAJBergodic}
Let $b\in(0,1/2)$ and $V(x,v)=e^{b|v|^2}$ for $(x,v)\in\T^d\times\R^d$. In the case $\mathcal X=\T^d$, under Assumption~\ref{assu}, there exist $\delta_0,\lambda,C>0$ such that for all $\delta\in(0,\delta_0],(x,v)\in\T^d\times\R^d$,
\[\|\delta_{(x,v)}Q^n    - \mu_\delta\|_V \leqslant C  e^{-\lambda n\delta} V(x,v)\,.\]
\end{thm}

Combining Theorems~\ref{thm:TalayTubaro} and \ref{thm:BJAOAJBergodic}, classical arguments then allow to bound the quadratic risk of an MCMC estimator based on the BJAOAJB chain, possibly with a Richardson extrapolation, as detailed in Section~\ref{sec:quadraRisk}:

\begin{corollary}
\label{cor:quadraticRisk}
Let $b\in(0,1/4)$ and $V(x,v)=e^{b|v|^2}$ for $(x,v)\in\T^d\times\R^d$.  In the case $\mathcal X=\T^d$, under Assumption~\ref{assu}, there exist $\delta_0>0$ such that for all $f\in\mathcal A$ with $\|f\|_V <\infty$, there exists $C>0$ such that for all $n\in\N$ and all $\delta\in(0,\delta_0)$, given $(X_n,V_n)_{n\in\N}$ a BJAOAJB chain,
 \[\mathbb E \po \left|\frac1n\sum_{k=1}^{n}  f(X_k)-\mu(f)\right|^2\pf \leqslant C \po  \delta^4 + \frac{1}{(n\delta)^2} \mathbb E \po e^{2b|V_0|^2}\pf \pf \,. \]
 Moreover, if $(\widetilde X_n,\widetilde V_n)_{n\in\N}$ is another independent BJAOAJB chain with step-size $\delta/2$,
  \[\mathbb E \po \left|\frac4{3n}\sum_{k=1}^{n}  f(\widetilde X_k)-\frac1{3n}\sum_{k=1}^{n}  f(X_k)-\mu(f)\right|^2\pf \leqslant C \po  \delta^6 + \frac{1}{n\delta}\mathbb E \po e^{2b|V_0|^2} + e^{2b|\widetilde V_0|^2} \pf \pf  \,. \]
 \end{corollary}
 Here we state a simple bound but it is easily checked in the proof that in fact the dependency in the initial condition disappears exponentially fast with $n\delta$, i.e. in the bounds above, $\mathbb E ( e^{2b|V_0|^2})$  can be replaced by $1+ e^{-cn\delta} \mathbb E ( e^{2b|V_0|^2})$ for some $c>0$, and similarly with $\widetilde V_0$.

\new{\subsection{Discussion and applications}

\paragraph{Decomposition of the forces.} Following the construction of the jump mechanism presented in~\cite{MRZ}, the velocity jump Langevin process could be defined for any decomposition of $\na U_1$ into arbitrary vector fields $F_i$. In the present work, the choice of decomposing $\na U_1$ along the canonical basis of $\R^d$ was made to avoid singular terms of the form $\nabla U_1/|\nabla U_1|$ in the hypocoercivity computations. However, each of the $F_i$ in~\eqref{eq:decomposition} could be further decomposed along the vector $e_i$, and all the results presented in the previous sections would still hold, as the proofs would lead to exactly the same computations.

To give an example, consider for instance the following classical functional form
\[ U = U_0 + \sum_{1\leqslant i<j\leqslant d}g(|x_i-x_j|)\,,\]
where $U_0$ gathers short-range, computationally inexpensive terms, to be treated by the Langevin part of the process, and $U_1 = U-U_0$ is a sum of pairwise long-range terms (for instance, in molecular dynamics, long-range electrostatic and van der Waals interactions), with $g:\R\rightarrow\R$. Here, the term ``long-range'' refers to the fact that we suppose that there exists a cutoff radius $r_c>0$ such that $g(r)=0$ for all $r<r_c$, which implies that there are no singularities when $x_i=x_j$ for any $i,j$. Following the decomposition~\eqref{eq:decomposition}, for all $i\in\cco 1,d\ccf$, $F_i(x) = \sum_{j\neq i}\partial_i g(|x_i-x_j|)e_i$. We can then define $F_{ij} = \partial_i g(|x_i-x_j|)e_i$ (namely the force of the particle $j$ acting on the particle $i$), and associate a jump mechanism to each $F_{ij}$.

This type of decomposition was used in the implementations of the method for MD applications presented in~\cite{GouraudNumeric,Weisman}, except that the splitting was done at the level of the three-dimensional atoms instead of individual coordinates.}

\new{
\paragraph{Conditions on $U$.} In general, when using piecewise deterministic samplers such as BPS or Zig-Zag, it is not necessary to assume that $\na U$ is bounded: jump times can be sampled with thinning using a time-inhomogeneous bound (for instance, if $\na^2 U$ is bounded, we can bound $|\na U(x+tv)| \leqslant |\na U(x)|+ t \|\na^2 U\|_\infty |v|$). We could do the same here but this is less convenient because, between jumps, the process follows a random Langevin diffusion and thus we cannot have an almost sure bound on $|\na U_1(x_t)|$. Regardless, this is not our cases of interest, which are the practical settings of \cite{Weisman,GouraudNumeric} in MD simulations. In this situation, $U$ has some singular parts (short-range forces) and thus, even if it is  possible in principle to use thinning, it is not efficient. The point made in \cite{Weisman,GouraudNumeric} is precisely that thinning is efficient for bounded forces, so that in these works $\na U_1$ only gathers long-range interaction forces, which vanish at infinity. This is why,  in Assumption~\ref{assu}, we assume $\na U_1$ to be bounded. In our theoretical study, this condition is used in the proof of Lemma~\ref{lem:commutator-jump}.}

\new{Note however that when $\mathcal  X=\T^d$, the condition on $U$ in Assumption~\ref{assu} is simply that $U_0$ and $U_1$ are smooth (taking $\mathcal K=\T^d$).}

\new{
\paragraph{Choice of $\Psi$.}

The condition~\eqref{eq:Psi} on $\Psi$ is here to ensure the invariance of $\mu$. The usual choice, used in the Zig-Zag or BPS samplers, and in the practical implementations in~\cite{GouraudNumeric,Weisman} is $\Psi(s) = \max(0,s)$, which indeed satisfies~\eqref{eq:Psi}. However, the computations in Sobolev spaces in our proofs require that the generator has smooth coefficients, which prevents such a $\Psi$. An admissible example is given by $\Psi(s) = a\log(\exp(s/a)+1)$ for any $a>0$.}

\new{
\paragraph{Role of $\rho$.}

In the jump mechanism, the case $\rho=-1$ corresponds to the Bouncy Particle Sampler, namely $\lambda_i(x,v) = \Psi(\partial_i U_1(x) v_i)$, with deterministic jumps: $q_i(x,v,\dd v') = \delta_{(x,-v_i)}$. This version, introduced and implemented in~\cite{Weisman} is the most natural one and simplest to implement. Moreover it  yields the least amount of jumps  (since $\lambda_i$ grows with $\rho$), and therefore corresponds to the best computational speedup.

However, as explored in~\cite{GouraudNumeric}, choosing the BPS as a jump mechanism has also several issues. First, numerical instabilities may arise in MD applications. Then, although the process samples the Gibbs measure~\eqref{eq:Gibbs-measure}, dynamical properties such as diffusion constants are not well preserved. Indeed, a specificity of MD applications, and contrary to other fields such as Bayesian statistics, is that sampling~\eqref{eq:Gibbs-measure} or its position  marginal (with density proportional to $e^{-\beta U}$) is not the only objective, and dynamical properties such as transition times or transport coefficients are also quantities of interest. In MD, the Langevin diffusion is in fact a model for the time evolution of molecular systems, and it is usually not possible to replace~\eqref{langevin} by some other Markov sampler of the Gibbs measure, like the overdamped Langevin diffusion, Hamiltonian Monte Carlo or random walk Metropolis. Additionally, in that case, in~\eqref{langevin}, the friction parameter $\gamma$ is fixed instead of chosen by the user.

That being said, the generalized jump mechanism~\eqref{eq:lambdaiqi}, first introduced in~\cite{MRZ} as a general class of kinetic jump processes sampling the Gibbs measure (of which the BPS and Zig-Zag process are particular cases), was integrated to the velocity jump Langevin approach in~\cite{GouraudNumeric}. An interesting property, proved in~\cite[Theorem 3.6]{MRZ}, is that this process can be tuned to be arbitrarily close to the Langevin diffusion in terms of stochastic trajectories: when $\rho\rightarrow 1$, the distribution of the velocity jump Langevin process converges to the Langevin diffusion in the space of \textit{càdlàg} trajectories endowed with the Skorokhod topology. Therefore, the general expression~\eqref{eq:lambdaiqi} can be seen as an interpolation between the BPS and Hamiltonian dynamics associated to $U_1$. As shown numerically in~\cite{GouraudNumeric}, a choice of $\rho$ not too far from $1$ makes the velocity jump Langevin process more suitable to estimate the dynamical statistics, and numerically more stable, with, of course, a trade-off between the accuracy of these dynamical properties and the numerical cost of the simulation, as would be the case with any numerical approximation of~\eqref{langevin}. Indeed, as $\rho$ grows, there are more jumps, therefore more pairwise interactions to compute.

In addition, note that in the present work, we take $\rho\in[-1,1)$ as a constant for simplicity, but it could actually be taken as a function of $x$ or depend on $i\in\cco 1,d\ccf$.}

\new{
\paragraph{Simulation of the jumps.}

As stated in the introduction, when simulating the jump process with generator $\mathcal{L}_J$, jump times can be sampled by using the thinning procedure, provided that a bound on the jump rates $\lambda_i$ is available. In our case, since $\Psi'$ and $\na U_1$ are bounded, for each $i\in\cco 1,d\ccf$,
    \begin{align*}
        \lambda_i(x,v) &= \frac{2}{1-\rho}\E\co\Psi\po\frac{\partial_i U_1(x))}{2}(v_i-V_i^i)\pf\cf \\ &\leqslant \frac{1}{1-\rho}(2\Psi(0)+||\Psi'||_\infty||\partial_i U_1||_\infty\E[|v_i-V_i^i|]) \\ &\leqslant
        \frac{1}{1-\rho}\po 2\Psi(0)+ ||\Psi'||_\infty||\partial_i U_1||_\infty\po(1-\rho)|v_i|+\sqrt{\frac{2(1-\rho^2)}{\pi}}\pf\pf=: \overline{\lambda_i}(|v_i|).
    \end{align*}
Note that since this bound depends on $|v_i|$, the value of $\overline{\lambda_i}$ changes after each jump, except if $\rho=-1$, where $v_i$ becomes $-v_i$ after a jump of type $i$. In this case, using the properties of exponential distributions, the simulation of $\mathcal{L}_J$ during a time $\delta$ using thinning is simpler: starting from $(x,v)$,
\begin{enumerate}
    \item Draw $M\sim \text{Poisson}(\delta\sum_{i}\overline{\lambda_i}(|v_i|))$, the total number of jump proposals.
    \item For each $m\in\cco 1,M\ccf$, draw $I_m\in\cco 1,d\ccf$ such that $\PP(I_m=i)=\overline{\lambda_i}/\sum_{j=1}^d \overline{\lambda_j}$ and with probability $\lambda_{I_m}/\overline{\lambda I_m}$, change the sign of $v_{I_m}$.
\end{enumerate}
If $\rho>-1$, when a jump of type $i$ occurs, $v_i$ becomes
\[ \rho v_i +\sqrt{1-\rho^2}\tilde{\xi}\,,\]
where $\tilde{\xi}$ is a random variable with density proportional to
\[ f_{(x,v)}(y) = \Psi\po\frac{\partial_i U_1(x)}{2}\po(1-\rho)v_i-\sqrt{1-\rho^2}y)\pf\pf e^{-y^2/2}\,. \]
It is shown in~\cite[Section 5.3]{MRZ}, and implemented in~\cite{GouraudNumeric} in the context of the velocity jump Langevin method, that this random variable can be simulated exactly with rejection sampling, by using either Gamma, exponential, shifted Rayleigh or Gaussian distributions as proposals, depending on the values of $(x,v)$. 
}

\new{

\paragraph{Complexity of the algorithm.}

Let us discuss the efficiency of the procedure. In a classical discretization scheme of the Langevin diffusion, all the $d$ derivatives of $U_1$ are computed at each time step, whereas in the BJAOAJB scheme, there are as many computations of terms $\partial_i U_1$ as jump proposals. As we saw, when $\rho=-1$, there are in average $\delta\sum_{i=1}^d\overline{\lambda_i}(|v_i|)$ such proposals per time step. At equilibrium, the velocity $v$ is distributed according to a standard Gaussian, which implies that on average,
\[\E_\mu[\overline{\lambda_i}] = \Psi(0)+\sqrt{\frac{2}{\pi}}||\Psi'||_\infty||\partial_i U_1||_\infty\,.\]
Therefore, for $\rho=-1$, we expect the simulation of the velocity jump Langevin using thinning to yield  a computational advantage if 
\begin{equation}\label{eq:advantage}
    \delta \po \Psi(0)+\sqrt{\frac{2}{\pi}}||\Psi'||_\infty||\na U_1||_\infty \pf < 1\,.
\end{equation}
When $\rho$ increases, the average number number of jumps increases, and goes to infinity when $\rho\rightarrow 1$, therefore computational advantage is only to be expected if $\rho$ is not too close to $1$. We provide some numerical values in the next paragraph  in the case of MD simulations.

\paragraph{Application to molecular dynamics.}

In MD simulations, the inequality~\eqref{eq:advantage} is easily satisfied since the commonly used time step is extremely small ($\delta\approx 10^{-15}$ seconds), a choice imposed by the fastest vibrational modes of molecular systems. In~\cite{GouraudNumeric,Weisman}, the part of the interaction potential to be treated by the jump process is constituted of the long-range electrostatic and van der Waals forces, namely
\[ U_1(x) =\sum_{1\leqslant i < j \leqslant N}\chi(r_{ij})\po \veps\co\po\frac{\sigma}{r_{ij}}\pf^{12}-\po\frac{\sigma}{r_{ij}}\pf^6\cf+\frac{q_iq_j}{r_{ij}}\pf\,,\]
where $N$ is the number of the atoms in the system, $r_{ij}$ is the distance between atoms $i$ and $j$, $q_i$ is the electric charge of the atom $i$, $\sigma$ and $\veps$ are parameters of the Lennard-Jones potential, and $\chi$ is a smooth switching function such that $\chi(r)=0$ if $r<r_c$, with $r_c$ a given cutoff radius.

In a box of water molecules with the classical SPC/Fw model, by choosing $r_c = 7$ angstroms as a cutoff, the bound on the derivatives of $U_1$ is of the order of $10^{-2}$ kcal.mol$^{-1}$.Å$^{-1}$, and if we consider $\Psi(s) = \log(\exp(s)+1)$, $\Psi(0) = \log(2)$, $\Psi'<1$, which implies that~\eqref{eq:advantage} is largely satisfied, with $\delta=10^{-3}$ps in standard MD units, the left hand side of~\eqref{eq:advantage} being of order $10^{-3}$. We refer to~\cite[Section 4.C]{Weisman} for further theoretical discussion on the speedup provided by the procedure in the Bouncy particle sampler case, when simulating boxes of water molecules. 

When $\rho>-1$, the average number of jumps increases. However, as shown extensively in~\cite{GouraudNumeric} on molecular dynamics simulations, a significant speedup is provided, as long as $\rho$ is not too close to $1$, for several reasons. First, thanks to the extremely small time step, the average number of jump proposals per time step remains significantly smaller than the total number of pairwise interactions. Second, the jump procedure is easily parallelizable over all atoms, which make it suitable for parallel architectures such as GPUs. In particular, jumps on pairwise interactions scale better than certain many-body interactions such as the Ewald sum appearing the reciprocal electrostatic force (classicaly used to compute long range electrostatic interactions, see for instance~\cite{spme}), and replacing a part of this sum by jumps on  pairwise direct electrostatic forces yields a better speedup. Finally, by combining the jump procedure with classical multi-time-step methods, the intrinsically random nature of jumps allows to mitigate some well-known resonance issues of multi-time-stepping, which allows to choose a slightly larger time step. 
In~\cite{GouraudNumeric}, the choice $\rho=0$ seems to provide a good compromise between speedup and preservation of dynamical properties, leading for instance to a total acceleration of $19.3\%$ in simulation speed compared to a state-of-art multi-time-step discretization of Langevin dynamics for a system of 96000 water molecules.}

\paragraph{Numerical illustrations.}

Let us give a numerical illustration of the algorithm with various parameters in the case of two-dimensional Gaussian distributions. We take $d=2$,
\[U(x,y) = \frac{x^2+ 11 y^2}{2}\,, \]
and consider the following splitting
\[U_0(x,y) = \frac{x^2+y^2}{2}\,,\quad U_1(x,y) = 5 y^2\,,\]
which means that the Langevin dynamics $\mathcal{L}_H+\mathcal{L}_D$ samples the standard Gaussian distribution, while the jumps $\mathcal{L}_J$ bring asymmetry to the process.
We fix a time step ($\delta = 0.01$) and a time of simulation ($n=10^4$ time steps) and show in Figure~\ref{fig:1} how the jump parameter $\rho$ and the friction $\gamma$ influence the trajectories. When the friction is very small (in Figure~\ref{fig:1} (a) and (b)), the Langevin part resembles the Hamiltonian dynamics associated to $U_0$. In that case, if $\rho$ is close to $1$, the process is close to the Hamiltonian dynamics sampling $U$. If $\rho=-1$, the jump process corresponds to the BPS, there are fewer jumps, but there are more visible since they tend to change the direction of the trajectory to bring asymmetry. When the friction is high (in Figure~\ref{fig:1} (e) and (f)), as expected, both trajectories (with either $\rho=1$ or $\rho$ close to $1)$ resemble an overdamped Langevin diffusion process. The intermediate case (in Figure~\ref{fig:1} (c) and (d)) correspond to a mix between those extreme situations, with noisy but ballistic dynamics.

For a thorough empirical study of the efficiency of the algorithm on MD applications, we refer to the previous~\cite{Weisman} (restricted to the case $\rho=-1$) or the companion paper~\cite{GouraudNumeric} in the general case.
\begin{figure*}[htbp]
\begin{center}
    \subfloat[$\gamma = 10^{-4}$, $\rho=-1$]{\includegraphics[width=0.5\textwidth]{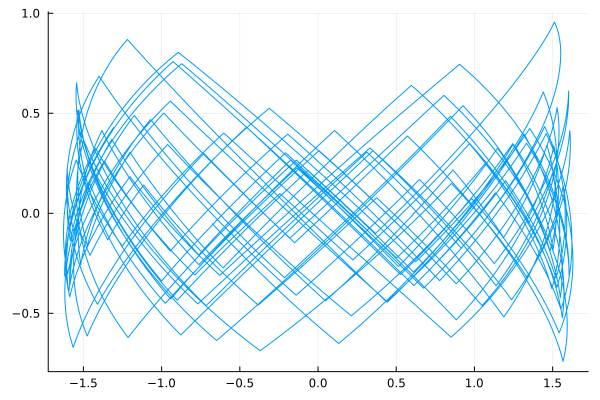}}
    \subfloat[$\gamma = 10^{-4}$, $\rho=1-10^{-4}$]{\includegraphics[width=0.5\textwidth]{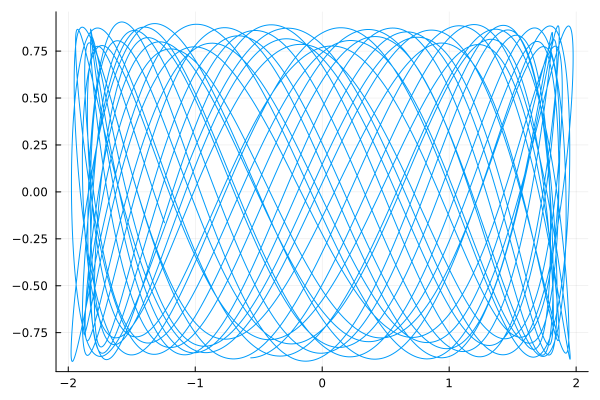}} \\
    \subfloat[$\gamma = 0.1$, $\rho = -1$]{\includegraphics[width=0.5\textwidth]{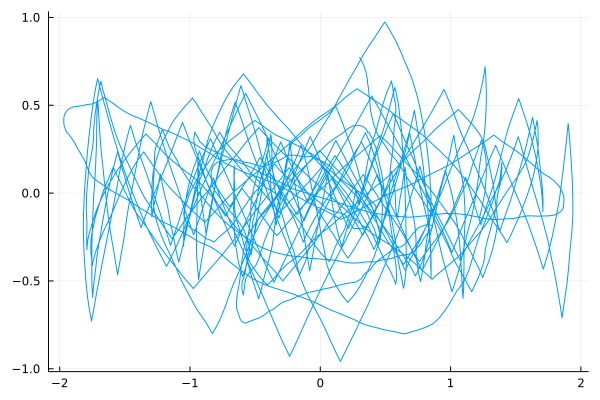}}
    \subfloat[$\gamma = 0.1$, $\rho=1-10^{-4}$]{\includegraphics[width=0.5\textwidth]{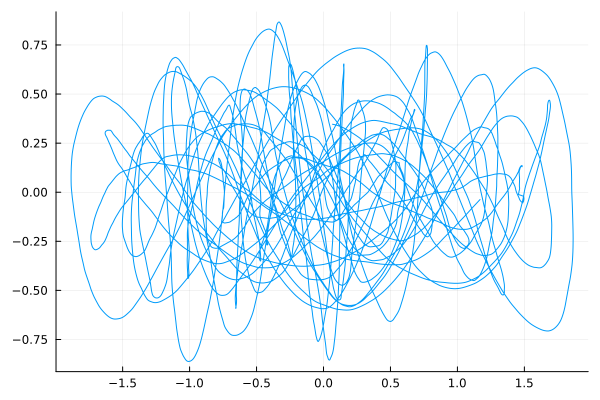}} \\
    \subfloat[$\gamma = 100$, $\rho = -1$]{\includegraphics[width=0.5\textwidth]{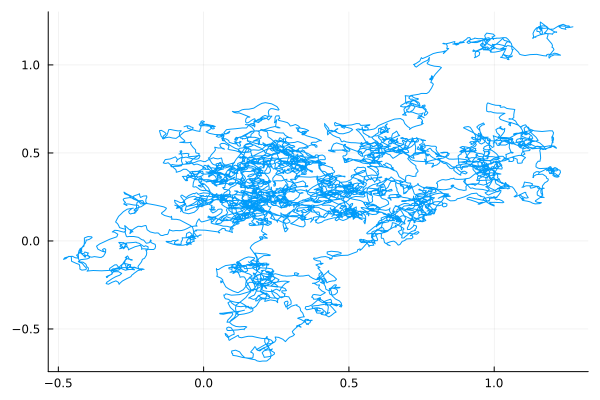}}
    \subfloat[$\gamma = 100$, $\rho=1-10^{-4}$]{\includegraphics[width=0.5\textwidth]{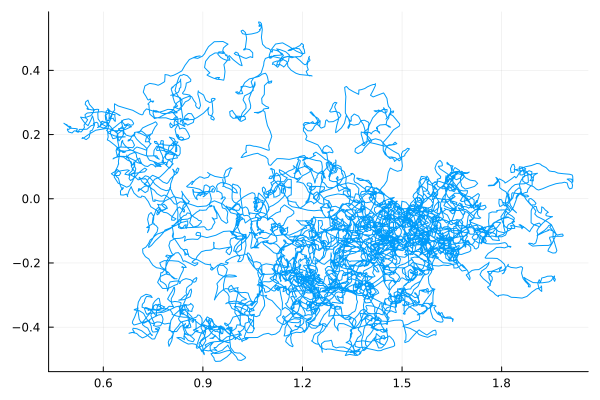}}
\end{center}
\caption{Trajectories of a splitting scheme for the velocity jump Langevin process}\label{fig:1}
\end{figure*}

\section{Geometric ergodicity}
\label{sec:ergodicity}

This section is devoted to the proof of Theorem~\ref{thm:estimates}, which relies on hypocoercive computations \textit{à la Villani} in Sobolev spaces. We only consider the case $\mathcal X=\R^d$, the proof being  easily adapted and simpler in the compact periodic case.

\subsection{Proof of Theorem~\ref{thm:estimates}}\label{s-sec:proof-thm-1}

In this section, we give the key steps of the proof, postponing the proofs of the main intermediary results to the rest of Section~\ref{sec:ergodicity}.

\emph{Step 1. Sobolev hypocoercivity.} Fix $k\in\N$ and define the Sobolev space of order $k$ by
\[
H^k(\mu) = \left\{f:\R^{2d}\to\R\, \text{ measurable}, \|f\|_{H^k}^2:=\sum_{i + j \leqslant k}\int_{\R^{2d}} |\na_x^i\na_v^j f|^2 \dd \mu <\infty \right\},
\]
where we denoted
\[
|\nabla_x^i\nabla_v^j f|^2 = \sum_{|\alpha_1|=i;|\alpha_2|=j} |\partial_x^{\alpha_1}\partial_v^{\alpha_2}f|^2\,.
\]
We will first establish the exponential convergence of $P_t$ to $\mu$ in $H^k$. The velocity jump Langevin process is hypocoercive, in the sense that the generator of this process does not satisfy
\[
\left< f,\mathcal Lf \right>_{H^k} \leqslant -\rho \|f\|_{H^k}^2,\, \forall f\in H^k,
\]
for any $\rho>0$, but such that there exists a norm equivalent to $\|\cdot\|_{H^k}$ such that the previous inequality holds. To prove exponential convergence to equilibrium, following~\cite{Journel1,Journel2,Zhang}, for $f\in H^k(\mu)$, we introduce  
\begin{equation}\label{eq:modified-norm}
\new{N}_k(f) = \int_{\R^{2d}} f^2 \dd\mu + \sum_{p=1}^k \int_{\R^{2d}}\po\sum_{i=0}^{p-1}\omega_{i,p}|\nabla_x^i\nabla_{v}^{p-i} f|^2 + \omega_{p,p}|(\nabla_x^p-\nabla_x^{p-1}\nabla_v)f|^2\pf\dd\mu\,,
\end{equation}
where $(\omega_{i,p})_{p\in\N, 1\leqslant i \leqslant p}$ are positive numbers, which is  equivalent to $\|f\|_{H^k}^2$. We refer to $\new{N}_k$ as the modified Sobolev norm (although strictly speaking  it is the square of a norm). With computations similar to~\cite{villani} (although rather presented in terms of Gamma calculus as in \cite{gammacalculus}), and proceeding by induction in $k$, our goal is to design $(\omega_{i,p})_{p\in\N, 1\leqslant i \leqslant p}$ in such a way that, at least informally,
\[
\partial_t \new{N}_k\po P_tf - \mu(f) \pf \leqslant -\rho_k \new{N}_k\po P_tf - \mu(f)\pf\,,
\]
for all $f\in H^k$ for some $\rho_k>0$. To make rigorous this argument, instead of justifying the time differentiation, we will rely on Lumer-Philips theorem~\cite[Chapter IX, p.250]{Func_anal}, see in Section~\ref{s-sec:Hk-conv}. This leads to the following:

\begin{thm}\label{thm:conv-Hk}
There exist positive $(\omega_{i,p})_{p\in\N, 1\leqslant i \leqslant p}$ and $(\rho_k)_{k\in\N^*}$ such that for all $k\in\N^*$, $f\in H^k(\mu)$ and $t\geqslant 0$,
\[\new{N}_k(P_tf - \mu(f))\leqslant e^{-\rho_k t}\new{N}_k(f-\mu(f))\,.\]
\end{thm} 
This is proven below (from Section~\ref{sec:Gammacalculus} to Section~\ref{s-sec:Hk-conv}).

Thanks to Sobolev embeddings, this already gives a result similar to Theorem~\ref{thm:estimates} except that in the right-hand-side of \eqref{eq:thm_ergodicite}, $b=1$. This would not be enough for establishing Theorem~\ref{thm:TalayTubaro} since $e^{H}$ is not integrable with respect to $\mu$, from which it cannot be a Lyapunov function for $P_t$,  and then neither for the BJAOAJB chain. Hence, we wouldn't be able to control the expectation of the discretization bias in the proof of Theorem~\ref{thm:TalayTubaro}.

\medskip

\emph{Step 2. Lyapunov function and re-weighted Sobolev norms.} To solve this issue, we will work with a particular weighted Sobolev norms. The weight $V_b:\R^{2d}\rightarrow\R$ is defined, for any $a,b>0$, by
\begin{equation}
    \label{eq:defV_b}
    V_b(x,v) = \exp\po b\po U_0(x)+\frac{|v^2|}{2}- a\frac{x\cdot v}{\sqrt{1+|x|^2}}\pf\pf.
\end{equation}

We show in Section~\ref{s-sec:proof-estimates} that it is a Lyapunov function for $\mathcal L^*$, the adjoint operator of $\mathcal L$ in $L^2(\mu)$ (which is computed in Proposition~\ref{prop:adjoint}):

\begin{prop}\label{prop:deflyapunov}
There exist $a>0,b_0\in(0,1)$ such that for all $b\in(0,b_0)$, there exist $\eta>0, C>0$ such that 
\[\mathcal{L}^*V_b\leqslant -\eta V_b + C\,.\]
\end{prop}

Combining Theorem~\ref{thm:conv-Hk} and Proposition~\ref{prop:deflyapunov}, following computations as in the proof of Theorem~\ref{thm:conv-Hk} with a modified norm similar to \eqref{eq:modified-norm} but with the additional weight $V_b$, we end up with the following:

\begin{prop}\label{prop:conv-lyapunov}
For all $\alpha\in\N^{2d}$ and all $f\in\mathcal A$, there exist $b\in (0,1),\eta>0,C>0$ such that for all $t\geqslant 0$, 
\[\int_{\R^{2d}}|\partial^{\alpha}(P_t f-\mu(f))|^2 V_b\dd\mu \leqslant C e^{-\eta t}\,.\]
\end{prop}
The proof is postponed to Section~\ref{s-sec:proof-estimates}.

\medskip

\emph{Step 3. Conclusion.} Using Step~1 and~2, we prove Theorem \ref{thm:estimates}:

\begin{proof}[Proof of Theorem \ref{thm:estimates}]
Let $f\in\mathcal{A}$ and $\alpha\in\N^{2d}$. Without loss of generality we assume that $\mu (f) = 0$. Let $b'\in(0,1)$ such that the result of Proposition~\ref{prop:conv-lyapunov} holds for $V_{b'}$, i.e.
\[\int_{\R^{2d}}|\partial^{\nu}P_t f|^2 V_{b'}\dd\mu \leqslant C e^{-\eta t}\,,\]
for all $\nu\leqslant\alpha$, and let $b\in(0,1)$ and $\veps>0$ such that $b'>1-(1-\veps)b$. We then have
\begin{align*}
    \int_{\R^{2d}}|\partial^{\alpha}(P_t f e^{-bH})|^2 &\leqslant C\sum_{\nu\leqslant \alpha}\int_{\R^{2d}}|\partial^\nu P_t f|^2|\partial^{\alpha-\nu}e^{-bH}|^2 \\
    &\leqslant C\sum_{\nu\leqslant \alpha}\int_{\R^{2d}}|\partial^\nu P_t f|^2 e^{-(1-\veps)bH} \\
    &= C\sum_{\nu\leqslant \alpha}\int_{\R^{2d}}|\partial^\nu P_t f|^2 e^{(1-(1-\veps)b)H}\dd\mu \\
    &\leqslant C\sum_{\nu\leqslant \alpha}\int_{\R^{2d}}|\partial^\nu P_t f|^2 V_{b'}\dd\mu \\
    &\leqslant Ce^{-\rho t}\,,
\end{align*}
for some $\rho>0$. We conclude with Sobolev embedding.

\end{proof}

\subsection{Preliminary: Gamma calculus}\label{sec:Gammacalculus}

The proof of Theorem~\ref{thm:conv-Hk} relies on an induction argument and Gamma calculus, as presented in~\cite{gammacalculus}, \new{which in our case is simply a convenient way to compute time derivatives of quantities of the form~\eqref{eq:modified-norm} along the semi-group}. Fix some smooth $\phi:\mathcal C^{\infty}_c(\R^{2d})\to \mathcal C^{\infty}_c(\R^{2d})$ and $L\in \left\{ \mathcal L_H,\mathcal L_J,\mathcal L_D \right\}$. Define the generalized \textit{carré du champ} for all $h\in \mathcal C^{\infty}_c(\R^{2d})$ by
\begin{equation}\label{def:Gamma}
    \Gamma_{L,\phi}(h) = \frac{1}{2}\left( L(\phi(h)) - D_h\phi(h)Lh\right),
\end{equation}
where $D_h\phi$ denotes the differential operator of $\phi$. When $\phi(h)=h^2$ we retrieve the usual carré du champ, which is simply denoted by $\Gamma_L$.  Formally, using that $\int L h \mathrm{d}\mu = 0$ for any suitable $h$ by invariance of $\mu$ (which will be proven in Section~\ref{s-sec:mesure-inv}), we have:
\begin{equation}\label{eq:norm-derivative}
\partial_t \new{N}_k(P_tf - \mu(f)) = -\int_{\R^{2d}} \Gamma(P_tf) \dd\mu - \sum_{p=1}^k \int_{\R^{2d}}\po\sum_{i=0}^{p-1}\omega_{i,p}\Gamma_{i,p}(P_tf) + \omega_{p,p}\Gamma_{p,p}(P_tf)\pf\dd\mu\,,
\end{equation}
where, for $i<p$,
\[
\Gamma_{i,p}(h) = \Gamma_{ L,|\nabla_x^i\nabla_{v}^{p-i} \cdot|^2}(h)\quad\text{and}\quad \Gamma_{p,p}(h) = \Gamma_{ L,|(\nabla_x^p-\nabla_x^{p-1}\nabla_v)\cdot|^2}(h)
\,.
\]
If we could show that for all $h\in\mathcal C^{\infty}_c(\R^{2d})$
\[
\int_{\R^{2d}} \Gamma(h) \dd\mu + \sum_{p=1}^k \int_{\R^{2d}}\po\sum_{i=0}^{p-1}\omega_{i,p}\Gamma_{i,p}(h) + \omega_{p,p}\Gamma_{p,p}(h)\pf\dd\mu \geqslant \rho_k \new{N}_k(h-\mu(h)) \,,
\]
for some $\rho_k>0$,  then Theorem~\ref{thm:conv-Hk} would follow, at least informally, in the sense that we haven't justified the time derivatives in~\eqref{eq:norm-derivative}. We will actually prove a slightly stronger result (see Proposition~\ref{recurrence} below) for the sake of the induction argument. Recall from~\cite{gammacalculus} the proposition that allows for the computations of this generalized \textit{carré du champ} (the proof is straightforward).

\begin{prop}\label{prop:calcul_Gamma}
If there exists $A=(A_1,\cdots,A_p):\mathcal C^{\infty}\to (\mathcal C^{\infty})^p$ a linear operator such that $\phi(h) = |Ah|^2$, then
\[
\Gamma_{L,\phi}(h) = \Gamma_{L}(Ah) + Ah\cdot [L,A]h,
\]
where $\Gamma_{L}(Ah)=\sum_{i=1}^p\Gamma_L(A_ih)$
and $[L,A]=\left([L,A_1],\cdots,[L,A_p] \right)$.
\end{prop}

In the remainder of this work, we will denote $\Gamma$ for $\Gamma_L$.
For all $i,j,i',j'$ such that $i+j = i'+j'$, the scalar product here has to be understood as
\[\nabla_x^i\nabla_v^{j}h\cdot \nabla_x^{i'}\nabla_v^{j'}h := \sum_{|\alpha_1|=i}\sum_{|\alpha_2|=j'}\sum_{|\nu|=i'-i}(\partial_x^{\alpha_1}\partial_v^{\alpha_2+\nu}h)(\partial_x^{\alpha_1+\nu}\partial_v^{\alpha_2}h)\,.\]

In order to justify the discussion above, we will start by proving that the Gibbs measure is indeed invariant for the velocity jump Langevin process in Section~\ref{s-sec:mesure-inv}. Section~\ref{s-sec:H1-conv} will treat the  case $k=1$ which is the initialization of the induction argument. Section~\ref{s-sec:Hk-conv} will present the proof of convergence in $H^k$ space, namely Theorem~\ref{thm:conv-Hk}. Finally, Propositions~\ref{prop:deflyapunov} and~\ref{prop:conv-lyapunov} will be proved in Section~\ref{s-sec:proof-estimates}.

\subsection{Invariance of \texorpdfstring{$\mu$}{mu}}\label{s-sec:mesure-inv}

The use of the velocity jump process for the sampling of $\mu$, as well as the use of hypocoercivity method rely on the fact that $\mu$ is indeed an invariant measure of velocity jump Langevin process. Let us prove this property by computing the adjoint of $\mathcal L$ in $L^2(\mu)$, denoted as $\mathcal L^*$.

\begin{prop}\label{prop:adjoint}
Under Assumption~\ref{assu}, for all $h\in \mathcal{C}_c^{\infty}(\R^{2d})$:
\[
\mathcal{L}^* h = -v\cdot\nabla_x h +\nabla U_0\cdot \nabla_vh-\gamma v\cdot \nabla_v h + \gamma\Delta_v h
    +\frac{2}{1-\rho}\sum_{i=1}^d\E\co(h(\cdot,V^i)-h)\Psi\po\frac{\partial_i U_1}{2}(V_i^i-v_i)\pf\cf\,.
\]
\end{prop}

\begin{proof}
Let $g,h \in \mathcal C^{\infty}_c(\R^{2d})$. By integrating by parts, we get 

\begin{align*}
    \int_{\R^{2d}} g\mathcal{L}_H h\dd\mu &= \mathcal{Z}_\mu^{-1}\int_{\R^{2d}} g(v\cdot\nabla_x h-\nabla U_0\cdot\nabla_v h)e^{-U(x)}e^{-|v|^2/2}\dd x \dd v \\
       &= -\mathcal{Z}_\mu^{-1}\int_{\R^{2d}} h v\cdot\nabla_x(ge^{-U(x)})e^{-|v|^2/2}\dd x\dd v +\mathcal{Z}_\mu^{-1}\int_{\R^{2d}} h\nabla_v\cdot(ge^{-|v|^2/2}\nabla U_0(x))e^{-U(x)}\dd x\dd v \\
       &= -\int_{\R^{2d}} hv\cdot\nabla_x g\dd\mu+\int_{\R^{2d}} gh v\cdot\nabla U(x)\dd\mu+\int_{\R^{2d}} h \nabla U_0(x)\cdot\nabla_v g\dd\mu-\int_{\R^{2d}} gh v\cdot \nabla U_0(x)\dd\mu \\
       &= \int_{\R^{2d}} h(-v\cdot \nabla_x g +\nabla U_0\cdot \nabla_v g)\dd\mu + \int_{\R^{2d}} gh v\cdot \nabla U_1\dd\mu\,.
\end{align*}
We also have
\begin{align*}
    \int_{\R^{2d}} g\mathcal{L}_D h\dd\mu &= \mathcal{Z}_\mu^{-1}\int_{\R^{2d}} g(-\gamma v\cdot \nabla_v h +\gamma\Delta_v h)e^{-U(x)}e^{-|v|^2/2}\dd x \dd v \\
    &= \gamma\mathcal{Z}_\mu^{-1}\po\int_{\R^{2d}} g\nabla_v(e^{-|v|^2/2})\cdot\nabla_v h e^{-U(x)}\dd x\dd v + \int_{\R^{2d}} h\Delta_v (ge^{-|v|^2/2})e^{-U(x)}\dd x\dd v\pf \\
    &= \gamma\mathcal{Z}_\mu^{-1}\po \int_{\R^{2d}}\nabla_v(g e^{-|v|^2/2})\cdot\nabla_v h e^{-U(x)}\dd x\dd v - \int_{\R^{2d}} \nabla_v g\cdot \nabla_v h e^{-U(x)}e^{-|v|^2/2}\dd x\dd y \pf \\ &\qquad + \gamma\mathcal{Z}_\mu^{-1}\int_{\R^{2d}} h\Delta_v (ge^{-|v|^2/2})e^{-U(x)}\dd x\dd v \\
    &= -\gamma\int_{\R^{2d}}\nabla_v g\cdot \nabla_v h \dd\mu = \int_{\R^{2d}} h\mathcal{L}_D g\dd\mu\,.
\end{align*}

For the jump part, denote $k_0(x,v,\cdot)$ the law of the random variable $V^i$. Then $k_0$ is reversible with respect to the standard Gaussian measure, in the sense that for all $x\in\R^d$,
\begin{equation}\label{revers}
    k_0(x,v,\dd v')e^{-|v|^2/2}\dd v = k_0(x,v',\dd v)e^{-|v'|^2/2}\dd v'.
\end{equation}
Indeed, since $V_j^i= v_j$ for $j\neq i$ and
\[V_i^i = \rho v_i + \sqrt{1-\rho^2}\new{\xi}\,,\]
the coordinates of $V^i$ are independent and each transition $v_j\mapsto V_j^i$ is reversible for the standard one-dimensional Gaussian. 
Using this reversibility, and the fact that $\Psi(s)-\Psi(-s) = s$, we get that
\begin{align*}
    \int_{\R^{2d}} g\mathcal{L}_J h\dd\mu &= \frac{2}{1-\rho}\sum_{i=1}^d\int_{\R^{2d}} g\E\co(h(x,V^i)-h(x,v))\Psi\po\frac{\partial_i U_1(x)}{2}(v_i-V^i_i)\pf\cf\dd\mu \\
    &= \frac{2}{1-\rho}\sum_{i=1}^d\int_{\R^{2d}} h\E\co (g(x,V^i)-g)\Psi\po-\frac{\partial_i U_1(x)}{2}(v_i-V^i_i)\pf\cf \\ &\qquad +\frac{2}{1-\rho}\sum_{i=1}^d\int_{\R^{2d}} gh\E\co\Psi\po-\frac{\partial_i U_1(x)}{2}(v_i-V^i_i)\pf-\Psi\po\frac{\partial_i U_1(x)}{2}(v_i-V^i_i)\pf\cf\dd\mu \\ 
    &= \frac{2}{1-\rho}\sum_{i=1}^d\int_{\R^{2d}} h\E\co (g(x,V^i)-g)\Psi\po-\frac{\partial_i U_1}{2}(v_i-V^i_i)\pf\cf \\ &\qquad - \frac{1}{1-\rho}\sum_{i=1}^d\int_{\R^{2d}} gh\partial_i U_1\E\co(v_i-V^i_i)\cf\dd\mu \\
    &= \frac{2}{1-\rho}\sum_{i=1}^d\int_{\R^{2d}} h\E\co (g(x,V^i)-g)\Psi\po-\frac{\partial_i U_1}{2}(v_i-V^i_i)\pf\cf -\int_{\R^{2d}} gh v\cdot\nabla U_1\dd\mu.
\end{align*}
Therefore, by summing the three terms, we get the result.
\end{proof}

\begin{corollary}
Under Assumption~\ref{assu}, for all $h\in\mathcal C_c^{\infty}(\R^{2d})$,
\[
\int_{\R^{2d}} \mathcal Lh \mu = 0.
\]
\end{corollary}

\begin{proof}
This is a direct consequence of the fact that $\mathcal L^*1 = 0$
 (although the constant function $1$ is not compactly supported, the computations of Proposition~\ref{prop:adjoint} clearly works when $g=1$).
\end{proof}

\subsection{Exponential decay in \texorpdfstring{$H^1$}{H1}}\label{s-sec:H1-conv}

In this section, we prove convergence of the process in $H^1(\mu)$\new{, which is the initialization step of the induction argument to prove Theorem~\ref{thm:conv-Hk}}. The modified Sobolev norm here reads
\[
\new{N}_1(h) = \int_{\R^{2d}} h^2 \dd\mu + \int_{\R^{2d}} \omega_{0,1}|\nabla_v h|^2 + \omega_{1,1}|(\na_x - \na_v) h|^2 \dd\mu,
\]
and Equation~\eqref{eq:norm-derivative} can be written:
\[
\partial_t \new{N}_1(P_tf - \mu(f)) = -\int_{\R^{2d}} \Gamma(P_tf) \dd\mu -  \int_{\R^{2d}}\po\omega_{0,1}\Gamma_{0,1}(P_tf) + \omega_{1,1}\Gamma_{1,1}(P_tf)\pf\dd\mu\,.
\]
\new{The goal of this section is to prove the following which, in view of~\eqref{eq:norm-derivative}, implies an exponential decay in $H^1$   as it bounds  the norm $N_1$ in terms of its dissipation along the semi-group.}

\begin{lem}\label{lem:initialisation}
Under Assumption~\ref{assu}, there exist $\omega_{0,1},\omega_{1,1},\rho_1$ such that for all $h\in\mathcal C_c^{\infty}(\R^{2d})$ satisfying $\mu(h)=0$:
\[
\int_{\R^{2d}} \Gamma(h) \dd\mu + \int_{\R^{2d}}\po\omega_{0,1}\Gamma_{0,1}(h) + \omega_{1,1}\Gamma_{1,1}(h)\pf\dd\mu \geqslant \rho_1 \po \new{N}_1(h) + \int_{\R^{2d}}|\nabla_x\nabla_v h|^2+|\nabla_v^2 h|^2\dd\mu \pf\,.
\]
\end{lem}
The terms of order $2$ in the right hand side are present to compensate future terms in the induction, but are not required for the the convergence in $H^1(\mu)$. This convergence relies on the so-called Poincare inequality.

\begin{prop}\label{prop:poincare-inquality}
Under Assumption~\ref{assu}, the Gibbs measure~\eqref{eq:Gibbs-measure} satisfies a Poincaré inequality: there exists $C_P>0$ such that 
\begin{equation}\label{eq:poincare}
\int_{\R^{2d}} \po h - \mu(h)\pf^2 \dd \mu \leqslant C_P \int_{\R^{2d}}|\nabla h|^2 \dd \mu\,. 
\end{equation}
for all $h\in\mathcal C^{\infty}_c(\R^{2d})$.
\end{prop}

\begin{proof}
Proof of such inequality can be found in~\cite{BakryGentilLedoux}.
\end{proof}

In order to apply Proposition~\ref{prop:calcul_Gamma}, we need to express the commutators. For the Hamiltonian part we get:
\[
[\mathcal L_H,\na_x] = \na^2U_0\na_v,\qquad [\mathcal L_H,\na_v] = - \na_x. 
\]
For the diffusion part we get:
\[
[\mathcal L_D,\na_x] = 0,\qquad [\mathcal L_D,\na_v] = \gamma \na_v. 
\]
Regarding the jump part, we have the following.
\begin{lem}\label{lem:commutator-jump}
Under Assumption~\ref{assu}, there exists $C>0$ such that for all $h\in \mathcal C_c^{\infty}(\R^{2d})$:
\begin{equation*}
\int_{\R^{2d}}|[\nabla_x,\mathcal{L}_J]h|^2\dd\mu + \int_{\R^{2d}}|[\nabla_v,\mathcal{L}_J]h|^2\dd\mu \leqslant C \po \int_{\R^{2d}} |\nabla_{v} h|^2\dd\mu +  \int_{\R^{2d}}|\nabla_v^2 h|^2\dd\mu \pf\,.
\end{equation*}
\end{lem}

To prove Lemma~\ref{lem:commutator-jump}, we need two intermediate results:

\begin{lem}\label{villanilemma} For any function $h\in \mathcal C_c^{\infty}(\R^{2d})$ and any $x\in\R^d$,
\[ \int_{\R^d} |v|^2 h^2(x,v)e^{-|v|^2/2}\dd v \leqslant 2d\int_{\R^d} h^2(x,v)e^{-|v|^2/2}\dd v + 4\int_{\R^d}|\nabla_v h(x,v)|^2 e^{-|v|^2/2}\dd v\,.\]
\end{lem}

\begin{proof}
This is a particular case of the Lemma A.24 of \cite{villani}. Notice that $v e^{-|v|^2/2} = -\nabla e^{-|v|^2/2}$. Hence, an integration by parts and Young's inequality yield, for any $x$,
\begin{equation*}
    \begin{split}
        \int_{\R^{d}} |v|^2 h^2(x,v) e^{-|v|^2/2} \dd v &= - \int_{\R^{d}} h^2(x,v) v\cdot \nabla e^{-|v|^2/2}\dd v \\
        &= \int_{\R^{d}} \nabla\cdot(h^2(x,v)v)e^{-|v|^2/2}\dd v \\
        &= d \int_{\R^d} h^2(x,v) e^{-|v|^2/2}\dd v + \int_{\R^{d}} 2 h(x,v)\nabla h(x,v) \cdot v e^{-|v|^2/2} \dd v \\
        &\leqslant d \int_{\R^d} h^2(x,v) e^{-|v|^2/2}\dd v + \int_{\R^{d}} 2|\nabla h(x,v)|^2 e^{-|v|^2/2} \dd v + \frac{1}{2}\int_{\R^d} |v|^2 h^2(x,v)e^{-|v|^2/2}\dd v\,,
    \end{split}
\end{equation*}
and thus the result.
\end{proof}

\begin{prop}[Poincaré for Gaussians] For any function $h\in \mathcal C_c^{\infty}(\R^{d})$,
\begin{equation}\label{poincaregauss}
    \int_{\R^d} \po h(x) - \frac{1}{(2\pi)^{d/2}}\int_{\R^d}h(y) e^{-|y|^2/2}\dd y\pf^2 e^{-|x|^2/2}\dd x \leqslant \int_{\R^{d}}|\nabla h(x)|^2 e^{-|x|^2/2}\dd x\,.
\end{equation}
\end{prop}

\begin{proof}
Proof of such inequality can be found in~\cite{BakryGentilLedoux}.
\end{proof}

\begin{proof}[Proof of Lemma~\ref{lem:commutator-jump}]
We have
\[
        [\nabla_x,\mathcal{L}_J]h = \sum_{i=1}^d [\nabla_x,\mathcal{L}_J^i]h = \frac{1}{1-\rho}\sum_{i=1}^d\E\co( v_i-V_i^i)(h(x,V^i)-h)\Psi'\po\frac{\partial U_1}{2}(v_i-V^i_i)\pf\cf\nabla\partial_i U_1\,,
\]
and
        \[[\nabla_v,\mathcal{L}_J]h = \sum_{i=1}^d\E\co\partial_i U_1(h(x,V^i)-h)\Psi'\po\frac{\partial U_1}{2}(v_i-V^i_i)\pf-2\partial_{v_i}h(x,V^i)\Psi\po\frac{\partial U_1}{2}(v_i-V^i_i)\pf\cf e_i\,.\]
Using the fact that $\Psi'$, $\nabla U_1$ and $\nabla^2 U_1$ are bounded, there exists a constant $C>0$ such that
\[|[\nabla_x,\mathcal{L}_J]h|^2\leqslant C\sum_{i=1}^d\E[(v_i-V_i^i)^2(h(x,V^i)-h)^2]\,,\]
and
\[|[\nabla_v,\mathcal{L}_J]h|^2\leqslant C\sum_{i=1}^d\E[(h(x,V^i)-h)^2+(\partial_{v_i}h(x,V^i))^2(1+(v_i-V^i_i)^2)]\,. \]
Recall from~\eqref{revers} that the law of the random variable $V^i$ is reversible with respect to the standard Gaussian measure. Therefore,
\begin{align*}
    \int_{\R^{2d}} \E\co (1+(v_i-V_i^i)^2) \po\partial_{v_i} h(x,V^i)\pf^2\cf\dd\mu &= \int_{\R^{2d}}(\partial_{v_i} h)^2\po 1+\E\co(v_i-V_i^i)^2\cf\pf\dd\mu \\
    &\leqslant \int_{\R^{2d}} (\partial_{v_i} h)^2\po 2+ v_i^2\pf\dd\mu \\
    &\leqslant 4\int_{\R^{2d}} \po(\partial_{v_i} h)^2 + (\partial_{v_i}^2 h)^2\pf \dd\mu\,,
\end{align*}
where we used Lemma~\ref{villanilemma} with $d=1$ in the last inequality.
By denoting $\Pi_{v_i}$ the projection 
\[
\Pi_{v_i} h(x,v) = \frac{1}{\sqrt{2\pi}}\int_{\R} h(x,v_1,\dots ,v_{i-1},w_i,v_{i+1},\dots,v_d)e^{-{w_i}^2/2}\dd w_i\,,
\]
we also have
\begin{align*}
    \int_{\R^{2d}} \E\co (1+(v_i-V^i_i)^2)(h(x,V^i)-h)^2 \cf \dd\mu &\leqslant 2\int_{\R^{2d}}\E\co (1+(v_i-V^i_i)^2)((h(x,V^i)-\Pi_{v_i} h)^2+(h - \Pi_{v_i} h)^2)\cf \dd\mu \\
    &\leqslant 4\int_{\R^{2d}}\po 1+\E\co(v_i-V^i_i)^2\cf\pf(h - \Pi_{v_i} h)^2 \dd\mu \\
    &\leqslant 4\int_{\R^{2d}}(2+v_i^2)(h - \Pi_{v_i} h)^2 \dd\mu \\
    &\leqslant 16\int_{\R^{2d}}(h - \Pi_{v_i} h)^2\dd\mu + 16\int_{\R^{2d}}(\partial_{v_i} h)^2\dd\mu\,,
\end{align*}
where we used again Lemma~\ref{villanilemma} with $d=1$ in the last inequality.
The Poincaré inequality~\eqref{poincaregauss} can be written
\[
\int_{\R^{2d}}(h - \Pi_{v_i} h)^2\dd\mu \leqslant \int_{\R^{2d}}|\partial_{v_i}h|^2\dd\mu\,,
\]
and thus
\[ \int_{\R^{2d}} \E\co (1+(v_i-V^i_i)^2)(h(x,V^i)-h)^2 \cf \dd\mu \leqslant 32\int_{\R^{2d}}(\partial_{v_i} h)^2\dd\mu\,.\]
Therefore, there exists $C>0$ such that
\[\int_{\R^{2d}}|[\nabla,\mathcal{L}_J]h|^2\dd\mu \leqslant C\sum_{i=1}^d\int_{\R^{2d}} ((\partial_{v_i} h)^2+|\nabla_v \partial_{v_i} h|^2)\dd\mu =  C\po\int_{\R^{2d}}|\nabla_v h|^2\dd\mu + \int_{\R^{2d}}|\nabla_v^2 h|^2\dd\mu\pf\, ,\]
which concludes the proof.
\end{proof}

We now have everything to prove Lemma~\ref{lem:initialisation}.

\begin{proof}[Proof of Lemma~\ref{lem:initialisation}]
Fix $\omega_{0,1}>\omega_{1,1}>0$ that we are going to chose later. Let's first treat the derivative of the $L^2$ norm. For all $h\in \mathcal C^{\infty}_c(\R^{2d})$, $(x,v)\in \R^{2d}$, we have
\begin{equation}\label{eq:Gamma}
    \Gamma(h)(x,v) = 2\gamma|\nabla_v h(x,v)|^2 + \sum_{i=1}^d\lambda_i(x,v)\int_{\R^d} (h(x,v')-h(x,v))^2 q_i(x,v,\dd v') \,.
\end{equation}
This already implies that
\[ \Gamma(h) \geqslant  2\gamma|\nabla_v h|^2\,. \] 
To get a lower bound on $\Gamma_{0,1}(h)$ and $\Gamma_{1,1}(h)$, we use Proposition~\ref{prop:calcul_Gamma}. Equation~\eqref{eq:Gamma} yields: 
\[
\Gamma(\na_v h ) \geqslant 2\gamma |\na_v^2 h|^2,\qquad  \Gamma((\na_x- \na_v)h) \geqslant 2\gamma|(\na_x\na_v - \na_v^2)h|^2,
\]
so that using $\omega_{0,1}>\omega_{1,1}$ and Young inequality: 
\[
\omega_{0,1}\Gamma(\na_v h ) + \omega_{1,1}\Gamma((\na_x- \na_v)h) \geqslant 2\gamma(\omega_{0,1}-\omega_{1,1})|\na_v^2 h|^2 + \omega_{1,1}\po\gamma |\na_v^2 h|^2 + 2\gamma /3|\na_x\na_v h|^2\pf.
\]
Let us now look at the commutators. Fix $\varepsilon>0$ and write using Young inequality and Lemma~\ref{lem:commutator-jump}:
\begin{align*}
    \int_{\R^{2d}} \na_vh\cdot [\mathcal L,\na_v]h \dd\mu&= \int_{\R^{2d}} \na_vh\cdot \po - \na_x h + \gamma \na_v + [\mathcal L_J,\na_v]h \pf \dd\mu \\ &\geqslant -\int_{\R^{2d}} \po \frac{1}{\varepsilon} + 1 \pf |\na_vh|^2 + \frac{\varepsilon}{2} |(\na_x-\na_v)h|^2 + \frac{\varepsilon}{2}C \po  |\nabla_{v} h|^2\dd\mu +  |\nabla_v^2 h|^2 \pf \dd\mu\,,
\end{align*}
for some $C>0$. Using again Young inequality and Lemma~\ref{lem:commutator-jump}, we get that there exists $C>0$ such that:
\begin{align*}
    \int_{\R^{2d}} (\na_x-\na_v)h\cdot [\mathcal L,\na_x-\na_v]h \dd\mu&= \int_{\R^{2d}} (\na_x-\na_v)h \cdot \po \na^2U_0\na_vh + \na_xh - \gamma \na_v h + [\mathcal L_J,\na_x-\na_v]h \pf \dd\mu \\ &\geqslant \int_{\R^{2d}} \frac{1}{2} |(\na_x - \na_v)h|^2 - C\po |\na_vh|^2 +   |\nabla_v^2 h|^2 \pf \dd\mu.
\end{align*}
Now all is left to do is to choose $\varepsilon$, $\omega_{0,1}$ and $\omega_{1,1}$. We choose them such that
\[
\frac{\omega_{1,1}}{\omega_{0,1}} \leqslant \frac{2\gamma-C\veps}{2\gamma+C},\quad 
\veps<\min\po\frac{\omega_{1,1}}{2\omega_{0,1}},\frac{2\gamma}{C}\pf,\quad C\omega_{1,1}+\omega_{0,1}\po\frac{1}{\veps} + 1 +\frac{C\veps}{2}\pf \leqslant \gamma
\]
so that 
\begin{multline*}
\int_{\R^{2d}} \Gamma(h) \dd\mu + \int_{\R^{2d}}\po\omega_{0,1}\Gamma_{0,1}(h) + \omega_{1,1}\Gamma_{1,1}(h)\pf\dd\mu  \\ \geqslant \int_{\R^{2d}} \gamma |\na_vh|^2 + \frac{\omega_{1,1}}{4}|(\na_x - \na_y)h|^2 + \omega_{1,1}\gamma \po  |\na_v^2 h|^2 + \frac{2}{3}|\na_x\na_v h|^2  \pf  \dd\mu \\ \geqslant \tilde \rho_1 \int_{\R^{2d}} \omega_{0,1}|\nabla_v h|^2 + \omega_{1,1}|(\na_x - \na_v) h|^2 + |\nabla_x\nabla_v h|^2+|\nabla_v^2 h|^2 \dd \mu  \,,
\end{multline*}
for some $\tilde \rho_1>0$. Since we consider $h$ such that $\mu(h) = 0$, the Poincaré inequality~\eqref{eq:poincare} yields that there exists $\rho_1>0$ such that 
\[
\tilde \rho_1 \int_{\R^{2d}} \omega_{0,1}|\nabla_v h|^2 + \omega_{1,1}|(\na_x - \na_v) h|^2 + |\nabla_x\nabla_v h|^2+|\nabla_v^2 h|^2 \dd \mu \geqslant \rho_1 \po \new{N}_1(h) + \int_{\R^{2d}}|\nabla_x\nabla_v h|^2+|\nabla_v^2 h|^2\dd\mu \pf  \,,
\]
which concludes the proof.
\end{proof}

\subsection{Exponential decay in \texorpdfstring{$H^k$}{Hk} }\label{s-sec:Hk-conv}

Recall that we defined the modified Sobolev norm as
\[\new{N}_k(h) := \int_{\R^{2d}} h^2 \dd\mu + \sum_{p=1}^k \int_{\R^{2d}}\po\sum_{i=0}^{p-1}\omega_{i,p}|\nabla_x^i\nabla_{v}^{p-i} h|^2 + \omega_{p,p}|(\nabla_x^p-\nabla_x^{p-1}\nabla_v)h|^2\pf\dd\mu\,.\]

In order to control its derivative~\eqref{eq:norm-derivative}, we prove by induction on $k$ the following lemma, \new{generalizing Lemma~\ref{lem:initialisation}}.

\begin{prop}\label{recurrence}
Under Assumption~\ref{assu}, there exist $(\omega_{i,p})_{1\leqslant i\leqslant p}$ such that for all $k\in\N$, there exists $\rho_k>0$ such that for all $h\in\mathcal{C}_c^{\infty}(\R^{2d},\R)$ with $\mu(h)=0$:
    \begin{equation}\label{rec}
        \int_{\R^{2d}} \Gamma(h) \dd\mu + \sum_{p=1}^k \int_{\R^{2d}}\po\sum_{i=0}^{p-1}\omega_{i,p}\Gamma_{i,p}(h) + \omega_{p,p}\Gamma_{p,p}(h)\pf\dd\mu \geqslant \rho_k\po \new{N}_k(h) + \int_{\R^{2d}}\sum_{i=0}^k|\nabla_x^i\nabla_v^{k+1-i}h|^2\dd\mu \pf. 
    \end{equation}
\end{prop}

To prove this lemma, we will need formulas for partial derivatives of products, provided by the Leibniz formula 

\begin{prop}\label{prop:leibniz}
For all $g,h\in\mathcal C_c^{\infty}$ and $\alpha_1,\alpha_2\in\N^{d}$,
\[\partial_x^{\alpha_1}\partial_v^{\alpha_2}(gh) = \sum_{\nu_1\leq\alpha_1,\nu_2\leq\alpha_2}\binom{\alpha_1}{\nu_1}\binom{\alpha_2}{\nu_2}(\partial_x^{\nu_1}\partial_v^{\nu_2} g)(\partial_x^{\alpha_1-\nu_1}\partial_v^{\alpha_2-\nu_2} h)\]
where for $\alpha=(\alpha_1,\dots,\alpha_d)$ and $\nu = (\nu_1,\dots,\nu_d)$,
\[\binom{\alpha}{\nu} = \prod_{i=1}^d \binom{\alpha_i}{\nu_i}\,,\]
and
\[\nu\leq\alpha \iff \forall i\in\cco 1,d\ccf, \nu_i\leq\alpha_i.\]
\end{prop}

For any $\alpha\in\N^{d}$ we will also denote $|\alpha| = \sum_{i=1}^d\alpha_i$.
Let us look at the commutators. As in the $H^1$ case, we will look separately the parts associated to the Hamiltonian dynamic $\mathcal{L}_H$, the diffusion $\mathcal{L}_D$ and the jump process $\mathcal{L}_J$.
\begin{lem}[Hamiltonian commutators]\label{lem:commutator-hamiltonian}
Let $i\in\cco 0,k\ccf$. There exists $C>0$ such that the following holds. If $i<k$, we have
\begin{equation*}
        \nabla_x^i\nabla_v^{k+1-i}h\cdot[\mathcal{L}_H,\nabla_x^i\nabla_v^{k+1-i}]h \geqslant  -|\nabla_x^i\nabla_v^{k+1-i}h|^2 -C\po|\nabla_x^{i+1}\nabla_v^{k-i} h|^2 +\sum_{j=0}^{i-1}|\nabla_x^j\nabla_v^{k+2-i}h|^2\pf\,.
\end{equation*}
In the case $i=k$, for all $\eta>0$, we have
\begin{equation*}
      \nabla_x^{k}\nabla_vh\cdot[\mathcal{L}_H,\nabla_x^{k}\nabla_v]h \geqslant -\frac{\eta}{8}|(\nabla_x^{k+1}-\nabla_x^{k}\nabla_v)h|^2 - C\po 1+\frac{1}{\eta}\pf|\nabla_x^{k}\nabla_v h|^2 - C \sum_{j=0}^{k-1}|\nabla_x^j\nabla_v^2 h|^2\,.
\end{equation*}
Finally, we have
\begin{equation*}
        (\nabla_x^{k+1} - \nabla_x^{k}\nabla_v)h\cdot [\mathcal{L}_{H},\nabla_x^{k+1} - \nabla_x^{k}\nabla_v]h
        \geqslant \frac{1}{2}|(\nabla_x^{k+1}-\nabla_x^{k}\nabla_v)h|^2 -C\po\sum_{j=0}^{k}|\nabla_x^j\nabla_v h|^2  + \sum_{j=0}^{k-1}|\nabla_x^j\nabla_v^2 h|^2 \pf\,.
\end{equation*}
\end{lem}

\begin{proof}
For all $\alpha_1,\alpha_2 \in \N^d$ with $|\alpha_1|=i$ and $|\alpha_2|=k+1-i$, where $\alpha_2 = (\alpha_{2,1},\dots,\alpha_{2,d})$, write 
\[\alpha_2^j = (\alpha_{2,1},\dots,\alpha_{2,j-1},\alpha_{2,j}-1,\alpha_{2,j+1},\dots,\alpha_{2,d}),\]
and we have
\begin{equation*}
    \begin{split}
        \partial_x^{\alpha_1}\partial_v^{\alpha_2}\mathcal{L}_H h &= \sum_{j=1}^d \partial_x^{\alpha_1}\partial_v^{\alpha_2}\po v_j\partial_{x_j}h-\partial_j U_0\partial_{v_j}h\pf\\
        &= \sum_{j=1}^d \partial_v^{\alpha_2}(v_j\partial_x^{\alpha_1}\partial_{x_j} h) - \partial_x^{\alpha_1}(\partial_j U_0\partial_v^{\alpha_2}\partial_{v_j}h) \\
        &= \sum_{j=1}^d\po\sum_{\nu_2\leqslant\alpha_2}\binom{\alpha_2}{\nu_2}\partial_v^{\nu_2}(v_j)\partial_v^{\alpha_2-\nu_2}\partial_x^{\alpha_1}\partial_{x_j} h-\sum_{\nu_1\leqslant\alpha_1}\binom{\alpha_1}{\nu_1}\partial_x^{\nu_1}\partial_j U_0\partial_x^{\alpha_1-\nu_1}\partial_v^{\alpha_2}\partial_{v_j}h\pf \\
        &= \sum_{j=1}^d \po\alpha_{2,j}\partial_v^{\alpha_2^j}\partial_x^{\alpha_1}\partial_{x_j} h + v_j\partial_v^{\alpha_2}\partial_x^{\alpha_1}\partial_{x_j} h - \sum_{\nu_1\leqslant\alpha_1}\binom{\alpha_1}{\nu_1}\partial_x^{\nu_1}\partial_j U_0\partial_x^{\alpha_1-\nu_1}\partial_v^{\alpha_2}\partial_{v_j}h\pf\,.
    \end{split}
\end{equation*}
Therefore,
\[
[\mathcal{L}_H,\partial_x^{\alpha_1}\partial_v^{\alpha_2}] h = \sum_{j=1}^d\po \sum_{\nu_1\leqslant\alpha_1;\nu_1\neq 0}\binom{\alpha_1}{\nu_1}\partial_x^{\nu_1}\partial_j U_0\partial_x^{\alpha_1-\nu_1}\partial_v^{\alpha_2}\partial_{v_j}h - \alpha_{2,j}\partial_v^{\alpha_2^j}\partial_x^{\alpha_1}\partial_{x_j} h \pf\,.
\]
Using the bound on the derivatives of $U_0$, this yields that there exists a $C>0$ such that
\begin{equation*}
    \nabla_x^i\nabla_v^{k+1-i}h\cdot[\mathcal{L}_H\nabla_x^i\nabla_v^{k+1-i}]h \geqslant -  |\nabla_x^i\nabla_v^{k+1-i}h|^2 - C\po |\nabla_x^{i+1}\nabla_v^{k-i} h|^2 + \sum_{j=0}^{i-1}|\nabla_x^j\nabla_v^{k+2-i}h|^2 \pf.
\end{equation*}
If $i=k$, then for any $\eta>0$, using the Young inequality
\begin{equation*}
    \begin{split}
        \sum_{\underset{j\in\cco 1,d\ccf}{|\alpha_1|=k;|\alpha_2|=1}} (\partial_x^{\alpha_1}\partial_v^{\alpha_2}h)(\alpha_{2,j}\partial_v^{\alpha_2^j}\partial_x^{\alpha_1}\partial_{x_j} h) &= \sum_{|\alpha_1|=k;j\in\cco 1,d\ccf } (\partial_x^{\alpha_1}\partial_{v_j}h)(\partial_x^{\alpha_1}\partial_{x_j} h) \\
        &= \sum_{|\alpha_1|=k;j\in\cco 1,d\ccf } (\partial_x^{\alpha_1}\partial_{v_j}h)(\partial_x^{\alpha_1}\partial_{x_j} - \partial_x^{\alpha_1}\partial_{v_j}+\partial_x^{\alpha_1}\partial_{v_j}) h \\
        &\leqslant \po 1+\frac{2}{\eta}\pf|\nabla_x^{k}\nabla_v h|^2 + \frac{\eta}{8}|(\nabla_x^{k+1}-\nabla_x^{k}\nabla_v)h|^2\,.
    \end{split}
\end{equation*}
which again implies that 
\begin{equation*}
      \nabla_x^{k}\nabla_v h\cdot[\mathcal{L}_H, \nabla_x^{k}\nabla_v]h \geqslant  -C\po 1+\frac{1}{\eta}\pf|\nabla_x^{k}\nabla_v h|^2 - \frac{\eta}{8}|(\nabla_x^{k+1}-\nabla_x^{k}\nabla_v)h|^2 - C \sum_{j=0}^{k-1}|\nabla_x^j\nabla_v^2 h|^2\,.
\end{equation*}
Similarly, for $\beta$ a multi-index such that $|\beta|=k$ and $i\in\cco 1,d\ccf$,
\begin{multline*}
        [\mathcal{L}_H,\partial_x^\beta\partial_{x_i}-\partial_x^\beta\partial_{v_i}] h = \partial_x^{\beta}\partial_{x_i}h +\sum_{j=1}^d\sum_{\nu\leqslant\beta}\binom{\beta}{\nu}\partial^{\nu}\partial_{i,j} U_0\partial_x^{\beta-\nu}\partial_{v_j}h \\  +\sum_{j=1}^d\sum_{\nu\leqslant\beta;\nu\neq 0}\binom{\beta}{\nu}\partial^{\nu}\partial_j U_0\partial_x^{\beta-\nu}\partial_{x_i}\partial_{v_j}h  -\sum_{j=1}^d\sum_{\nu\leqslant\beta;\nu\neq 0}\binom{\beta}{\nu}(\partial^{\nu}\partial_j U_0)( \partial_x^{\beta-\nu}\partial_{v_i}\partial_{v_j}h)\,,
\end{multline*}
and thus, by applying Young inequality,
\begin{equation*}
    \begin{split}
        (\nabla_x^{k+1} - \nabla_x^{k}\nabla_v)h\cdot [\mathcal{L}_{H},\nabla_x^{k+1} - \nabla_x^{k}\nabla_v]h \geqslant  \frac{1}{2}|(\nabla_x^{k+1}-\nabla_x^{k}\nabla_v)h|^2 - C \po\sum_{j=0}^{k}|\nabla_x^j\nabla_v h|^2 + \sum_{j=0}^{k-1}|\nabla_x^j\nabla_v^2 h|^2 \pf\,,
    \end{split}
\end{equation*}
which concludes the proof.
\end{proof}

Regarding the diffusion part, we have
\begin{lem}\label{lem:commutators-diffusion}
For $i\in\cco 0,k\ccf$, 
\[ \nabla_x^i\nabla_v^{k+1-i}h\cdot[\mathcal{L}_D,\nabla_x^i\nabla_v^{k+1-i}]h = \gamma(k+1-i)|\nabla_x^i\nabla_v^{k+1-i}h|^2\,,\]
and there exists $C>0$ such that 
\begin{equation*}
      (\nabla_x^{k+1}-\nabla_x^{k}\nabla_v)h\cdot[\mathcal{L}_D, \nabla_x^{k+1}-\nabla_x^{k}\nabla_v]h
      \geqslant -\frac{1}{8}|(\nabla_x^{k+1}-\nabla_x^{k}\nabla_v)h|^2 -  C|\nabla_x^{k}\nabla_v h|^2 \, .
\end{equation*}
\end{lem}

\begin{proof}
We have, for all $\alpha_1,\alpha_2\in\N^d$,
\begin{equation*}
    \begin{split}
        \partial_x^{\alpha_1}\partial_v^{\alpha_2}\mathcal{L}_Dh(x,v) &= -\gamma\sum_{j=1}^d \partial_v^{\alpha_2} (v_j\partial_x^{\alpha_1}\partial_{v_j}h) + \gamma\sum_{j=1}^d\partial_x^{\alpha_1}\partial_v^{\alpha_2}(\partial_{v_j^2}^2h) \\
        &= -\gamma\sum_{j=1}^d v_j\partial_v^{\alpha_2}\partial_x^{\alpha_1}\partial_{v_j}h-\gamma\sum_{j=1}^d \alpha_{2,j}\partial_v^{\alpha_2^j}\partial_x^{\alpha_1}\partial_{v_j}h+ \gamma\sum_{j=1}^d\partial_x^{\alpha_1}\partial_v^{\alpha_2}(\partial_{v_j^2}^2h).
    \end{split}
\end{equation*}
Therefore, if $|\alpha_1|=i$ and $|\alpha_2| = k+1-i$, we have
    \[[\mathcal{L}_D,\partial_x^{\alpha_1}\partial_v^{\alpha_2}]h(x,v) = \gamma\sum_{j=1}^d \alpha_{2,j}\partial_v^{\alpha_2}\partial_x^{\alpha_1}h = \gamma (k+1-i)\partial_x^{\alpha_1}\partial_v^{\alpha_2}h\,,\]
and
\[ \nabla_x^i\nabla_v^{k+1-i}h\cdot[\mathcal{L}_D,\nabla_x^i\nabla_v^{k+1-i}]h = \gamma(k+1-i)|\nabla_x^i\nabla_v^{k+1-i}h|^2.\]
Similarly, for all $\beta\in\N^d$ and $i\in\cco 1,d\ccf$,
\begin{equation*}
    \begin{split}
        (\partial_x^\beta\partial_{x_i}-\partial_x^\beta\partial_{v_i})\mathcal{L}_D h(x,v) &= -\gamma\sum_{j=1}^d v_i\partial_x^{\beta}\partial_{x_i}\partial_{v_j}h + \gamma\sum_{j=1}^d\partial_{v_i}(v_j\partial_x^{\beta}\partial_{v_j}h)+\gamma\sum_{j=1}^d\partial_x^{\beta}\partial_{x_i}\partial_{v_j^2}^2h-\gamma\sum_{j=1}^d\partial_x^{\beta}\partial_{v_i}\partial_{v_j^2}^2h\\
        &= -\gamma\sum_{j=1}^d v_j\partial_x^{\beta}\partial_{x_i}\partial_{v_j}h + \gamma\sum_{j=1}^dv_j\partial_{v_i}\partial_x^{\beta}\partial_{v_j}h+\gamma\partial_x^{\beta}\partial_{v_i}h \\ &\qquad +\gamma\sum_{j=1}^d\partial_x^{\beta}\partial_{x_i}\partial_{v_j^2}^2h-\gamma\sum_{j=1}^d\partial_x^{\beta}\partial_{v_i}\partial_{v_j^2}^2h\,.
    \end{split}
\end{equation*}
Therefore, 
\[ [\mathcal{L}_D,\partial_x^\beta\partial_{x_i}-\partial_x^\beta\partial_{v_i}] h(x,v) = \gamma\partial_x^{\beta}\partial_{v_i}h\,,\]
and finally, by applying Young inequality, we get
\begin{multline*}
    \new{(\nabla_x^{k+1}-\nabla_x^{k}\nabla_v)}h\cdot[\mathcal{L}_D,\nabla_x^{k+1}-\nabla_x^{k}\nabla_v]h = \gamma\sum_{|\beta|=k;j\in\cco 1,d\ccf} (\partial_x^\beta\partial_{x_j}-\partial_x^\beta\partial_{v_j})h\partial_x^{\beta}\partial_{v_j}h  \\ \geqslant \new{-\frac{1}{8}}|(\nabla_x^{k+1}-\nabla_x^{k}\nabla_v)h|^2 -  C|\nabla_x^{k}\nabla_v h|^2\,,
\end{multline*}
for some $C>0$.
\end{proof}

Finally, let us look at the jump part.
\begin{lem}\label{lem:commutator-jump-Hk} 
For all $k\in \N$, there exists $C>0$ such that for $i\in\cco 0,k\ccf$ and $\veps>0$,
\begin{multline*}
    \int_{\R^{2d}}\nabla_x^i\nabla_v^{k+1-i}h \cdot [\mathcal{L}_J,\nabla_x^i\nabla_v^{k+1-i}]h\dd\mu
    \geqslant -\frac{1}{\veps}\int_{\R^{2d}}|\nabla_x^i\nabla_v^{k+1-i}h|^2\dd\mu \\ -  C\veps\int_{\R^{2d}}\po\sum_{\underset{l\leqslant k+1-j}{j\leqslant i}}|\nabla_x^j\nabla_v^l h|^2+|\nabla_x^i\nabla_v^{k+2-i} h|^2\pf\dd\mu\,.
\end{multline*}
and
\begin{multline*}
\int_{\R^{2d}}(\nabla_x^{k+1}-\nabla_k^{k}\nabla_v h)\cdot [\mathcal{L}_J,\nabla_x^{k+1}-\nabla_x^{k}\nabla_v]h\dd\mu \geqslant  -\frac{1}{8}\int_{\R^{2d}}|\new{\po\nabla_x^{k+1}-\nabla_x^{k}\nabla_v\pf} h|^2\dd\mu \\ - C\po\int_{\R^{2d}}\sum_{\underset{j\leqslant k+1-i}{i\leqslant k}}|\nabla_x^i\nabla_v^{j}h|^2\dd\mu + \int_{\R^{2d}}|\nabla_x^{k}\nabla_v^{2} h|^2\dd\mu\pf\,.
\end{multline*}
\end{lem}

To prove this Lemma, we need a generalization of Proposition~\ref{villanilemma}:

\begin{lem}\label{lem:generalized-vil}
Let $n\in\N^*$. There exists $C>0$ such that for any $h\in\mathcal{C}_c^{\infty}(\R)$,
\[\int_{\R}v^{2n}h^2(v)e^{-v^2/2}\dd v \leqslant C \sum_{k=0}^n \int_{\R}(h^{(k)}(v))^2e^{-v^2/2}\dd v\,.\]
\end{lem}

\begin{proof}
We prove the inequality by induction on $n$. The case $n=1$ corresponds to Lemma~\ref{villanilemma}.
Suppose that the inequality holds for some $n\in\N^*$. Then
\begin{align*}
        \int_{\R}v^{2n+2}h^2(v)e^{-v^2/2}\dd v &= \int_{\R}v^{2n}(vh(v))^2e^{-v^2/2}\dd v \\ &\leqslant C \sum_{k=0}^n \int_{\R}((vh(v))^{(k)})^2e^{-v^2/2}\dd v \\&
         =C \sum_{k=0}^n \int_{\R}(kh^{(k-1)}+vh^{(k)})^2e^{-v^2/2}\dd v 
        \\ &\leqslant C \sum_{k=0}^n 2\int_{\R}(k^2(h^{(k-1)})^2+v^2(h^{(k)})^2)e^{-v^2/2}\dd v \\ &\leqslant 2Cn^2 \sum_{k=0}^{n-1} \int_{\R}(h^{(k)}(v))^2e^{-v^2/2}\dd v+ 4C\sum_{k=0}^n \int_{\R}(h^{(k)}(v))^2e^{-v^2/2}\dd v \\ &\qquad + 8C\sum_{k=1}^{n+1} \int_{\R}(h^{(k)})^2e^{-v^2/2}\dd v \\ & \leqslant \widetilde C \sum_{k=0}^{n+1}\int_{\R}(h^{(k)})^2e^{-v^2/2}\dd v\,,
\end{align*}
and hence it holds for $n+1$.
\end{proof}

\begin{proof}[Proof of Lemma~\ref{lem:commutator-jump-Hk}]
For all $i,j\in\cco 1,d\ccf$,
\[ \partial_{v_j}(x,V^i) = ( 0, e_j+(\rho-1)e_i\1_{j=i}), \quad \partial_{x_j}(x,V^i) = (e_j,0),\quad \partial^{\nu}_x\partial^{\alpha}_x (x,V^i) = 0\quad \forall\, |\nu| + |\alpha|\geqslant 2\,.\]
As a consequence, for any $\nu_1,\nu_2\in\N^d$, we have
\[\partial_x^{\nu_1}\partial_v^{\nu_2}(h(x,V^i)) = \rho^{\nu_{2,i}}(\partial_x^{\nu_1}\partial_v^{\nu_2}h)(x,V^i)\,.\]
For all $i\in \cco 1,d \ccf$,
\[\partial_x^{\nu} \frac{\partial_i U_1}{2}(v_i-V_i^i) = \partial^\nu\partial_i U_1\frac{v_i-V_i^i}{2},\quad \partial_x^\nu\partial_{v_i}\frac{\partial_i U_1}{2}(v_i-V_i^i) = \frac{1-\rho}{2}\partial^\nu\partial_i U_1\,.\]
This yields that the only non-vanishing derivatives of $\Psi(\frac{\partial_i U_1}{2}(v_i-V_i^i))$ are of the type $\partial_x^\alpha\partial_{v_i}^k\Psi(\frac{\partial_i U_1}{2}(v_i-V_i^i))$. For $k\in\N$ and $\alpha\in \N^d$, we have:
\[
\partial_{v_i}^k\Psi\po\frac{\partial_i U_1}{2}(v_i-V_i^i)\pf = \Psi^{(k)}\po\frac{\partial_i U_1}{2}(v_i-V_i^i)\pf\po\frac{1-\rho}{2}\pf^k(\partial_i U_1)^k\,,
\] 
and 
\[
\partial_x^\alpha\partial_{v_i}^k\Psi\po\frac{\partial_i U_1}{2}(v_i-V_i^i)\pf = \po\frac{1-\rho}{2}\pf^k\sum_{\nu\leqslant\alpha}\binom{\alpha}{\nu}\partial_x^{\nu}\Psi^{(k)}\po\frac{\partial_i U_1}{2}(v_i-V_i^i)\pf\partial_x^{\alpha-\nu}(\partial_i U_1)^k. 
\]
Therefore, using that $\Psi(s) \leqslant c+|s|$ for some $c>0$, and that $\Psi'$ and $\partial_i U_1$ and their derivatives of all order are bounded, we get that for all $k\in \N$ and $\alpha\in \N^d$ there exists $C>0$ such that:

\[   \left|\partial_x^{\alpha}\partial_{v_i^k}\Psi\po\frac{\partial_i U_1}{2}(v_i-V_i^i)\pf\right| \leqslant  C\po 1+|v_i-V_i^i|^{|\alpha|+k}\pf \,. \]

Thus, for $\alpha\in\N^{2d}$, we have
\begin{align*}
        [\partial^{\alpha},\mathcal{L}_J^i]h &= \frac{2}{1-\rho}\po\E\co\partial^{\alpha}\co(h(x,V^i)-h)\Psi\po\frac{\partial_i U_1}{2}(v_i-V_i^i)\pf\cf\cf-\E\co(\partial^\alpha h(x,V^i) - \partial^\alpha h)\Psi\po\frac{\partial_i U_1}{2}(v_i-V_i^i)\pf\cf\pf \\ 
        &= \frac{2}{1-\rho} \sum_{\nu<\alpha}\binom{\alpha}{\nu}\E[(\partial^\nu(h(x,V^i))-\partial^\nu h)\partial^{\alpha-\nu}\po\Psi\po\frac{\partial_i U_1}{2}(v_i-V_i^i)\pf\pf] \\
        &\qquad +\frac{2}{1-\rho}\E\co(\partial^{\alpha}(h(x,V^i))-\partial^{\alpha}h(x,V^i))\Psi\po\frac{\partial_i U_1}{2}(v_i-V_i^i)\pf\cf \\
        &= \frac{2}{1-\rho} \sum_{\nu<\alpha}\binom{\alpha}{\nu}\E\co\po \rho^{\nu_{i+d}}\partial^\nu h(x,V^i)-\partial^\nu h\pf\partial^{\alpha-\nu}\po\Psi\po\frac{\partial_i U_1}{2}(v_i-V_i^i)\pf\pf\cf \\ &\qquad +\frac{2}{1-\rho}\E\co\po\rho^{\alpha_{i+d}}-1\pf\partial^{\alpha}h(x,V^i)\Psi\po\frac{\partial_i U_1}{2}(v_i-V_i^i)\pf\cf\,.
\end{align*}
As a consequence, there exists a $C>0$ such that
\[
    |[\partial^{\alpha},\mathcal{L}_J^i]h|^2 \leqslant C \po \E[(\partial^\alpha h(x,V^i))^2(1+|v_i-V^i_i|^2)]+\sum_{\nu<\alpha}\E\co(\partial^\nu h(x,V^i)-\partial^\nu h)^2(1+|v_i-V^i_i|^{2|\alpha-\nu|})\cf\pf\,.
\]
As in the $H^1$ case, the integrals with respect to $\mu$ of the previous expectations can be bounded using the reversibility of the law of $V^i$ with respect to the Gaussian measure, as well as Lemma \ref{lem:generalized-vil}.
For any $n\in\N^*$,
\[ \E[|v_i-V_i^i|^{2n}] = \E[|(1-\rho)v_i+\sqrt{1-\rho^2}\new{\xi}|^{2n}] \leqslant C(1+|v_i|^{2n})\,.\]
Therefore, using~\eqref{revers} and Lemma~\ref{lem:generalized-vil}, for all $\nu \in \N^{2d}$ and $n\in\N$:
\begin{multline*}
    \int_{\R^{2d}}\E[|v_i-V^i_i|^{2n}\partial^\nu h(x,V^i)|^2]\dd\mu = \int_{\R^{2d}} (\partial^\nu h)^2\E[|v_i-V^i_i|^{2n}] \\ \leqslant C\int_{\R^{2d}}\po 1 + v_i^{2n}\pf |\partial^\nu h|^2\dd\mu \leqslant C
    \sum_{k=0}^{n}\int_{\R^{2d}}|\partial_{v_i}^k\partial^\nu h|^2\dd\mu\,. 
\end{multline*}
Putting everything together, we get that for all $\alpha \in\N^{2d}$, there is $C>0$ such that
\begin{equation*}
\int_{\R^{2d}}|[\partial^{\alpha},\mathcal{L}_J^i]h|^2\dd\mu \leqslant  C\po \int_{\R^{2d}} (\partial_{v_i}\partial^\alpha h)^2\dd\mu + \sum_{\nu\leqslant\alpha}\sum_{k=0}^{|\alpha-\nu|}\int_{\R^{2d}}(\partial_{v_i}^k\partial^\nu h)^2 \dd \mu \pf\,.
\end{equation*}
For any $i\in\cco 0,k\ccf$ and $\veps>0$, Young inequality yields
    \begin{multline*}
    \int_{\R^{2d}}\nabla_x^i\nabla_v^{k+1-i}h \cdot [\mathcal{L}_J,\nabla_x^i\nabla_v^{k+1-i}]h\dd\mu
    \geqslant -\frac{1}{\veps}\int_{\R^{2d}}|\nabla_x^i\nabla_v^{k+1-i}h|^2\dd\mu \\ -  C\veps\int_{\R^{2d}}\po\sum_{\underset{l\leqslant k+1-j}{j\leqslant i}}|\nabla_x^j\nabla_v^l h|^2+|\nabla_x^i\nabla_v^{k+2-i} h|^2\pf\dd\mu\,.
\end{multline*}
To treat the term with derivative $\nabla_x^{k+1}-\nabla_k^{k}\nabla_v$, write for $\alpha\in\N^d$
\begin{equation*}
    \begin{split}
        [\partial_x^{\alpha},\mathcal{L}_J^i]h
        &= \frac{2}{1-\rho} \sum_{\nu<\alpha}\binom{\alpha}{\nu}\E\co\po\partial_x^\nu h(x,V^i)-\partial_x^\nu h\pf\partial_x ^{\alpha-\nu}\po\Psi\po\frac{\partial_i U_1}{2}(v_i-V_i^i)\pf\pf\cf\,,
    \end{split}
\end{equation*}
so that no derivatives of order $k+2$ appear, nor derivatives of the form $\na_x^{k+1}$, and we have
\[
\int_{\R^{2d}}|[\partial_x^{\alpha},\mathcal{L}_J^i]h|^2\dd\mu \leqslant C\int_{\R^{2d}}\sum_{\nu<\alpha}\sum_{k=0}^{|\alpha-\nu|}(\partial_{v_i}^k\partial_x^\nu h)^2\dd\mu.\]
Hence, the only derivatives of order $k+2$ that appear in $[\mathcal{L}_J,\nabla_x^{k+1}-\nabla_x^{k}\nabla_v]h$ come from the term $[\mathcal{L}_J,\nabla_x^{k}\nabla_v]h$ that we treated above. Using Young inequality, we get
\begin{multline*}
\int_{\R^{2d}}(\nabla_x^{k+1}-\nabla_x^{k}\nabla_v h)\cdot [\mathcal{L}_J,\nabla_x^{k+1}-\nabla_x^{k}\nabla_v]h\dd\mu \\\geqslant  -\frac{1}{8}\int_{\R^{2d}}|\nabla_x^{k+1}-\nabla_x^{k}\nabla_v h|^2\dd\mu - C\po\int_{\R^{2d}}\sum_{\underset{j\leqslant k+1-i}{i\leqslant k}}|\nabla_x^i\nabla_v^{j}h|^2\dd\mu  - \int_{\R^{2d}}|\nabla_x^{k}\nabla_v^{2} h|^2\dd\mu\pf\,,
\end{multline*}
which concludes the proof.
\end{proof}

Combining all the commutator terms, we get that for $i\in \cco 0,k-1\ccf$ there exists $C>0$ such that for all $\veps>0$ small enough,
\begin{multline*}
        \int_{\R^{2d}} \nabla_x^i\nabla_v^{k+1-i}h\cdot[\mathcal{L},\nabla_x^i\nabla_v^{k+1-i}]h\dd\mu \\ \geqslant -\frac{C}{\veps}\int_{\R^{2d}}|\nabla_x^i\nabla_v^{k+1-i}h|^2\dd\mu -C\int_{\R^{2d}}|\nabla_x^{i+1}\nabla_v^{k-i} h|^2\dd\mu -C\sum_{j=0}^{i-1}\int_{\R^{2d}}|\nabla_x^j\nabla_v^{k+2-i}h|^2\dd\mu \\
         -C\varepsilon\int_{\R^{2d}}\po\sum_{\underset{l\leqslant k+1-j}{j\leqslant i}}|\nabla_x^j\nabla_v^l h|^2+|\nabla_x^i\nabla_v^{k+2-i} h|^2\pf\dd\mu\,,
\end{multline*}
as well as
\begin{multline*}
      \int_{\R^{2d}}\nabla_x^{k}\nabla_vh\cdot[\mathcal{L},\nabla_x^{k}\nabla_v]h\dd\mu \geqslant \\ -C\po \frac{1}{\veps} + \frac{1}{\eta} \pf \int_{\R^{2d}}|\nabla_x^{k}\nabla_v h|^2\dd\mu  - \frac{\eta}{8}\int_{\R^{2d}}|(\nabla_x^{k+1}-\nabla_x^{k}\nabla_v)h|^2\dd\mu - C \sum_{j=0}^{k-1}\int_{\R^{2d}}|\nabla_x^j\nabla_v^2 h|^2\dd\mu \\ - C\veps\int_{\R^{2d}}\po\sum_{\underset{l\leqslant k+1-j}{j\leqslant k}}|\nabla_x^j\nabla_v^l h|^2+|\nabla_x^k\nabla_v^{2} h|^2\pf\dd\mu\,,
\end{multline*}
and finally
\begin{multline*}
        \int_{\R^{2d}}(\nabla_x^{k+1}-\nabla_k^{k}\nabla_v h)\cdot [\mathcal{L},\nabla_x^{k+1}-\nabla_x^{k}\nabla_v]h\dd\mu \\ \geqslant \frac{1}{4}\int_{\R^{2d}}|\nabla_x^{k+1}-\nabla_x^{k}\nabla_v h|^2\dd\mu - C\int_{\R^{2d}}\sum_{\underset{j\leqslant k+1-i}{i\leqslant k}}|\nabla_x^i\nabla_v^{j}h|^2\dd\mu  - C \new{\int_{\R^{2d}}}\po\sum_{j=0}^{k}|\nabla_x^j\nabla_v h|^2  + \sum_{j=0}^{k}|\nabla_x^j\nabla_v^2 h|^2 \pf\new{\dd\mu}\,.
\end{multline*}
Thanks to those expressions, we are now able to prove Proposition~\ref{recurrence}.

\begin{proof}[Proof of Proposition~\ref{recurrence}.]
The fact that inequality~\eqref{rec} holds for $k=1$ is proven in Lemma~\ref{lem:initialisation}. Let $k\in\N^*$, and suppose that inequality~\eqref{rec} holds for such $k$ and some $\rho_k>0$ and $\omega_{i,p}>0$, $0\leqslant i\leqslant p \leqslant k$. Fix some $\omega_{i,k+1}>0$, $0\leqslant i\leqslant k+1$, $\varepsilon>0$ small enough so that the previous formulas for the commutators hold, and choose $\eta=\omega_{k+1,k+1} / \omega_{k,k+1}$. By assumption we have
\begin{multline*}
    \int_{\R^{2d}} \Gamma(h) \dd\mu + \sum_{p=1}^{k+1} \int_{\R^{2d}}\po\sum_{i=0}^{p-1}\omega_{i,p}\Gamma_{i,p}(h) + \omega_{p,p}\Gamma_{p,p}(h)\pf\dd\mu \\ \geqslant \rho_k\po \new{N}_k(h) + \int_{\R^{2d}}\sum_{i=0}^k|\nabla_x^i\nabla_v^{k+1-i}h|^2\dd\mu \pf + \int_{\R^{2d}}\po\sum_{i=0}^k\omega_{i,k+1}\Gamma_{i,k+1}(h) + \omega_{k+1,k+1}\Gamma_{k+1,k+1}(h)\pf\dd\mu \,.
\end{multline*}
By using Proposition \ref{prop:calcul_Gamma}, the bounds on the commutators and the fact that 
\begin{align*}
\Gamma(\na_x^i\na_v^{j}h)& \geqslant 2\gamma|\na_x^i\na_v^{j+1} h|^2,\\ \Gamma((\na_x^{k+1}-\na_x^{k}\na_v)h) & \geqslant 2\gamma|(\na_x^{k+1}\na_v-\na_x^{k}\na_v^2)h|^2 \geqslant \gamma |\na_x^{k+1}\na_v h|^2- 2\gamma |\na_x^{k}\na_v^2 h|^2\,,
\end{align*}
the term or order $k+2$ are bounded below by 
\[
2\gamma \sum_{i=0}^{k} \omega_{i,k+1}|\na_x^i\na_v^{k+2-i}h|^2 + \gamma\omega_{k+1,k+1} |\na_x^{k+1}\na_v h|^2 - (C_1+2\gamma) \omega_{k+1,k+1} |\na_x^k\na_v^2h|^2 - C_1\varepsilon \sum_{i=0}^{k}\omega_{i,k+1} |\na_x^i\na_v^{k+2-i}h|^2\,,
\]
for some constant $C_1>0$.
We may bound from below the terms of order $k+1$ for $\varepsilon$ small enough by
\begin{multline*}
    \rho_k \sum_{i=0}^k|\nabla_x^i\nabla_v^{k+1-i}h|^2 - \sup_{0\leqslant i\leqslant k+1}\omega_{i,k+1}\frac{2C_2}{\varepsilon}\sum_{i=0}^k |\na_x^i\na_v^{k+1-i}h|^2 \\ - C_2\frac{\omega_{k,k+1}^2}{\omega_{k+1,k+1}} |\na_x^k\na_vh|^2 +  \frac{\omega_{k+1,k+1}}{8}  |(\nabla_x^{k+1}-\nabla_x^{k}\nabla_v)h|^2\,,
\end{multline*}
for some constant $C_2>0$. 
Finally, the term of order at most $k$ may be bounded from below by
\[
\po \rho_k\inf_{i+j\leqslant k}\omega_{i,i+j} - 3C_3\sup_{0\leqslant i\leqslant k+1}\omega_{i,k+1} \pf \sum_{i+j\leqslant k}|\na^i_x\na_v^jh|^2\,,
\]
for some constant $C_3>0$. Set $C_4 = \max(C_1,C_2,C_3)$ (which is independent from $\varepsilon$ and the weights $\omega_{i,j}$).
Now, choose $\veps$ and $\omega_{i,k+1}$ for $i\in\cco 1,k+1\ccf$ such that 
\[
\frac{\omega_{k+1,k+1}}{\omega_{k,k+1}}\leqslant \frac{\gamma-C_4\veps}{2\gamma+C_4},\quad \varepsilon < \gamma/C_4,\quad \omega_{i,k+1} < \varepsilon\rho_k/4C_4,\,1\leqslant i \leqslant k-1,\quad \omega_{k,k+1}\po \frac{2}{\veps}+\frac{\omega_{k,k+1}}{\omega_{k+1,k+1}}\pf < \rho_k/(2C_4)\,,
\]
and 
\[ 6C_4\sup_{0\leqslant i\leqslant k+1}\omega_{i,k+1} < \rho_k\inf_{i+j\leqslant k}\omega_{i,i+j}  \,. \]
We get
\begin{multline*}
\int_{\R^{2d}} \Gamma(h) \dd\mu + \sum_{p=1}^{k+1} \int_{\R^{2d}}\po\sum_{i=0}^{p-1}\omega_{i,p}\Gamma_{i,p}(h) + \omega_{p,p}\Gamma_{p,p}(h)\pf\dd\mu \\ \geqslant \tilde\rho_{k+1}\po \sum_{p=1}^{k+1} \int_{\R^{2d}}\po\sum_{i=0}^{p-1}\omega_{i,p}|\nabla_x^i\nabla_{v}^{p-i} f|^2 + \omega_{p,p}|(\nabla_x^p-\nabla_x^{p-1}\nabla_v)f|^2\pf  + \sum_{i=0}^k|\nabla_x^i\nabla_v^{k+2-i}h|^2\dd\mu \pf, 
\end{multline*}
for any 
\[ 
\tilde \rho_{k+1} < \min\po \frac{\rho_k}{2}, \frac{\omega_{k+1,k+1}}{8},\gamma \inf_{i+j\leqslant k+1}\omega_{i,i+j}\pf\,.
\]
As in the $H^1(\mu)$ case, an application of the Poincaré inequality~\eqref{eq:poincare} yields inequality~\eqref{rec}, which concludes the induction.
\end{proof}

\begin{proof}[Proof of Theorem~\ref{thm:conv-Hk}]
Our goal is to apply Lumer-Phillips theorem to the operator $\mathcal L+\rho_k/2I$, where $I$ denotes the identity operator of $H^k$. This theorem can be stated as follows: an operator $A$ on a Hilbert space generates a contraction semi-group if and only if it is maximally dissipative, see~\cite[Chapter IX, p.250]{Func_anal}. Fix $k\in\N$. The scalar product
\begin{equation*}
\left\langle f,g \right\rangle_{k}^2 = \int_{\R^{2d}}  fg \dd\mu + \int_{\R^{2d}} \sum_{p=1}^k \po \sum_{i=0}^{p-1}\omega_{i,p} \na_x^i\na_v^{p-i} f \cdot  \na_x^i\na_v^{p-i} g + \omega_{p,p} \po\na_x^p-\na_x^{p-1}\na_v\pf f \cdot \po\na_x^p-\na_x^{p-1}\na_v\pf g \pf \dd \mu
\end{equation*}
generates a norm equivalent to the usual norm of $H^k$.
Proposition~\ref{recurrence} and a density argument yield that the operator $\mathcal L+\rho_k/2I$, is dissipative:
\[
\forall h\in D(\mathcal L),\, \left\langle(\mathcal L+\rho_k/2I)h,h\right\rangle_{k}\leqslant 0,
\]
where $D(\mathcal L)$ denote the domain of $\mathcal L$ in $H^k$ defined by:
\[
f\in D(\mathcal L),\, g = \mathcal Lf \Leftrightarrow f\in H^k,\,  \lim_{t\rightarrow 0} \left\|\frac{P_tf-f}{t} - g\right\|_{H^k} = 0.
\]
We are left to show that $\mathcal L+\rho I$, for some $\rho<\rho_k/2$, is surjective. Fix such a $\rho<\rho_k/2$. Thanks to Proposition~\ref{recurrence}, we have that \[\Lambda:(f,g)\mapsto \left\langle-(\mathcal L+\rho I)f,g\right\rangle_{k}\] is coercive and continuous from $\po H^k\pf^2$ to $\R$ for $k\geqslant 1$. Hence, we may apply Lax-Milgram theorem to get that for all $g\in H^k$, there exists $f\in H^k$ such that for all $h\in H^k$, we have: 
\[
\left\langle-(\mathcal L+\rho I)f,h\right\rangle_{H^k} = \left\langle-g,h\right\rangle_{H^k},
\]
which implies that $f$ is a solution to the equation
\[
(\mathcal L+\rho I)f = g,
\]
and $\mathcal L+\rho_k/2I$ is maximally dissipative. Lumer-Phillips Theorem then yields that the semi-group generated by $\mathcal L+\rho_k/2$ is a contraction on $H^k$: for all $f\in H^k$
\[
\new{N}_k\po e^{\rho_kt/2}(P_tf-\mu(f))\pf \leqslant \new{N}_k\po  f - \mu(f)\pf,
\]
which concludes the proof.
\end{proof}

\subsection{Lyapunov function and re-weighted Sobolev norm}\label{s-sec:proof-estimates}

\new{The goal of this section is to establish Propositions~\ref{prop:deflyapunov} and~\ref{prop:conv-lyapunov}, thanks to which the previous exponential decay in $H^k$ is adapted to a weighted-Sobolev norm. }

Recall the definition of $V_b$ in \eqref{eq:defV_b}. 

\begin{proof}[Proof of Proposition~\ref{prop:deflyapunov}]
Let $a>0$ and $b\in(0,1)$. First, notice that
\begin{align*}
    \nabla_x V_b &= b\po\nabla U_0 - a\po\frac{v}{\sqrt{1+|x|^2}}-\frac{x\cdot v}{(1+|x|^2)^{3/2}}x\pf \pf V_b\,,\\
    \nabla_v V_b &= b\po v-a \frac{x}{\sqrt{1+|x|^2}}\pf V_b\,,\\
    \Delta_v V_b &= b\po d+b\po|v|^2+a^2\frac{|x|^2}{1+|x|^2}-2 a \frac{x\cdot v}{\sqrt{1+|x|^2}}\pf\pf V_b\,.
\end{align*}
Let us denote
\begin{align*}
    \mathcal{L}_1^* f &= -v\cdot\nabla_x f +(\nabla U_0-\gamma v)\cdot \nabla_v f + \gamma\Delta_v f\,, \\
    \mathcal{L}_2^* f &= \frac{2}{1-\rho}\sum_{i=1}^d\E\co(f(x,V^i)-f(x,v))\Psi\po\frac{\partial_i U_1}{2}(V_i^i-v_i)\pf\cf\,.
\end{align*}
Notice that $\mathcal{L}_1^*$ is the adjoint of the generator of the Langevin diffusion associated to the measure with density proportional to $\exp(-(U_0(x)+|v|^2/2)$, not $\mu$. When we integrate with respect to $\mu$, as we saw in Section~\ref{s-sec:mesure-inv}, additional terms appear in $(\mathcal{L}_D+\mathcal{L}_H)^*$ and $\mathcal{L}_J^*$ and then cancel out when they are summed. Therefore,  
although $\mathcal{L}_1^* \neq (\mathcal{L}_D+\mathcal{L}_H)^*$ and $\mathcal{L}_2^*\neq\mathcal{L}_J^*$, we indeed have $\mathcal{L}^* = \mathcal{L}^*_1+\mathcal{L}^*_2$.

Using that $-\nabla U_0(x)\cdot x\leqslant -\kappa |x|^2+C$, and the Young inequality, for all $\veps>0$,
\begin{multline*}
    \frac{\mathcal{L}_1^* V_b}{bV_b} \\= \po\frac{a}{\sqrt{1+|x|^2}}-\gamma(1-b)\pf|v|^2-a\frac{(x\cdot v)^2}{(1+|x|^2)^{3/2}}-a\frac{x\cdot\nabla U_0}{\sqrt{1+|x|^2}} +a\gamma(1-2b)\frac{x\cdot v}{\sqrt{1+|x|^2}}+b\gamma a^2\frac{|x|^2}{1+|x|^2} + d\gamma 
    \\\leqslant \po\frac{a}{\sqrt{1+|x|^2}}+\frac{a\gamma(1-2b)}{2\veps\sqrt{1+|x|^2}}-\gamma(1-b)\pf|v|^2 +\po \veps \gamma(1-2b)-\kappa\pf\frac{a|x|^2}{\sqrt{1+|x|^2}} + C+d\gamma+b\gamma a^2\,.
\end{multline*}

Then, by taking small enough $\veps$ and $a$, the terms in front of $|v|^2$ and $|x|^2$ are negative. This shows that there exist constants $C_1,C_2,C_3$ such that

\begin{equation}\label{eq:lyapulangevin}
\mathcal{L}^*_1 V_b \leqslant (C_1-C_2|v|^2-C_3|x|)V_b\,.
\end{equation}
Now,
\begin{align*}
    \mathcal{L}_2^* V_b &= CV_b\sum_{i=1}^d\E\co\po\exp\po b\po\frac{|V_i^i|^2-|v_i|^2}{2}-a \frac{x_i\cdot(V_i^i-v_i)}{\sqrt{1+|x|^2}}\pf\pf-1\pf\Psi\po\frac{\partial_i U_1}{2}(V_i^i-v_i)\pf\cf\,.
\end{align*} 
The term inside the previous exponential can be expressed as:
\begin{multline*}
    \frac{b}{2}\po (\rho^2-1)v_i^2+(1-\rho^2)\new{\xi}^2+2\rho\sqrt{1-\rho^2}v_i \new{\xi}-\frac{2 a x_i}{\sqrt{1+|x|^2}}((\rho-1)v_i+\sqrt{1-\rho^2}\new{\xi}))\pf \\
    \leqslant \frac{b}{2}\po \po\rho^2-1+\veps a(1-\rho)\pf v_i^2+(1-\rho^2)\new{\xi}^2-\po\frac{a\sqrt{1-\rho^2}}{\sqrt{1+|x|^2}}x_i-2\rho\sqrt{1-\rho^2}v_i\pf \new{\xi} + \frac{a(1-\rho)}{\veps(1+|x|^2)} x_i^2\pf\,.
\end{multline*}
For any $z\in\mathbb{C}$, if $\new{\xi}\sim\mathcal{N}(0,1)$,
\[\E[e^{z\new{\xi}}] = e^{z^2/2}\,,\]
and for any $x\in\R$, $\E[e^{x\new{\xi}^2}]<\infty$ if and only if $x<1/2$. Therefore, the previous exponential is integrable if $b<\frac{1}{1-\rho^2}$, which is always the case since $b<1$ and $\rho^2<1$. We have 
\begin{multline*}
    \E\co\exp\po b\po\frac{|V_i^i|^2-|v_i|^2}{2}-\alpha \frac{x_i\cdot(V_i^i-v_i)}{\sqrt{1+|x|^2}}\pf\pf\cf \\\leqslant \frac{\exp\po\frac{b}{2}\po \po\rho^2-1+\veps a(1-\rho)\pf v_i^2+\frac{a(1-\rho)}{\veps(1+|x|^2)} x_i^2\pf\pf}{\sqrt{1-b(1-\rho^2)}}  
    \times \exp\po \frac{b^2}{4} \po \frac{a\sqrt{1-\rho^2}}{\sqrt{1+|x|^2}}x_i- 2\rho\sqrt{1-\rho^2}v_i\pf^2\pf \\
    \leqslant C\exp\po \frac{b}{2} \po\rho^2-1+\veps a(1-\rho)+4b\rho^2(1-\rho^2)\pf v_i^2\pf 
    \times \exp\po\frac{b^2}{2}\po \frac{a(1-\rho)}{\veps}+a^2(1-\rho^2)\pf \frac{x_i^2}{1+|x|^2}\pf\,.
\end{multline*}
By taking small enough $a$ and $b$, the term in front of $v_i^2$ is negative, and the previous expectation is bounded. Recall that for all $s\in\R$, $\Psi(s)\leq C+|s|$, therefore
\begin{align*}
    \mathcal{L}_2^* V_b &= CV_b\sum_{i=1}^d\E\co(V_b(x,V^i)/V_b-1)\Psi\po\frac{\partial_i U_1}{2}(V_i^i-v_i)\pf\cf \\
    &\leqslant CV_b\sum_{i=1}^d\sqrt{\E[(V_b(x,V^i)/V_b-1)^2]\E\co \Psi^2\po\frac{\partial_i U_1}{2}(V_i^i-v_i)\pf\cf} \\
    &\leqslant CV_b\sum_{i=1}^d\sqrt{1+\E[(V_i^i-v_i)^2]} \leqslant C(1+|v|)V_b\,.
\end{align*}
Thus, there exist positive constants $C_1,C_2$ and $C_3$ such that 
\[\mathcal{L}^* V_b \leqslant (C_1-C_2|v|^2-C_3|x|)V_b\,,\]
which shows that $V_b$ is indeed a Lyapunov function for $\mathcal{L}^*$.
\end{proof}

\begin{proof}[Proof of Proposition~\ref{prop:conv-lyapunov}.]
We define, for all $k\in\N^*$ and $h\in\mathcal{C}_c^\infty$, the function $\phi_k$ by
\[\phi_k(h) = h^2 + \sum_{p=1}^k\sum_{i=0}^{p-1}\omega_{i,p}|\nabla_x^i\nabla_v^{p-i}h|^2+\omega_{p,p}|(\nabla_x^p-\nabla_x^{p-1}\nabla_v)h|^2\,,\]
where the $\omega_{i,j}$ are the same as in the norms $\new{N}_k$.
In order to prove the result of the lemma, it is sufficient to show the exponential decay of $\int\phi_k(f_t)V_b\dd\mu$, with $f_t = P_t f - \mu(f)$, for a certain $b\in(0,1)$.
Following the notations introduced in equation \eqref{def:Gamma}, and by denoting $\mathcal{L}_0 = \mathcal{L}_H+\mathcal{L}_D$,
\begin{align*}
    \partial_t \int_{\R^{2d}}\phi_k(f_t)V_b\dd\mu &= \int_{\R^{2d}}D_{h}\phi_k(f_t)\mathcal{L}f_t V_b\dd\mu \\
    &= \int_{\R^{2d}}D_{h}\phi_k(f_t)\mathcal{L}_0f_t V_b\dd\mu + \int_{\R^{2d}}D_{h}\phi_k(f_t)\mathcal{L}_J f_t V_b\dd\mu \\
    &= -2\int_{\R^{2d}}\Gamma_{\mathcal{L}_0,\phi_k}(f_t)V_b\dd\mu + \int_{\R^{2d}}\mathcal{L}_0(\phi_k(f_t)) V_b\dd\mu + \int_{\R^{2d}}D_h\phi_k(f_t)\mathcal{L}_Jf_t V_b\dd\mu \\
    &= -2\int_{\R^{2d}}\Gamma_{\mathcal{L}_0,\phi_k}(f_t)V_b\dd\mu + \int_{\R^{2d}}\phi_k(f_t)\mathcal{L}_0^* V_b\dd\mu + \int_{\R^{2d}}D_h\phi_k(f_t)\mathcal{L}_Jf_t V_b\dd\mu\,.
\end{align*}
The adjoint of $\mathcal{L}_0$ in $L^2(\mu)$ is given by
\[\mathcal{L}_0^* h = -v\cdot\nabla_x h +(\nabla U_0-\gamma v)\cdot \nabla_v h + \gamma\Delta_v h + h v\cdot\nabla U_1\, .\]
Therefore, using the inequality \eqref{eq:lyapulangevin} shown in the proof of Proposition~\ref{prop:deflyapunov}, there exists a $b\in(0,1)$ such that
\begin{align*}
    \mathcal{L}_0^* V_b &\leqslant b(-c_1|v|^2-c_2|x|+C)V_b + C'|v|V_b \\ &\leqslant (c_0|v|-c_1|v|^2-c_2|x|+c_3)V_b\, ,
\end{align*}
which shows that $V_b$ is a Lyapunov function for $\mathcal{L}_0^*$. Moreover, since Lemmas \ref{lem:commutator-hamiltonian} and \ref{lem:commutators-diffusion} give bounds on the commutators associated to $\mathcal{L}_H$ and $\mathcal{L}_D$ (without having to integrate them with respect to $\mu$ as it is the case for the commutators with the jump part), then in the proof of Proposition \ref{recurrence}, removing the parts associated with the jumps and the integrals with respect to $\mu$ shows that $\Gamma_{\mathcal{L}_0,\phi}\geqslant 0$. As a consequence, there exist $b\in(0,1),\rho>0,\eta>0$ and $C>0$ such that

\begin{align*}
    \partial_t \int_{\R^{2d}}\phi_k(f_t)V_b\dd\mu &\leqslant -\rho\int_{\R^{2d}}\phi_k(f_t)V_b\dd\mu + C\int_{\R^{2d}}\phi_k(f_t)\dd\mu + \int_{\R^{2d}}D_h\phi_k(f_t)\mathcal{L}_Jf_t V_b\dd\mu \\
    &\leqslant -\rho\int_{\R^{2d}}\phi_k(f_t)V_b\dd\mu + Ce^{-\eta t} + \int_{\R^{2d}}D_h\phi_k(f_t)\mathcal{L}_Jf_t V_b\dd\mu\,,
\end{align*}
where the last inequality comes from Theorem \ref{thm:conv-Hk}, since $\new{N}_k(h) = \int\phi_k(h)\dd\mu$. Let us now look at the term $\int_{\R^{2d}}D_h\phi_k(f_t)\mathcal{L}_Jf_t V_b\dd\mu$. The expression $D_h\phi_k(f_t)\mathcal{L}_Jf_t$ is a linear combination of terms of the form $\partial^\alpha f_t\partial^\alpha \mathcal{L}_J f_t$, with $\alpha\in\N^{2d}$. Notice that if we show that for any $\alpha\in\N^{2d}$, there exists $b\in(0,1)$ such that 
\[\int_{\R^{2d}} \partial^\alpha f_t\partial^\alpha \mathcal{L}_J(f_t) V_b\dd\mu\leqslant Ce^{-\rho t}\,,\]
then we will have
\[\int_{\R^{2d}}D_h\phi_k(f_t)\mathcal{L}_Jf_t V_b\dd\mu \leq Ce^{-\rho t}\,,\]
thus,
\[\partial_t \int_{\R^{2d}}\phi_k(f_t)V_b\dd\mu \leqslant -\rho\int_{\R^{2d}}\phi_k(f_t)V_b\dd\mu + Ce^{-\eta t}\,,\,\] 
and finally 
\[\int_{\R^{2d}}\phi_k(f_t)V_b\dd\mu \leqslant C e^{-\eta t}\,.\]
Let $\alpha\in\N^{2d}$. Successive integrations by parts yield

\begin{align*}
    \int_{\R^{2d}} \partial^\alpha h\partial^\alpha \mathcal{L}_J(h) V_b\dd\mu &= (-1)^{|\alpha|} \int_{\R^{2d}}\mathcal{L}_J(h)\partial^{\alpha}(\partial^{\alpha} h V_b e^{-H})\dd x\dd v \\
    &= (-1)^{|\alpha|}\sum_{\nu\leqslant\alpha}\binom{\alpha}{\nu}\int_{\R^{2d}}\mathcal{L}_J(h)\partial^{2\alpha-\nu}h\partial^{\nu}(V_b e^{-H})\dd\mu\,.
\end{align*}
Since $\nabla U_1$ and $\nabla^2 U_0$ are bounded (which in particular implies that $|\nabla U_0(x)| \leqslant C(1+|x|)$), 
for any $\nu\in\N^{2d}$, there exist $p>0$ and $C>0$ such that 
\[|\partial^{\nu} (V_b(x,v) e^{-H(x,v)})| \leqslant C(1+|x|^{p}+|v|^{p})V_b(x,v)e^{-H(x,v)}\,.\]
Moreover, since $\nabla U_0(x)\cdot x \geqslant \kappa |x|^2$ outside of a compact set, there exist $C>0$ such that
$U_0(x) \geqslant \kappa/2|x|^2-C$. This implies that for any $b<b'$ and any $p\in\N^*$,
\[ (1+|x|^p+|v|^p) V_b(x,v)\leqslant C V_{b'}\,.\]
Hence,
\begin{align*}
    \int_{\R^{2d}} \partial^\alpha h\partial^\alpha \mathcal{L}_J(h) V_b\dd\mu &\leqslant C\int_{\R^{2d}}|\mathcal{L}_J(h)|\po\sum_{\alpha\leqslant\nu\leqslant 2\alpha}|\partial^\nu h|\pf V_{2b}\dd\mu\\
    &\leqslant C\sqrt{\new{N}_{2|\alpha|}(h)\int_{\R^{2d}}|\mathcal{L}_J(h)|^2 V_{4b}\dd\mu}\,.
\end{align*}
Using the assumptions on the function $\Psi$, we have that
\[\Psi\po\pm\frac{\partial_i U_1}{2}(v_i-V^i_i)\pf\leqslant C(1+|v_i-V^i_i|)\,.\]
Recall that we also have shown in the proof of Proposition~\ref{prop:deflyapunov} that 
\[\E[V_b(x,V^i)]\leqslant C V_b\,,\]
and that, for any $b<b'$,
\[|v|^{2}V_b\leqslant C V_{b'}\,.\]
Therefore, using the reversibility of the law of $V^i$ with respect to the standard Gaussian measure,
\begin{align*}
    \int |\mathcal{L}_J h|^2 V_b\dd\mu &\leqslant C\sum_{i=1}^d\int \E\co (h(x,V^i)-h)^2\Psi^2\po \frac{\partial_i U_1}{2}(v_i-V^i_i)\pf\cf V_b\dd\mu \\
    &\leqslant C\sum_{i=1}^d\int\E\co (h(x,V^i)^2+h^2)\Psi^2\po \frac{\partial_i U_1}{2}(v_i-V^i_i)\pf\cf V_b\dd\mu \\
    &\leqslant C\sum_{i=1}^d\int h^2\po\E\co V_b(x,V^i)\Psi^2\po\frac{\partial_i U_1}{2}(V^i_i-v_i)\pf\cf+\E\co\Psi^2\po\frac{\partial_i U_1}{2}(v_i-V^i_i)\pf\cf V_b\pf \dd\mu \\ &\leqslant C\sum_{i=1}^d\int h^2(\E[V_b(x,V^i)(1+|v_i-V_i^i|^2)]+(1+|v_i|^2)V_b)\dd\mu \\
    &\leqslant C\int h^2 V_{b'}\dd\mu\,.
\end{align*}

Finally, notice that
\begin{align*}
    \partial_t\int f_t^2 V_b \dd\mu &= 2\int f_t\mathcal{L}f_t V_b \dd\mu = -2\int\Gamma(f_t)V_b\dd\mu + \int\mathcal{L}(f_t^2)V_b\dd\mu \\
    &\leqslant \int f_t^2\mathcal{L}^* V_b\dd\mu  
    \leqslant -\eta\int f_t^2 V_b\dd\mu + C\int f_t^2\dd\mu \\
    &\leqslant -\eta\int f_t^2 V_b\dd\mu + Ce^{-\rho t}\,.
\end{align*}
Therefore,
\[\int f_t^2 V_b\dd\mu \leqslant Ce^{-\rho t}\,.\]
Using the Theorem \ref{thm:conv-Hk}, we then have that for $b$ small enough, there exist $C>0$ and $\rho >0$, such that

\begin{align*}
\int_{\R^{2d}} \partial^\alpha f_t\partial^\alpha \mathcal{L}_J(f_t) V_b\dd\mu &\leqslant C\sqrt{\int_{\R^{2d}}\new{N}_{2|\alpha|}(f_t)\dd\mu\int_{\R^{2d}}f_t^2 V_{5b}\dd\mu} \\
&\leqslant Ce^{-\rho t}\,,
\end{align*}
which concludes the proof.
\end{proof}

\section{Weak error of the numerical scheme}\label{sec:weakerror}

In this section, devoted to the proof of Theorem~\ref{thm:TalayTubaro}, $\mathcal X=\T^d$. The structure of the proof is the following. First, in Section~\ref{sec:Lyapunovnumerique}, by a Lyapunov argument, we establish uniform-in-time Gaussian moment bounds for the BJAOAJB chain (this is Proposition~\ref{prop:lyapunovscheme}). This is then used in Section~\ref{sec:finitetimeerror} to provide an upper bound for the weak error of the scheme in a finite-time horizon, leading to Theorem~\ref{thm:finitetimeerror}, which is of interest for itself (the proof essentially relies on Taylor expansions of the semi-groups, the expectation of the local errors being bounded thanks to the Gaussian moments established in Proposition~\ref{prop:lyapunovscheme}). Geometric ergodicity for the BJAOAJB chain is established in Section~\ref{sec:ergodicityscheme} (since a Lyapunov function is already available, in order to apply Harris' theorem, it only remains to prove a minorization condition). It is then possible to let time go to infinity in Theorem~\ref{thm:finitetimeerror} and, following the classical proof of Talay and Tubaro, to conclude the proof of Theorem~\ref{thm:TalayTubaro}, as detailed in Section~\ref{sec:expansion}. For completeness,  the classical proof of Corollary~\ref{cor:quadraticRisk} is provided in Section~\ref{sec:quadraRisk}.

\subsection{Lyapunov function for the numerical scheme}\label{sec:Lyapunovnumerique}

For all $z=(x,v)\in\T^d\times\R^d$ and $b>0$, we define the function $V_b$ by
\[V_b(z) = e^{b|v|^2}\,.\]
In the following, we will denote by $(\overline{Z}_n)_{n\in\N}$ the discrete-time Markov chain with transition kernel $Q$ defined in~\eqref{eq:transition-scheme}, corresponding to the BJAOAJB splitting scheme with time step $\delta$, and by $A,B,J$ and $O$ the kernels of the different steps of the scheme, namely $B=e^{\delta/2\mathcal{L}_B} ,A=e^{\delta/2\mathcal{L}_A}, J=e^{\delta/2\mathcal{L}_J}, O=e^{\delta\mathcal{L}_O}$.
The goal of this section is to show the following result.
\begin{prop}\label{prop:lyapunovscheme}
Let $b<\frac{1}{2}$. There exist $\delta_0,C,C'>0$ such that for all $\delta\leqslant\delta_0$,
\[QV_b(z)\leqslant(1-C\delta)V_b(z)+C'\delta\,.\]
As a consequence, for all $n\in\N$,
\[\E_z[V_b(\overline{Z}_n)] = Q^nV_b(z)\leqslant e^{-C\delta n}V_b(z)+\frac{C'}{C}\,.\]
\end{prop}

In order to prove this proposition, we will need to bound the terms $BV_b,AV_b,OV_b$ and $JV_b$. First, since $A$ only acts on positions, $AV_b=V_b$.
\begin{lem}\label{lem:Bpartlyapu}
For all $\delta_0,\veps>0,b\in(0,1/2)$, there exists $C_\veps>0$ such that for all $0<\delta\leqslant\delta_0$,
\[BV_b\leqslant (1+C_\veps\delta)e^{b(1+\veps\delta)|v|^2}\,.\]
\end{lem}
\begin{proof}
For all $z=(x,v)\in\T^d\times\R^d$, using the Young inequality,
\begin{equation*}
    BV_b(z) = e^{b|v-\frac{\delta}{2}\nabla U_0(x)|^2}= e^{b\po |v|^2-\delta\nabla U_0(x)\cdot v+\frac{\delta^2}{4}|\nabla U_0(x|^2\pf } \leqslant e^{b\po(1+\veps\delta)|v|^2+\delta\po \frac{\delta}{4}+\frac{1}{4\veps}\pf||\nabla U_0||_\infty^2\pf}\leqslant (1+C_\veps\delta)e^{b(1+\veps\delta)|v|^2}\,,
\end{equation*}
the last inequality being satisfied if $\delta$ is small enough.
\end{proof}

\begin{lem}\label{lem:Opartlyapu}
There exists $\delta_0>0$ (independent from $b$) such that for all $b\in(0,1/2)$, there exist $C_1,C_2>0$ such that for $0<\delta\leqslant\delta_0$,
\begin{align*}
    OV_b(z)\leqslant (1+C_2\delta)e^{b(1-C_1\delta)|v|^2}\,.
\end{align*}
\end{lem}
\begin{proof}
We start by noticing that for all $a_1,a_2\in\R$ and $\new{\xi}\sim\mathcal{N}(0,1)$, the random variable $\exp(a_1\new{\xi}^2+a_2\new{\xi})$ is integrable if and only if $a_1<1/2$. In that case, by recognizing the density of $\mathcal{N}(a_2/(1-2a_2),1/(1-2a_1))$,
\begin{align*}
    \E[\exp(a_1\new{\xi}^2+a_2\new{\xi})] &= \frac{1}{\sqrt{2\pi}}\int_\R \exp(a_1x^2+a_2x-x^2/2)\dd x \\ &= \frac{1}{\sqrt{2\pi}}\int_\R\exp\po\frac{-(1-2a_1)}{2}\po x-\frac{a_2}{1-2a_1}\pf^2+\frac{a_2^2}{2(1-2a_1)}\pf \\
    &= \frac{1}{\sqrt{1-2a_1}}\exp\po\frac{a_2^2}{2(1-2a_1)}\pf\,.
\end{align*}
Using this remark,
\begin{align*}
    OV_b(z) &=  \E\co\exp\po b|e^{-\gamma\delta}v+\sqrt{1-e^{-2\gamma\delta}}\new{\xi}|^2\pf\cf \\
    &= \exp\po be^{-2\gamma\delta}|v|^2\pf \E\co\exp\po 2be^{-\gamma\delta}\sqrt{1-e^{-2\gamma\delta}}v\cdot \new{\xi}+b(1-e^{-2\gamma\delta})|\new{\xi}|^2\pf\cf \\
    &= \frac{1}{\po 1-2b(1-e^{-2\gamma\delta})\pf^{d/2}}\exp\po be^{-2\gamma\delta}|v|^2+\frac{2b^2e^{-2\gamma\delta}(1-e^{-2\gamma\delta})|v|^2}{1-2b(1-e^{-2\gamma\delta})}\pf \\
    &= \frac{1}{\po 1-2b(1-e^{-2\gamma\delta})\pf^{d/2}}\exp\po b|v|^2\frac{e^{-2\gamma\delta}}{1-2b(1-e^{-2\gamma\delta})} \pf\,.
\end{align*}
Then, there exist $\delta_0>0$ and constants $C_1,C_2>0$ such that for all $0<\delta\leqslant\delta_0$,
\[ \frac{e^{-2\gamma\delta}}{1-2b(1-e^{-2\gamma\delta})} \leqslant 1-C_1\delta\,,\]
and
\[ \frac{1}{\po 1-2b(1-e^{-2\gamma\delta})\pf^{d/2}} \leqslant 1+C_2\delta\,.\]
Hence
\[ OV_b(z)\leqslant (1+C_2\delta)e^{b(1-C_1\delta)|v|^2}\,. \]
Notice that $\delta_0$ can in fact be chosen independently from $b$.
On one side, 
\[\frac{1}{\po 1-2b(1-e^{-2\gamma\delta})\pf^{d/2}} \leqslant \frac{1}{(1-2\gamma\delta)^{d/2}}\leqslant 1+d\gamma\delta+O(\delta^2),\]
and on the other side, the Taylor-Lagrange formula shows that there exists $0<c<2\gamma\delta$ such that
\[\frac{e^{-2\gamma\delta}}{1-2b(1-e^{-2\gamma\delta})} = 1-2\gamma(1-2b)\delta+\po 4b^2-3b+\frac{1}{2}\pf c^2.\]
If $1/4<b<1/2$, $4b^2-3b+\frac{1}{2}<0$ so we directly have
\[\frac{e^{-2\gamma\delta}}{1-2b(1-e^{-2\gamma\delta})} \leqslant 1-2\gamma(1-2b)\delta\,,\]
without any condition on $\delta$. Now, if $0<b<1/4$, $4b^2-3b+\frac{1}{2}<1/2$, and
\[\frac{e^{-2\gamma\delta}}{1-2b(1-e^{-2\gamma\delta})} \leqslant 1-\gamma\delta+2\gamma^2\delta^2,\]
and $\delta_0$ can be chosen independently from $b$.
\end{proof}

\begin{lem}\label{lem:Jpartlyapu}
For all $0<b_1<b_2<1/2$, there exist $\delta_0>0$ and $K>0$ such that for all $\delta\leqslant\delta_0$ and all $b\in[b_1,b_2]$,
\[JV_b(z)\leqslant V_b(z)(1+\delta K) + \delta K\,.\]
\end{lem}
\begin{proof}

Let $0<b_1<b_2<1/2$ and $b\in[b_1,b_2]$. In the proof of this lemma, constants $C$ can depend on $b_1$ and $b_2$, but not on $b$. We will denote $Z_k=(X_k,V_k)$ the jump process after $k$ jumps, and $Z_\delta$ the jump process after a time $\delta$. We start by expanding $JV_b(z)$ according to the number of jumps:
\[ \E_z[V_b(Z_\delta)] = \sum_{k=0}^\infty \E_z[V_b(Z_\delta)\1_{\{k \text{ jumps}\}}]\,.\]
First, 
\[\E_z[V_b(Z_\delta)\1_{\{0 \text{ jumps}\}}] = V_b(z)\PP(0 \text{ jumps}) \leqslant V_b(z)\,.\]
Then, let us show by induction on $k$ that there exists a constant $C$ such that for all $k\geqslant 1$,
\begin{equation}\label{eq:reclyapujump}
    \E_z[V_b(Z_\delta)\1_{\{k \text{ jumps}\}}] \leqslant C^k\delta^k(1+V_b(z))\,.
\end{equation}
In order to prove this inequality, we will use some useful bounds. First, recall that there exist a constant $C$ such that for all $z\in\T^d\times\R^d$,
\begin{equation}\label{eq:majtauxsaut}
    \lambda(z) \leqslant C(1+|v|)\,.
\end{equation}
As we saw in the proof of Lemma \ref{lem:Opartlyapu}, if $a_1<1/2$, then the random variable $\exp(a_1\new{\xi}^2+a_2\new{\xi})$ is integrable and 
\begin{align*}
    \E[\exp(a_1\new{\xi}^2+a_2\new{\xi})] &= \frac{1}{\sqrt{1-2a_1}}\exp\po\frac{a_2^2}{2(1-2a_1)}\pf\,.
\end{align*}
More generally, by recognizing the density of a Gaussian random variable with mean $a_2/(1-2a_1)$ and variance $1/(1-2a_1)$,
\[\E[|\new{\xi}|^n\exp(a_1\new{\xi}^2+a_2\new{\xi})] = \frac{1}{\sqrt{1-2a_1}}\exp\po\frac{a_2^2}{2(1-2a_1)}\pf\E[|X|^n]\,,\]
where $X\sim\mathcal{N}\po a_2/(1-2a_1),1/(1-2a_1)\pf$. As a consequence, the random variable $V_b(Z_1)$ is integrable if and only if $b(1-\rho^2)<1/2$, which is always true since $b\in(0,1/2)$.
\begin{align*}
    \E_z[V_b(Z_1)] &\leqslant \sum_{i=1}^d\E_z\co\exp(b|\rho v_i + \sqrt{1-\rho^2}\new{\xi}|^2)\Psi\po\frac{\partial_i U_1}{2}\po (1-\rho)v_i+\sqrt{1-\rho^2} \new{\xi}\pf\pf\cf
    \\ &\leqslant C \sum_{i=1}^d\E_z\co\exp(b|\rho v_i + \sqrt{1-\rho^2}\new{\xi}|^2)(1+ |(1-\rho)v_i+\sqrt{1-\rho^2} \new{\xi}|)\cf
    \\ &\leqslant C \sum_{i=1}^d \exp(b\rho^2|v_i|^2)\E[\exp(b(1-\rho^2)\new{\xi}^2+2b\rho\sqrt{1-\rho^2}v_i\new{\xi})(1+2|v_i|+|\new{\xi}|)]\,,
\end{align*}
so that
\begin{equation}\label{eq:vbz1}
 \E_z[V_b(Z_1)] \leqslant C(1+|v|^2)e^{b\alpha|v|^2}\,,
\end{equation}
where we denoted $\alpha = \frac{\rho^2}{1-2b(1-\rho^2)}$. Notice that
\begin{align*}
    \alpha<1 \iff \rho^{2}< 1-2b(1-\rho^{2}) \iff b<1/2\,,
\end{align*}
and that $b\alpha = b\rho^2/(1-2b(1-\rho^2))$, seen as a function of $b$, is continuous and strictly positive on $(0,1/2)$, so is bounded and reaches its extrema on $[b_1,b_2]$. Since $|v|^2 e^{b\alpha|v|^2}= o(V_b)$ and $|v|=o(V_b)$ when $v$ goes to infinity, let 
\[K=\max(R^2,R)\sup_{b\in[b_1,b_2]}e^{b\alpha R^2}\,,\]
where $R>0$ satisfies
\[ |v|\geqslant R \implies \max(|v|,|v|^2)\leqslant \inf_{b\in[b_1,b_2]} e^{b(1-\alpha)|v|^2}\,,\]
so that
\begin{equation}\label{eq:petito}
    |v|^2 e^{b\alpha|v|^2}\leqslant V_b(z)+K\,,
\end{equation}
and
\begin{equation}\label{eq:petito2}
|v|\leqslant V_b(z)+K\,.
\end{equation}
Furthermore, we also have that
\begin{equation}\label{eq:integralemaj}
    \int_0^\delta (\delta-t)\lambda(z)e^{-\lambda(z)t}\dd t = \frac{\lambda(z)\delta+e^{-\lambda(z)\delta}-1}{\lambda(z)} \leqslant \frac{\lambda(z)\delta^2}{2}\leqslant C\delta^2(1+|v|)\,.
\end{equation}
We start by showing \eqref{eq:reclyapujump} in the case $k=1$. We suppose that $\lambda(z)>0$ (otherwise there is no jump and the result is straightforward). By conditioning with respect to the only jump time,

\begin{align*}
\E_z[V_b(Z_\delta)\1_{\{1 \text{ jump}\}}] &= \int_{0}^\delta \lambda(z)e^{-\lambda(z)t}\int_{\T^d\times\R^d}q(z,\dd \tilde{z})V_b(\tilde{z})e^{-(\delta-t)\lambda(\tilde{z})}\dd t \\
&\leqslant  \int_{0}^\delta \lambda(z) e^{-\lambda(z)t}\int_{\T^d\times\R^d} q(z,\dd \tilde{z})V_b(\tilde{z})\dd t \\
&\leqslant \E_z[V_b(Z_1)]\lambda(z)\delta\,.
\end{align*}
Then, using~\eqref{eq:vbz1},~\eqref{eq:petito} and~\eqref{eq:petito2}, we have 
\begin{equation*}
\E_z[V_b(Z_\delta)\1_{\{1 \text{ jump}\}}] \leqslant C \delta e^{b\alpha|v|^2}(1+|v|^2) 
\leqslant C \delta \po V_b+K\pf \leqslant C\delta(1+V_b)\,,
\end{equation*}
 which proves the inequality for $k=1$. Now let $k\geqslant 1$ and assume that \eqref{eq:reclyapujump} holds for $k$. By conditioning with respect to the first jump time,
\begin{align*}
\E_z[V_b(Z_\delta)\1_{\{k+1 \text{ jumps}\}}] &= \int_0^\delta \lambda(z)e^{-\lambda(z)t}\int_{\T^d\times\R^d}q(z,\dd \tilde{z})\E_{\tilde{z}}[V_b(Z_{\delta-t})\1_{\{k \text{ jumps}\}}]\dd t \\
&\leqslant C^k\int_0^\delta (\delta-t)^k\lambda(z)e^{-\lambda(z)t}\int_{\T^d\times\R^d}q(z,\dd \tilde{z})(V_b(\tilde{z})+1)\dd t \\
&\leqslant \delta^{k-1}C^k\po \E_z[V_b(Z_1)]+ 1\pf\int_0^\delta (\delta-t)\lambda(z)e^{-\lambda(z)t}\dd t\,.
\end{align*}
Then, using~\eqref{eq:integralemaj},
\begin{equation*}
\E_z[V_b(Z_\delta)\1_{\{k+1 \text{ jumps}\}}] \leqslant \delta^{k+1}C^{k+1}\po e^{b\alpha|v|^2}+1\pf(1+|v|^2)\leqslant \delta^{k+1}C^{k+1} \po 1+V_b\pf\,,
\end{equation*}
which proves the inequality for $k+1$ and ends the induction. By choosing $\delta<1/C$, we can sum the inequalities \eqref{eq:reclyapujump}:
\begin{multline*}
    JV_b(z) = \E_z[V_b(Z_\delta)] = \sum_{k=0}^\infty \E_z[V_b(Z_\delta)\1_{\{k \text{ jumps}\}}] \leqslant V_b(z) + (1+V_b(z))\sum_{k=1}^\infty (C\delta)^k \\ \leqslant V_b(z) + \delta\frac{C}{1-C\delta}(1+V_b(z))\,,
\end{multline*}
so there exist $\delta_0$ and $K>0$ such that for all $\delta<\delta_0$,
\[JV_b \leqslant V_b(1+\delta K)+\delta K.\]
\end{proof}
We are now able to prove Proposition \ref{prop:lyapunovscheme}.

\begin{proof}[Proof of Proposition \ref{prop:lyapunovscheme}]
Let $b\in(0,1/2)$ and $\veps>0$, whose value will be specified later.
Thanks to Lemma \ref{lem:Bpartlyapu}, for all $\delta_1>0$, there exist $C_\veps>0$ such that for all $\delta<\delta_1$,
\[BV_b\leqslant (1+C_\veps\delta) e^{b(1+\veps\delta)|v|^2}\,.\]
Now let $\delta_2>0$ such that $b_2=b(1+\veps\delta_2)<1/2$. Notice that for all $0<\delta\leqslant\delta_2$, $b<b':=b(1+\veps\delta)<b_2$. Thanks to Lemma 
\ref{lem:Jpartlyapu}, there exist $\delta_3>0$ and $K>0$ (whose value is uniform over $[b,b_2]$) such that for all $\delta<\delta_3$,
\[JV_{b'}\leqslant V_{b'}(1+\delta K)+\delta K\,.\]
Therefore, for all $\delta<\min(\delta_1,\delta_2,\delta_3)$,
\[JBV_b(z)\leqslant (1+C_\veps\delta) JV_{b'}(z) \leqslant (1+C_\veps\delta) V_{b'}(z)(1+\delta K)+\delta (1+C_\veps\delta) K\,,\]
and
\[AJBV_b(z) \leqslant (1+C_\veps\delta) AV_{b'}(z)(1+\delta K)+\delta (1+C_\veps\delta) K = (1+C_\veps\delta) V_{b'}(z)(1+\delta K)+\delta (1+C_\veps\delta) K \,.\]
Lemma \ref{lem:Opartlyapu} gives a $\delta_4>0$ (whose value is independent from $b$ or $b'$) such that for $\delta<\delta_4$,
\begin{align*}
    OV_{b'}&\leqslant (1+C_2\delta)\exp\po b'(1-C_1\delta)|v|^2\pf \\
    &\leqslant (1+C_2\delta)\exp\po b(1+\veps\delta)(1-C_1\delta))|v|^2\pf \\
    &\leqslant (1+C_2\delta)\exp\po b(1-(C_1-\veps)\delta)|v|^2\pf\,.
\end{align*}
We would like to have $\veps<C_1$ (In fact, because of the two applications of the kernel $B$, we will need to have $\veps<C_1/2$). When looking at the proof of the Lemma~\ref{lem:Opartlyapu}, the value of $C_1$ can depend on $b'$ (and therefore on $\veps$) if $1/4<b'<1/2$, in which case $C_1=2\gamma(1-2b')=2\gamma(1-2b(1+\veps\delta))$. The constant $\veps$ can always be chosen smaller than half of this value if $\veps<\frac{\gamma(1-2b)}{1+4\gamma b \delta_1}$. In the other cases, $C_1$ does not depend on $b'$, so we can also choose $\veps< C_1/2$. We now denote $a=C_1-\veps$. For all $\delta<\min(\delta_1,\delta_2,\delta_3,\delta_4)$,
\[OAJBV_b(z) \leqslant (1+C_\veps\delta) O V_{b'}(z)(1+\delta K)+\delta (1+C_\veps\delta) K \leqslant (1+C_\veps\delta)(1+C_2\delta) \exp(b(1-a\delta)|v|^2)(1+\delta K)+\delta (1+C_\veps\delta) K\,,\]
and 
\[AOAJBV_b(z) \leqslant (1+C_\veps\delta)(1+C_2\delta) \exp(b(1-a\delta)|v|^2)(1+\delta K)+\delta (1+C_\veps\delta) K\,. \]
From there, in order to re-apply the bounds on the jump kernel, we can let $\delta_5$ such that $b_3 = b(1-a\delta_5)>0$, which gives a $\delta_6>0$ such that for $0<\delta\leqslant\min(\delta_1,\delta_2,\delta_3,\delta_4,\delta_5,\delta_6)$, corresponding to a $b_3<b(1-a\delta)<b$,
\[J\po e^{b(1-a\delta)|v|^2}\pf\leqslant e^{b(1-a\delta)|v|^2}(1+\delta K')+\delta K'\,,\]
which yields
\[JAOAJBV_b(z) \leqslant (1+C_\veps\delta)(1+C_2\delta)(1+\delta K)((1+\delta K') \exp(b(1-a\delta)|v|^2)+\delta K')+\delta (1+C_\veps\delta) K\,.\]
Finally, for the last application of the kernel $B$, we need to have
\[(1+\veps\delta)(1-a\delta) = 1-(a-\veps)\delta -a\veps\delta^2<1\,,\]
which is possible because $a=C_1-\veps$ and that $\veps<C_1/2$. We then denote $c=a-\veps$. Finally, let $\delta_0 = \min\po\delta_1,\delta_2,\delta_3,\delta_4,\delta_5,\delta_6\pf$. There exists  $C'>0$ such that for all $0<\delta<\delta_0$,
\begin{align*}
    BJAOAJBV_b(z)&\leqslant (1+C_\veps\delta)(1+C_2\delta)(1+\delta K)((1+\delta K') B\exp(b(1-a\delta)|v|^2)+\delta K')+\delta (1+C_\veps\delta) K  \\
    &\leqslant (1+C_\veps\delta)(1+C_2\delta)(1+\delta K)((1+\delta K')(1+C_\veps\delta) \exp(b(1-c\delta)|v|^2)+\delta K')+\delta (1+C_\veps\delta) K 
     \\ &\leqslant V_b(z)(1+C'\delta)^5\exp(-bc\delta|v|^2)+\delta C' \,,
\end{align*}
and hence
\begin{equation}\label{eq:demohop}
    BJAOAJBV_b(z) \leqslant V_b(z)\exp(5C'\delta-bc\delta|v|^2)+\delta C'\,.
\end{equation}
Now, writing $R_0^2 = 10C'/(bc)$, distinguishing cases, we see that
\begin{multline*}
    e^{5C'\delta-bc\delta|v|^2} \leqslant e^{-5C'\delta}\1_{|v|\geqslant R_0} + e^{5C'\delta}\1_{|v|< R_0}  \leqslant \po 1 - 5C' e^{-5C'\delta_0}\delta\pf \1_{|v|\geqslant R_0} + \po 1 + 5C' e^{5C'\delta_0}\delta\pf\1_{|v|< R_0} \\
    \leqslant 1-c'\delta + C''\delta \1_{|v|<R_0}\,,
\end{multline*} 
for some $c',C''>0$ independent from $\delta$.
Plugging this in \eqref{eq:demohop} we obtain that there exist $K,K'>0$ independent from $\delta$ such that for all $z$,
\[QV_b(z)\leqslant (1-K\delta)V_b(z)+K'\delta\,,\]
which by induction yields
\begin{align*}
    Q^n V_b(z) &\leqslant (1-K\delta)^n V_b(z)+K'\delta\sum_{k=0}^{n-1}(1-K\delta)^k \\
    &\leqslant  e^{-K\delta n}V_b(z)+\frac{K'}{K}\,,
\end{align*}
concluding the proof.
\end{proof}
\subsection{Finite-time error expansion of the numerical scheme}\label{sec:finitetimeerror}

In this section, $Z_t = (X_t,V_t)$ denotes the continuous-time process with generator $\mathcal{L}$ and with Markov semi-group $P_t$, $(\overline{Z}_n)_{n\in\N}$ is the discrete-time process corresponding to the BJAOAJB splitting-scheme with time-step $\delta$.
We will denote 
\[\overline{P}_t = e^{\frac{t}{2} B}e^{\frac{t}{2} J}e^{\frac{t}{2} A}e^{t O}e^{\frac{t}{2} A}e^{\frac{t}{2} J}e^{\frac{t}{2} B}\,,\]
so that $\overline{P}_\delta f = \E[f(\overline{Z}_1)]$. The goal of this section is to show that the finite-time weak error of the numerical scheme is of order $2$ in the time-step.
\begin{thm}\label{thm:finitetimeerror}
Let $\delta_0>0$ such that Proposition \ref{prop:lyapunovscheme} holds, $f\in\mathcal{A}$, $z\in\T^d\times\R^d$. There exists a function  $t\mapsto C(t,z)>0$ from $\R_+$ to $\R_+$ such that for all $n\in\N$ and $\delta \in(0,\delta_0]$,
\[|\E_z[f(Z_{n\delta})]-\E_z[f(\overline{Z}_n)]|\leqslant C(n\delta,z)\delta^2\,.\]
 More precisely, by denoting $S_3$ the operator 
\begin{multline*}
    S_3 = \frac{7}{144}(2[L_O,[L_O,L_A]]-[L_A,[L_A,L_O]] + \\
    2[L_A+L_O,[L_A+L_O,L_J]]-[L_J,[L_J,L_A+L_O]]+2[L_A+L_J+L_O,[L_A+L_J+L_O,L_B]]-[L_B,[L_B,L_A+L_J+L_O]])\,,
\end{multline*}
then
\begin{align*}
    \E_z[f(Z_{n\delta})]-\E_z[f(\overline{Z_n})] &= \delta^3\sum_{k=0}^{n-1}\E[S_3 P_{(n-k-1)\delta}f(\overline{Z_k})] + R_{n,z}\,,
\end{align*}
with
\begin{align*}
\left|\sum_{k=0}^{n-1}\E[S_3 P_{(n-k-1)\delta}f(\overline{Z_k})]\right| &\leqslant  \frac{C}{\delta} \po V_b(z) e^{-qn\delta} +1\pf\,,
\end{align*}
and 
\begin{align*}
|R_{n,z}| &\leqslant C\delta^3  \po V_b(z) e^{-qn\delta} +1\pf\,.
\end{align*}
\end{thm}
In order to prove this theorem, we will need the following lemmas.  Recall that given $V:\T^d\times \R^d\rightarrow [1,\infty)$, for $f:\T^d\times\R^d\rightarrow \R$ we write $\|f\|_V=\|f/V\|_\infty$. Let, for all $b\in(0,1/2),k\in\N^*$ and $f\in\mathcal{A}$,

\[||f||_{b,k} = \sum_{|\alpha|\leqslant k} ||\partial^\alpha f||_{V_b}\,.\]
When $\mathcal X = \T^d$, the set $\mathcal{A}$ reads
\begin{align*}
    \mathcal{A} &= \left\{ f\in\mathcal C^{\infty}(\T^d\times\R^d,\R)|\forall \alpha\in\N^{2d},\,\exists\, C>0,\,b\in(0,1/2),\, |\partial^{\alpha} f|\leqslant CV_b\right\} \\
    &= \left\{ f\in\mathcal C^{\infty}(\T^d\times\R^d,\R)|\forall k\in\N,\,\exists\, b\in(0,1/2),\, ||f||_{b,k}<\infty\right\}\,.
\end{align*}

\begin{lem}\label{lem:stability}

For all $t_0>0$, $f\in\mathcal A$, $k\geqslant 2$, $b\in(0,1/2)$ such that $||f||_{b,k}<\infty$ and $b'\in(b,1/2)$, there exists $C>0$ such that for $t\in[0,t_0]$ and $T\in\{\mathcal{L}_A,\mathcal{L}_B,\mathcal{L}_J,\mathcal{L}_O$, $e^{t\mathcal{L}_A},e^{t\mathcal{L}_B},e^{t\mathcal{L}_J},e^{t\mathcal{L}_O}\}$, 
\[ ||Tf||_{b',k'}\leqslant C||f||_{b,k}\,,\]
with $k' =k,k-1 \text{ or } k-2$ (depending on $T$). In particular, the operator $T$ leaves stable the set $\mathcal{A}$.
\end{lem}

In the proof of this lemma, we will need the formula 
for higher-order derivatives of composite functions \cite{Ma2009}:
\begin{prop}[Higher derivatives of composite functions]\label{composee}
Let $g:\R^n\rightarrow\R^l$ and $f:\R^l\rightarrow\R$ be smooth functions, and $\alpha=(\alpha_1,\dots,\alpha_n)\in\N^n$ be a multi-index. $\alpha$ is said to be decomposed into $s\in\N^{*}$ (pairwise distinct) parts $p_1,...,p_s$ in $\N^n$ with multiplicities $m_1,\dots,m_s$ in $\N^l$ (the $p$'s and the $m$'s are multi-indexes) if 
\[ \alpha = |m_1|p_1+\dots |m_s|p_s\]
holds and all parts are different. We define the total multiplicity by
\[ m = m_1+\dots +m_s\]
and $\mathcal{D}_\alpha$ as the set of all such decompositions of $\alpha$. Then,
\[\partial^\alpha (f\circ g)(x) = \alpha!\sum_{(s,p,m)\in\mathcal{D}} (\partial^m f)(g(x))\prod_{k=1}^s\frac{1}{m_k!}\co\frac{1}{p_k!}\partial^{p_k} g(x)\cf^{m_k} \]
where $\alpha! = \prod_j \alpha_j!$ and $x^\alpha = \prod_j x_j^{\alpha_j}$. The term $\co\frac{1}{p_k!}\partial^{p_k} g(x)\cf^{m_k}$ is thus the short form for
\[\prod_{i=1}^{l}\co\frac{1}{p_{k,1}!\dots p_{k,n}!}\partial^{p_k}g_i(x)\cf^{m_{k,i}}\]
where $g=(g_1,\dots,g_l)$.
\end{prop}
\begin{proof}[Proof of Lemma \ref{lem:stability}]
Let $f\in\mathcal A$, $k\in\N^*$, $b\in(0,1/2)$ such that $||f||_{b,k}<\infty$ and $b'\in(b,1/2)$. As a first step, we prove the result for $t_0$ small enough.

First, for all $\alpha\in\N^{2d}$ with $|\alpha|\leqslant k-1$,
\begin{equation*}
|\partial^\alpha \mathcal{L}_B f| = |\partial^{\alpha}(\nabla U_0\cdot\nabla_v f)| = \sum_{i=1}^d\sum_{\nu\leqslant\alpha}|\partial^\nu\partial_i U_0\partial^{\alpha-\nu}\partial_{v_i}f|
\leqslant CV_b(z)||f||_{b,k}\,,
\end{equation*}
thus
\[ ||\mathcal{L}_B f||_{b,k-1} \leqslant C||f|_{b,k}\,.\]
Similarly, for $|\alpha|\leqslant k-1$,
\begin{equation*}
    |\partial^\alpha \mathcal{L}_A f| = |\partial^\alpha(v\cdot\partial_x f)| \leqslant |v|\sum_{i=1}^d \sum_{\nu\leqslant\alpha}\partial^\nu\partial_{x_i} f \leqslant C|v|V_b(z)||f||_{b,k} \leqslant CV_{b'}(z)||f||_{b,k}\,,
\end{equation*}
which yields
\[||\mathcal{L}_A f||_{b',k-1}\leqslant C||f|_{b,k}\,.\]
If $|\alpha|\leqslant k-2$,
\[
|\partial^\alpha \mathcal{L}_O f| = |-\gamma\partial^{\alpha}(v\cdot\nabla_v f)+\gamma\partial^\alpha\Delta_v f| \leqslant CV_{b'}(z)||f||_{b,k}\,,
\]
and thus
\[||\mathcal{L}_O f||_{b',k-2}\leqslant C||f|_{b,k}\,.\]
Regarding the jump generator, we saw in the proof of Lemma \ref{lem:commutator-jump-Hk} that for all $\nu_1,\nu_2\in\N^d$,
\[\partial_x^{\nu_1}\partial_v^{\nu_2}(f(x,V^i)) =  \rho^{\nu_{2,i}}(\partial_x^{\nu_1}\partial_v^{\nu_2}f)(x,V^i)\,,\]
and that the only non-vanishing derivatives of $\Psi\po \frac{\partial_i U_1}{2}(v_i-V^i_i)\pf$ are of the type $\partial_x^\nu\partial_{v_i}^n \Psi\po \frac{\partial_i U_1}{2}(v_i-V^i_i)\pf$ with
\begin{equation*}
    \left|\partial_x^{\nu}\partial_{v_i^n}\Psi\po\frac{\partial_i U_1}{2}(v_i-V^i_i)\pf\right| \leqslant C\po 1+|v_i-V_i^i|^{|\nu|+n} \pf\,.
\end{equation*}
Therefore, if $|\alpha|\leqslant k$,
\begin{align*}
    \partial^\alpha \mathcal{L}_J f(z) &= \frac{2}{1-\rho}\sum_i\partial^\alpha\E\co(f(x,V^i)-f(x,v))\Psi\po\frac{\partial_i U_1}{2}(v_i-V^i_i)\pf\cf \\
    &= \frac{2}{1-\rho}\sum_i\sum_{\nu\leqslant\alpha}\E\co\partial^\nu (f(x,V^i)-f(x,v))\partial^{\alpha-\nu}\Psi\po\frac{\partial_i U_1}{2}(v_i-V^i_i)\pf\cf \\
    &\leqslant C\frac{2}{1-\rho}\sum_i\sum_{\nu\leqslant\alpha}\E\co\po\rho^{\nu_{d+i}}\partial^\nu f(x,V^i)-\partial^\nu f(x,v)\pf(1+|v_i-V_i^i|^{|\alpha-\nu|})\cf \\
    &\leqslant C||f||_{b,k}\sum_i\sum_{\nu\leqslant\alpha}\E\co \po V_b(x,V^i)+V_b(z)\pf(1+|v_i-V_i^i|^{|\alpha-\nu|})\cf \\
    &\leqslant C||f||_{b,k}\po V_{b'}(z)+\E[(1+|v|^{|\alpha|}+|\new{\xi}|^{|\alpha|})V_b(x,V^i)]\pf\,.
\end{align*}
In the proof of Lemma \ref{lem:Jpartlyapu}, we saw that
\[
    \E_z[\exp(b|\rho v_i + \sqrt{1-\rho^2}\new{\xi}|^2)]\leqslant C V_b(z)\,,
\] 
and that if $a_1<1/2$,
\[\E[|\new{\xi}|^n\exp(a_1\new{\xi}^2+a_2\new{\xi})] = \frac{1}{\sqrt{1-2a_1}}\exp\po\frac{a_2^2}{2(1-2a_1)}\pf\E[|X|^n]\,,\]
where $X\sim\mathcal{N}\po a_2/(1-2a_1),1/(1-2a_1)\pf \sim  a_2/(1-2a_1)+(1/\sqrt{1-2a_1})\new{\xi}$ with $\new{\xi}\sim\mathcal{N}(0,1)$. Therefore,
\[\E[|\new{\xi}|^{|\alpha|}V_b(x,V^i)] \leqslant CV_{b'}(z)\,,\]
and thus
\[\partial^{\alpha}Jf(z) \leqslant C||f||_{b,k}V_{b'}(z)\,,\]
which yields
\[||\mathcal{L}_J||_{b',k}\leqslant C||f||_{b,k}\,.\]
Let us now look at the semi-groups. The derivatives of $U$ being bounded (recall that the position space is a torus in this section), we have for all $t<1$ and $|\alpha|\leqslant k$,
\begin{align*}
    \partial^\alpha e^{tB}f(z) &= \partial^\alpha f(x,v-t\nabla U_0(x)) \\
    &= \alpha!\sum_{(s,p,m)\in\mathcal{D}} (\partial^m f)(x,v-t\nabla U_0(x))\prod_{k=1}^s\frac{1}{m_k!}\co\frac{1}{p_k!}\partial^{p_k} (x,v-t\nabla U_0(x))\cf^{m_k} \\
    &\leqslant C||f||_{b,k} V_b(x,v-t\nabla U_0(x)) \leqslant C(1+C_\veps t)||f||_{b,k} V_{b(1+\veps t)}(z)\,,
\end{align*}
where the last inequality comes from the result of Lemma \ref{lem:Bpartlyapu}. By choosing  $\veps$ small enough, we have $b(1+\veps t)<b'$, hence the result:
\[ ||e^{tB}f||_{b',k}\leqslant C||f||_{b,k}\,.\]
Similarly, for  $t<1$ and $|\alpha|\leqslant k$,
\begin{align*}
    \partial^\alpha e^{tA}f(z) &= \partial^\alpha f(x+tv,v) \\
    &= \alpha!\sum_{(s,p,m)\in\mathcal{D}} (\partial^m f)(x+tv,v)\prod_{k=1}^s\frac{1}{m_k!}\co\frac{1}{p_k!}\partial^{p_k} (x+tv,v)\cf^{m_k} \\
    &\leqslant C ||f||_{b,k}V_b(z)\,,
\end{align*}
which yields
\[ ||e^{tA}f||_{b,k}\leqslant C||f||_{b,k}\,.\]
Again, for $t<1$ and $|\alpha|\leqslant k$,
\begin{align*}
     \partial^\alpha e^{tO}f(z) &= \E[\partial^\alpha f(x,e^{-\gamma t}v+\sqrt{1-e^{-2\gamma t}}\new{\xi})] \\
    &= \alpha!\sum_{(s,p,m)\in\mathcal{D}} \E\co(\partial^m f)(x,e^{-\gamma t}v+\sqrt{1-e^{-2\gamma t}}\new{\xi})\prod_{k=1}^s\frac{1}{m_k!}\co\frac{1}{p_k!}\partial^{p_k} (x,e^{-\gamma t}v+\sqrt{1-e^{-2\gamma t}}\new{\xi})\cf^{m_k}\cf \\
    &\leqslant C||f||_{b,k}\E[V_b(x,e^{-\gamma t}v+\sqrt{1-e^{-2\gamma t}}\new{\xi})] \\
    &\leqslant C||f||_{b,k}(1+C_2 t)|v|^{|\alpha|}V_{b(1-C_1t)}(z) \\
    &\leqslant C||f||_{b,k}V_{b}(z)\,,
\end{align*}
where the second to last inequality comes from the result of Lemma \ref{lem:Opartlyapu}, hence
\[ ||e^{tO}f||_{b,k}\leqslant C||f||_{b,k}\,.\]
Finally, by denoting $(Z_t)$ the jump process after a time $t$ and $Z_n$ the jump process after $n$ jumps,
\begin{align*}
    \partial^\alpha e^{tJ}f(z) &= \sum_{n=0}^\infty \partial^\alpha\E[f(Z_t)\1_{\{n \text{ jumps}\}}]\,.
\end{align*}
Recall that we showed in the proof of Lemma \ref{lem:Jpartlyapu} that for any $b<1/2$, there exist $C>0$ such that
\[\E_z[V_b(Z_\delta)\1_{\{n \text{ jumps}\}}] \leqslant C^n\delta^n(1+V_b(z))\,.\]
As we mentioned earlier, 
\begin{equation*}
    \left|\partial_x^{\nu}\partial_{v_i^n}\Psi\po\frac{\partial_i U_1}{2}(v_i-V_i^i)\pf \right| \leqslant C\po 1+|v_i-V^i_i|^{|\nu|+n} \pf\,,
\end{equation*}
therefore
\[\partial^\alpha\lambda(z)\leqslant C(1+|v|^{|\alpha|})\,.\]
As a consequence, for $t<1$ and $|\alpha|\leqslant k$,
\begin{align*}
    E_z[\partial^\alpha f(Z_t)\1_{\{0 \text{ jumps}\}}] &= \partial^\alpha (f(z) e^{-\lambda(z)t}) \\
    &= \sum_{\nu\leqslant\alpha}\binom{\alpha}{\nu}\partial^\nu(e^{-\lambda(z)t})\partial^{\alpha-\nu}f(z) \\
    &\leqslant C||f||_{b,k}V_b(z)te^{-\lambda(z)t}\sum_{\nu\leqslant\alpha}\alpha!\sum_{(s,p,m)\in\mathcal{D}} \prod_{k=1}^s\frac{1}{m_k!}\co\frac{1}{p_k!}\partial^{p_k}\lambda(z) \cf^{m_k} \\
    &\leqslant C||f||_{b,k}t(1+|v|^{|\alpha|})V_b(z) \\
    &\leqslant C||f||_{b,k}tV_{b'}(z)\,.
\end{align*}
Then, notice that for $k\geq 1$,
\begin{align*}
    \partial^\alpha\E_z[f(Z_\delta)\1_{\{n \text{ jumps}\}}] &= \partial^\alpha \int_0^t\int_{\T^d\times\R^d}\lambda(z)e^{-\lambda(z) s}q(z,\tilde{z})\E_{\tilde{z}}[f(Z_{t-s})\1_{\{n-1 \text{ jumps}\}}]\dd\tilde{z}\dd s \\
    &\leqslant C||f||_{b,k}\int_0^t\int_{\T^d\times\R^d}\partial^{\alpha}(\lambda(z)e^{-\lambda(z) s}q(z,\tilde{z}))\E_{\tilde{z}}[V_b(Z_{t-s})\1_{\{n-1 \text{ jumps}\}}]\dd\tilde{z}\dd s \\
    &\leqslant C^k||f||_{b,k} t^{k-1}(1+V_b(z))\int_0^t\int_{\T^d\times\R^d}\partial^{\alpha}(\lambda(z)e^{-\lambda(z) s}q(z,\tilde{z}))\dd\tilde{z}\dd s\,.
\end{align*}
Let us remind that by denoting $z=(x,v)$ and $\tilde{z}=(\tilde{x},\tilde{v})$,
\[ \lambda(z)q(z,\tilde{z}) \propto \sum_{i=1}^d \Psi\po\frac{\partial_i U_1}{2}\po v_i-\tilde{v}_i\pf\pf \exp\po-\frac{(\tilde{v}_i-\rho v_i)^2}{2(1-\rho^2)}\pf\,,\]
and hence
\begin{align*}
    \partial^\alpha (\lambda(z)q(z,\tilde{z})) \propto \sum_{i=1}^d\sum_{\nu\leqslant\alpha}\partial^{\nu}\Psi\po\frac{\partial_i U_1}{2}\po v_i-\tilde{v}_i\pf\pf\partial^{\alpha-\nu}\exp\po-\frac{(\tilde{v}_i-\rho v_i)^2}{2(1-\rho^2)}\pf\,.
\end{align*}
On one hand,
\[\partial^\nu\Psi\po\frac{\partial_i U_1}{2}\po v_i-\tilde{v}_i\pf\pf\leqslant C(1+|v_i|^{|\nu|}+|\tilde{v}_i|^{|\nu|})\,,\]
and on the other hand
\[\partial^\nu \exp\po-\frac{(\tilde{v}_i-\rho v_i)^2}{2(1-\rho^2)}\pf \leqslant C(|v_i|+|\tilde{v}_i|)\exp\po-\frac{(\tilde{v}_i-\rho v_i)^2}{2(1-\rho^2)}\pf,\]
hence there exist $n_1,n_2\in\N$ such that
\[\partial^\alpha (\lambda(z)q(z,\tilde{z})) \leqslant C(1+|v|^{n_1}+|\tilde{v}|^{n_2})\exp\po-\frac{(\tilde{v}_i-\rho v_i)^2}{2(1-\rho^2)}\pf\,.\]
Then, since we already showed that 
\[\partial^\alpha e^{-\lambda(z)t}\leqslant Ct(1+|v|^{|\alpha|})\,,\]
by combining the obtained bounds we get
\[\partial^{\alpha}(\lambda(z)e^{-\lambda(z) s}q(z,\tilde{z}))\leqslant Cs(1+|v|^{n_1}+|\tilde{v}|^{n_2})\exp\po-\frac{(\tilde{v}_i-\rho v_i)^2}{2(1-\rho^2)}\pf\,.\]
This last expression is integrable with respect to $\tilde{z}$ with
\[\int |\tilde{v}_i|^m\exp\po-\frac{(\tilde{v}_i-\rho v_i)^2}{2(1-\rho^2)}\pf\dd\tilde{v}_i \leqslant C(1+|v|^m)\,,\]
therefore
\[\partial^\alpha\E_z[f(Z_\delta)\1_{\{n \text{ jumps}\}}]\leqslant ||f||_{b,k}C^n t^n V_{b'}(z)\,.\]
If $t<1/C$, we can sum those inequalities over $n$ to get
\[\partial^\alpha e^{tJ}f(z)\leqslant C ||f||_{b,k}V_{b'}(z)\,,\]
thus,
\[||e^{tJ}f||_{b',k}\leqslant C||f||_{b,k}\,.\]
This concludes the proof  of when $t_0< 1 \wedge 1/C$.

To get the result with any $t_0>0$, we simply use that, for an operator $\mathcal L$, $e^{t\mathcal L} = (e^{\frac1n\mathcal L})^n$, and thus it is sufficient to iterate the result with $n$ taken such that $t_0/n < 1 \wedge 1/C$.

\end{proof}

\begin{lem}\label{lem:bornereste}
For all $f\in\mathcal{A}$, all $t\geqslant 0$ and all $s\leqslant t$, there exist $0<b<1/2$, $C>0$ and $q>0$ such that
\begin{equation*}
    |S_3 P_t f(z)| +
    |L^4 P_t f (z)| +
    |\partial^4_s\overline{P}_s P_t f(z)| \leqslant Ce^{-qt}V_b(z).
\end{equation*}
\end{lem}
\begin{proof}
Let $f\in\mathcal{A}$ and $t\geqslant 0$. Theorem \ref{thm:estimates} implies that $ f_t = P_t f-\mu f\in\mathcal{A}$: for all $k\in\N^*$, there exist $C,q>0$ and $b\in(0,1/2)$ such that for all $t>0$,
\[||f_t||_{b,k}\leqslant Ce^{-qt}\,.\]
On one side, $S_3$ and $L_4$ consist in sums of several successive applications of $\mathcal{L}_J$, $\mathcal{L}_A$, $\mathcal{L}_B$ and $\mathcal{L}_O$. Since they are Markov generators (that leave at zero the constants), $S_3P_tf =S_3 f_t$ and $L^4P_t f = L^4 f_t$, and Lemma \ref{lem:stability} yields that for $b'\in(b,1/2)$, there exist $C>0$ such that 
\begin{align*}
    ||S_3 f_t||_{b',k'} &\leqslant C||f||_{b,k}\\
    ||L^4 f_t||_{b',k'} &\leqslant C||f||_{b,k}\,,
\end{align*}
hence the result when $k'=0$. Regarding the derivatives of $\overline{P}_t$, we have 
\begin{align*}
    \partial_t \overline{P}_t = \frac{B}{2}e^{\frac{t}{2}B}e^{\frac{t}{2}J}...e^{\frac{t}{2}B}+e^{\frac{t}{2}B}\frac{J}{2}e^{\frac{t}{2}J}...e^{\frac{t}{2}B}+...e^{\frac{t}{2}B}e^{\frac{t}{2}J}e^{\frac{t}{2}A}e^{tO}e^{\frac{t}{2}A}e^{\frac{t}{2}J}\frac{B}{2}e^{\frac{t}{2}B}\,,
\end{align*}
then, by iterating,
\[\partial^4_t \overline{P}_t = \sum_{k_1,...k_7} \frac{B^{k_1}}{2^{k_1}}e^{\frac{t}{2}B}\frac{J^{k_2}}{2^{k_2}}e^{\frac{t}{2}J}\frac{A^{k_3}}{2^{k_3}}e^{\frac{t}{2}A}O^{k_4}e^{tO}\frac{A^{k_5}}{2^{k_5}}e^{\frac{t}{2}A}\frac{J^{k_6}}{2^{k_6}}e^{\frac{t}{2}J}\frac{B^{k_7}}{2^{k_7}}e^{\frac{t}{2}B}\,,\]
where the sum is over all nonnegative integers $k_1...k_7$ such that $k_1+...+k_7=4$.
Again, the result follows from Lemma \ref{lem:stability}.
\end{proof}

\begin{lem}\label{lem:bch}
For all $t>0$, all $f\in\mathcal{A}$ and all $z\in\T^d\times\R^d$,
\[ \overline{P}_t f(z) - P_t f(z) = t^3 S_3 f(z) + \mathcal{O}(t^4)\,.\]
More explicitly, the term of order $4$ is given by
\[\overline{P}_t f(z) - P_t f(z) - t^3 S_3 f(z)= \frac{1}{6}\int_0^t(t-s)^3(\partial^4_s\overline{P}_s-L^4 P_s)f(z)\dd s\,.\]
\end{lem}
\begin{proof}

We use the same method as in \cite{stoltzsplitting}. Let $f\in\mathcal{A}$ and $z\in\T^d\times\R^d$. An order $3$ Taylor expansion of $P_t f (z)$ and $\overline{P}_t f(z)$  at $t=0$ gives
\[P_t f(z) = f(z) + tLf(z)+\frac{t^2}{2} L^2 f(z) + \frac{t^3}{6}L^3 f(z) + \frac{1}{6}\int_0^t (t-s)^3 L^4 P_sf(z)\dd s\,,\]
\[\overline{P}_t f(z) = f(z) + t\partial_t\overline{P}_t f(z)|_{t=0} + \frac{t^2}{2}\partial^2_t\overline{P}_t f(z)|_{t=0}+\frac{t^3}{6}\partial^3_t\overline{P}_t f(z)|_{t=0}+\frac{1}{6}\int_0^t (t-s)^3 \partial^4_s\overline{P}_sf(z)\dd s\,.\]
We proved in Lemma~\ref{lem:bornereste} that
\begin{align*}
    |L^4 P_t f (z)| &\leqslant Ce^{-qt}V_b(z) \\
    |\partial^4_s\overline{P}_s P_t f(z)| &\leqslant Ce^{-qt}V_b(z)\,,
\end{align*}
which shows that the previous integral remainders can be bound by $C(z)t^4$.
As we saw in the proof of Lemma~\ref{lem:bornereste}, for all $n\in\cco 1,3\ccf$,
\[\partial^n_t \overline{P}_t = \sum_{k_1,...k_7} \frac{B^{k_1}}{2^{k_1}}e^{\frac{t}{2}B}\frac{J^{k_2}}{2^{k_2}}e^{\frac{t}{2}J}\frac{A^{k_3}}{2^{k_3}}e^{\frac{t}{2}A}O^{k_4}e^{tO}\frac{A^{k_5}}{2^{k_5}}e^{\frac{t}{2}A}\frac{J^{k_6}}{2^{k_6}}e^{\frac{t}{2}J}\frac{B^{k_7}}{2^{k_7}}e^{\frac{t}{2}B}\,,\]
where the sum is over all nonnegative integers $k_1\dots k_7$ such that $k_1+\dots+k_7=n$. In fact, thanks to the particular palindromic form of $\overline{P}_t$, those operators can be computed algebraically using the symmetric Baker-Campbell-Hausdorff (BCH) formula (see for instance \cite{BCH2}, Section III.4.2), that states that for any $A,B$ elements of a Lie algebra of a Lie group,
\[\exp\po\frac{t}{2}A\pf\exp\po\frac{t}{2}B\pf\exp\po\frac{t}{2}A\pf = \exp(tS_1+t^3S_3+t^5S_5+...)\]
where
\begin{align*}
    S_1 &= A+B\\
    S_3 &= \frac{1}{24}(2[B,[B,A]] - [A,[A,B]])\,.
\end{align*}
Therefore, in our case,
\[e^{\frac{t}{2} A}e^{t O}e^{\frac{t}{2} A} = e^{t\mathcal{B}_1}\,,\]
with $\mathcal{B}_1 = \hat{S_1}+t^2\hat{S_3}+t^4\hat{S_5}+...$ with $\hat{S_1} = A+O$ and
\[\hat{S_3} = \frac{1}{24}(2[O,[O,A]]-[A,[A,O]])\,.\]
A second application of the formula gives
\[e^{\frac{t}{2} J}e^{t \mathcal{B}_1}e^{\frac{t}{2} J} = e^{t\mathcal{B}_2}\,,\]
where $\mathcal{B}_2 = \tilde{S_1}+t^2\tilde{S_3}+...$ with $\tilde{S_1} = \hat{S_1}+J = A+O+J$ and
\begin{align*}
    \tilde{S_3} &= \hat{S_3}+\frac{1}{24}(2[\hat{S_1},[\hat{S_1},J]]-[J,[J,\hat{S_1}]]) \\
    &= \frac{1}{24}(2[O,[O,A]]-[A,[A,O]] + 2[A+O,[A+O,J]]-[J,[J,A+O]])\,.
\end{align*}
A third and final application gives
\[\overline{P}_t = e^{\frac{t}{2} B}e^{t \mathcal{B}_2}e^{\frac{t}{2} B} = e^{t\mathcal{B}_3}\,,\]
with $\mathcal{B}_3 = \overline{S_1}+t^2\overline{S_3}+t^4 \overline{S_5}+...$ where $\overline{S_1} =\tilde{S_1}+B=A+B+J+O = L$ and 
\begin{align*}
    \overline{S_3} &= \tilde{S_3}+\frac{1}{24}(2[\tilde{S_1},[\tilde{S_1},B]]-[B,[B,\tilde{S_1}]]) \\
    &=\frac{1}{24}(2[O,[O,A]]-[A,[A,O]] + 2[A+O,[A+O,J]]-[J,[J,A+O]] \\ &\qquad +  2[A+J+O,[A+J+O,B]]-[B,[B,A+J+O]])\,.
\end{align*}
What we just showed is that 
\begin{equation*}
    \partial_t\overline{P}_t f(z)|_{t=0} = Lf(z),\qquad
    \partial_t^2\overline{P}_t f(z)|_{t=0} = L^2 f(z),\qquad 
    \partial_t^3\overline{P}_t f(z)|_{t=0} = (L^3 + 7\overline{S_3})f(z)\,.
\end{equation*}
Finally, set $S_3 = \frac{7}{6}\overline{S_3}$, which yields
\[ \overline{P}_t f(z) - P_t f(z) = t^3 S_3f(z) + \frac{1}{6}\int_0^t(t-s)^3(\partial^4_s\overline{P}_s-L^4 P_s)f(z)\dd s\,, \]
with 
\[\frac{1}{6}\left|\int_0^t(t-s)^3(\partial^4_s\overline{P}_s-L^4 P_s)f(z)\dd s\right| \leqslant C(z)t^4\,,\]
which concludes the proof.
\end{proof}

\begin{proof}[Proof of Theorem \ref{thm:finitetimeerror}]
We write the Talay-Tubaro expansion of the weak error.
\begin{align*}
    \E_z[f(Z_{n\delta})]-\E_z[f(\overline{Z}_n)] &= \E[P_{n\delta} f(z) - P_0 f(\overline{Z}_n)] \\
    &= \sum_{k=0}^{n-1} \E_z[ P_{(n-k)\delta} f(\overline{Z}_k) - P_{(n-k-1)\delta} f(\overline{Z}_{k+1})]\,.
\end{align*}
By conditioning each term by $\overline{Z}_k$, we have, since $(\overline{Z}_n)$ is a Markov chain,
\[\E_z[f(Z_{n\delta})]-\E_z[f(\overline{Z}_n)] = \sum_{k=0}^{n-1} \E[h_k(\overline{Z}_k)]\,, \]
where, by using the semi-group property of $(P_t)_{t\geqslant 0}$,
\begin{align*}
h_k(z) &= P_{(n-k)\delta} f(z) - \E[P_{(n-k-1)\delta}f(\overline{Z}_1)] \\
&= P_\delta P_{(n-k-1)\delta}f(z) - \overline{P}_\delta P_{(n-k-1)\delta}f(z)\,.
\end{align*}
We are now interested in the difference $P_\delta\varphi-\overline{P}_\delta\varphi$ for $\varphi\in\mathcal{A}$. As we saw in Lemma~\ref{lem:bch},
\[ \overline{P}_t \varphi(z) - P_t \varphi(z) = t^3 S_3\varphi(z) + \frac{1}{6}\int_0^t(t-s)^3(\partial^4_s\overline{P}_s-L^4 P_s)\varphi(z)\dd s\,. \]
Now, Lemma \ref{lem:bornereste} implies that there exist $b<1/2$ and $C>0$ such that
\[ |h_k(z)| = |P_\delta  P_{(n-k-1)\delta}f(z) - \overline{P}_\delta  P_{(n-k-1)\delta}f(z)| \leqslant C\delta^3 e^{-q(n-k-1)\delta} V_b(z)\,,\]
and then, by using Proposition \ref{prop:lyapunovscheme},
\[|\E_z[h_k(\overline{Z}_k)]|\leqslant C\delta^3 e^{-q(n-k-1)\delta}\E_z[V_b(\overline{Z}_k)] \leqslant  C\delta^3 e^{-q(n-k-1)\delta}\po e^{-C'\delta k}V_b(z)+1\pf\,.\]
By summing those inequalities on $k$, we have
\begin{align*}
    |\E_z[f(Z_t)]-\E_z[f(\overline{Z_n})]| &\leqslant \sum_{k=0}^{n-1}\E[|h_k(\overline{Z_k})|]  \\ 
    &\leqslant C\delta^3\po V_b(z)e^{-q(n-1)\delta}\sum_{k=0}^{n-1} e^{(q-C')k\delta}+\sum_{k=0}^{n-1} e^{-qk\delta}\pf \\
    &\leqslant C\delta^3V_b(z)\frac{e^{-qn\delta}-e^{-C'n\delta}}{e^{-q\delta}-e^{-C'\delta}}+C''\delta^2 \leqslant C\delta^2(V_b(z)e^{-qn\delta}+1)\,,
\end{align*}
which shows the result, with $C(n\delta,z) = C(V_b(z)e^{-qn\delta}+1)$. More explicitly,
\begin{align*}
    \E_z[f(Z_t)]-\E_z[f(\overline{Z_n})] &= \delta^3\po\sum_{k=0}^{n-1}\E[S_3 P_{(n-k-1)\delta}f(\overline{Z_k})]\pf + \frac{1}{6}\sum_{k=0}^{n-1}\int_0^\delta (\delta-s)^3\E_z[(\partial^4_s\overline{P}_s-L^4 P_s)g_k(\overline{Z_k})]\dd s\,,
\end{align*}
and the same computations lead to
\begin{align*}
\left|\sum_{k=0}^{n-1}\E[S_3 P_{(n-k-1)\delta}f(\overline{Z_k})]\right| &\leqslant C\po V_b(z)\frac{e^{-qn\delta}-e^{-C'n\delta}}{e^{-q\delta}-e^{-C'\delta}}+\frac{1}{1-e^{-q\delta}}\pf \leqslant \frac{C}{\delta}(V_b(z)e^{-qn\delta}+1)\,,
\end{align*}
and 
\begin{align*}
\frac{1}{6}\left|\sum_{k=0}^{n-1}\int_0^\delta  (\delta-s)^3\E_z[(\partial^4_s\overline{P}_s-L^4 P_s)g_k(\overline{Z_k})]\dd s\right| &\leqslant C\sum_{k=0}^{n-1}e^{-q(n-k-1)\delta}\E_z[V_b(\overline{Z_k})]\int_0^\delta (\delta-s)^3e^{-qs}\dd s \\
&\leqslant C\delta^4\po V_b(z)\frac{e^{-qn\delta}-e^{-C'n\delta}}{e^{-q\delta}-e^{-C'\delta}}+\frac{1}{1-e^{-q\delta}}\pf \\
&\leqslant C\delta^3(V_b(z)e^{-qn\delta}+1)\,.
\end{align*}
\end{proof}

\subsection{Ergodicity of the numerical scheme}\label{sec:ergodicityscheme}

\begin{proof}[Proof of Theorem~\ref{thm:BJAOAJBergodic}]
The result follows from  Harris Theorem (see for instance \cite[Theorem 1.2]{yetanother}; the constants are made explicit in \cite[Theorem 24]{BPSDGM}). More precisely, we will prove that, for $\delta$ small enough, $Q^{\lfloor 1/\delta\rfloor}$ satisfy the conditions of Harris theorem with constants which are independent from $\delta$. Namely, we need to check, first, a Lyapunov condition
\begin{equation}
    \label{eq:LyapHarris}
Q^{\lfloor 1/\delta\rfloor} V_b \leqslant \gamma V_b + C\,,\end{equation}
where $\gamma\in(0,1)$ and $C>0$ are independent from $\delta$, and, second, a local coupling condition: for any compact set $K$ of $\T^d \times \R^d$, there exists $\alpha>0$ (again, independent from $\delta$) such that for all $(x,v),(x',v')\in K$, 
\begin{equation}
    \label{eq:DoeblinHarris}
\|\delta_{(x,v)} Q^{\lfloor 1/\delta\rfloor} - \delta_{(x',v')} Q^{\lfloor 1/\delta\rfloor}\|_{TV} \leqslant 2 (1-\alpha)\,.
\end{equation}
When we will have proven these two conditions, we will get by \cite[Theorem 1.2]{yetanother} that there exists $C'',\lambda>0$ (which are independent from $\delta\in(0,\delta_0]$ since they can be expressed in terms of $b,\gamma,C$ and $\alpha$ where $\alpha$ is given by the coupling condition on the compact set $K=\{V_b \leqslant 4C/(1-\gamma) \}$) such that for all $n\in\N$ and all probability measures $\nu_1,\nu_2$ over $\T^d\times\R^d$,
\[\|\nu_1Q ^{n \lfloor 1/\delta\rfloor} - \nu_2 Q ^{n \lfloor 1/\delta\rfloor}\|_V \leqslant C'' e^{-\lambda n }\|\nu_1- \nu_2 \|_V\,.\]
Besides, \cite[Theorem 1.2]{yetanother} also gives the existence and uniqueness of an invariant measure $\mu_\delta$ for $Q^{\lfloor 1/\delta\rfloor}$ (which by standard semi-group argument is thus the unique invariant measure for $Q$). Before proceeding with the proofs of~\eqref{eq:LyapHarris} and \eqref{eq:DoeblinHarris}, let us first explain the conclusion of the proof of Theorem~\ref{thm:BJAOAJBergodic} from this. Thanks to~Proposition~\ref{prop:lyapunovscheme}, for all $k\geqslant 1$,
\[Q^k V_b \leqslant (1-C\delta)^k V_b + C' \frac{1-(1-C\delta)^k}{C} \leqslant \po 1+ \frac{C'}{C}\pf V_b\,,\]
since $V_b\geqslant 1$. As noticed in \cite[Equation (86)]{BPSDGM}, this implies that for all $\nu_1,\nu_2$,
\[\|\nu_1Q^k-\nu_2 Q^k\|_V \leqslant \po 1+ \frac{C'}{C}\pf\|\nu_1-\nu_2\|_V\,.\]
Decomposing any $n\geqslant 1$ as $n=m\lfloor 1/\delta\rfloor+k $ with $k\in\cco 0,\lfloor 1/\delta\rfloor-1\ccf$, using the invariance of $\mu_\delta$ by $Q$,
\begin{eqnarray*}
\|\delta_{(x,v)} Q^n - \mu_\delta\|_V & = & \|\delta_{(x,v)} Q^n - \mu_\delta Q^n \|_V \\
& \leqslant &  \po 1+ \frac{C'}{C}\pf \|\delta_{(x,v)} Q^{m\lfloor 1/\delta\rfloor} - \mu_\delta Q^{m\lfloor 1/\delta\rfloor} \|_V \\
& \leqslant & C'' e^{-\lambda m }\po 1+ \frac{C'}{C}\pf \|\delta_{(x,v)} - \mu_\delta  \|_V \\
& \leqslant & C'' e^{-\lambda (n-1)\delta  }\po 1+ \frac{C'}{C}\pf \po V_b(x,v) + \mu_\delta(V_b)\pf \,.
 \end{eqnarray*}
 Thanks to Proposition~\ref{prop:lyapunovscheme} and the invariance of $\mu_\delta$,
 \[\mu_\delta(V_b) \leqslant (1-C\delta) \mu_\delta(V_b)  + C'\delta \implies \mu_\delta(V_b) \leqslant \frac{C'}{C}\,.\]
 Using again that $V_b\geqslant 1$, we end up with 
\[\|\delta_{(x,v)} Q^n - \mu_\delta\|_V \leqslant C'' e^{-\lambda (n\delta-\delta_0)  }\po 1+ \frac{C'}{C}\pf^2 V_b(x,v) \,,\]
with $C,\lambda,C',C''$ which are independent from $\delta\in(0,\delta_0)$. This will conclude the proof of Theorem~\ref{thm:BJAOAJBergodic}.

It remains to check~\eqref{eq:LyapHarris} and \eqref{eq:DoeblinHarris}. The Lyapunov condition~\eqref{eq:LyapHarris} straightforwardly follows from Proposition~\ref{prop:lyapunovscheme}, using that $(1-C\delta)^{\lfloor 1/\delta\rfloor}$ goes to $e^{-C}<1$ as $\delta\rightarrow 0$. The rest of the proof is dedicated to establishing the coupling condition~\eqref{eq:DoeblinHarris}. This result has been established for the BAOAB chain in \cite{durmus2023}. More precisely, given a compact set $K$ of $\T^d\times\R^d$, \cite[Theorem 3]{durmus2023} shows that there exists $\delta_0',\varepsilon>0$ such that for all $\delta\in(0,\delta'_0)$ and $(x,v),(x',v')\in K$,
\[\|\delta_{(x,v)} \widetilde Q^{\lfloor 1/\delta\rfloor} - \delta_{(x',v')} \widetilde Q^{\lfloor 1/\delta\rfloor}\|_{TV} \leqslant 2 (1-\varepsilon)\,,\]
with $\widetilde Q$ the transition operator of the BAOAB chain with potential $U_0$ (i.e. the BJAOAJB chain with $U_1=0$). By the coupling characterisation of the total variation norm, it means that, for any $(x,v),(x',v') \in K$, it is possible to define two BAOAB chains $(X_n,V_n)_{n\in\N}$ and $(X_n',V_n')_{n\in\N}$ initialized respectively at these points which are equal at time $\lfloor 1/\delta\rfloor$ with probability lager than $\varepsilon>0$ (to see that we can define such a trajectory, see e.g. the proof of \cite[Lemma 3.2]{M13}). Denote $\mathcal A_1$ the event where these two BAOAB chains have merged before $\lfloor 1/\delta\rfloor $ transitions. In particular, for all $\delta\in(0,\delta_0')$, $\mathbb P(\mathcal A_1) \geqslant \varepsilon$.

Next, let us prove that there exists a compact set $K'$ of $\T^d\times\R^d$, independent from $\delta$ small enough and from $(x,v)\in K$, such that the probability that the BAOAB chain $(X_n,V_n)_{n\in\N}$ initialized at $(x,v)$ exits $K'$ before $\lfloor 1/\delta\rfloor$ transitions is smaller than $\varepsilon/4$. This classically follows from the fact $V_b$ is a Lyapunov function for the BAOAB chain. Indeed, notice that Assumption~\ref{assu} implies that the same assumptions are satisfied when replacing $U_1$ by $0$, which means that Proposition~\ref{prop:lyapunovscheme} applies to BAOAB. Using that $V_b\geqslant 1$, this implies that $\widetilde Q V_b \leqslant (1+C'\delta) V_b \leqslant e^{C'\delta} V_b$, i.e. $e^{-C'\delta n} e^{b|V_n|^2}$ is a supermartingale. For $R>0$, consider the stopping time $\tau = \inf\{n,|V_n|\geqslant R\}$. Then, letting $n\rightarrow \infty$ in $\mathbb E( e^{-C'\delta (n\wedge \tau)} e^{b|V_{n\wedge \tau}|^2}) \leqslant e^{b|v|^2}$, we get $\mathbb E( e^{-C'\delta  \tau}) \leqslant D e^{-bR^2}$ with $D=\sup\{e^{b|v_0|^2},(x_0,v_0)\in K\}$. In particular, $\mathbb P(\tau \leqslant \lfloor 1/\delta\rfloor ) \leqslant e^{C'} D e^{-bR^2}$. Taking $R$ large enough (depending on $K,b,C'$ and $\varepsilon$, but not on $\delta\in(0,\delta_0)$), conclusion follows with $K'=\T^d\times[0,R]$.

Considering the two coupled BAOAB chains $(X_n,V_n)_{n\in\N}$ and $(X_n',V_n')_{n\in\N}$  as before, denote by $\mathcal A_2$ the event where both chains remain in $K'$ up to time $\lfloor 1/\delta\rfloor$. The choice of $K'$ ensures that $\mathcal P(\mathcal A_2) \geqslant 1-\varepsilon/2$, from which $\mathbb P \po \mathcal A_1\cap \mathcal A_2\pf \geqslant \varepsilon/2$.

Now, we define two BJAOAJB chains by following the two previous BAOAB chains and adding random jumps. More precisely, let $E$ be a standard exponential variable independent from the BAOAB chains and let $S=\inf\{k\geqslant 0, E \leqslant \frac{\delta}{2}\sum_{j=0}^k \lambda(\widetilde X_j,\widetilde V_j)\}$, where $(\widetilde X_j,\widetilde V_j)_{j\in\N}$ are the intermediary of the chain $(X_n,V_n)_{n\in\N}$ (after the first B step of transition, and then after the BAOA steps; which are the two places where the jumps are performed in BJAOAJB, hence the jump rate $\lambda$ is evaluated). We define a BJAOAJB chain $(Y_n,W_n)$ which follows $(X_n,V_n)$ up to half-step $S$ (i.e. up to the intermediary half-step J where a jump occurs), at which point a velocity jump is performed, and after that $(Y_n,W_n)$ is evolved independently from the BAOAB chain.  Similarly we define a BJAOAJB chain $(Y_n',W_n')_{n\in\N}$ initialized at $(x',v')$ from the BAOAB trajectory $(X_n',V_n')_{n\in\N}$ (with the same variable $E$ to define the first jump time). Let $\lambda_*=\sup_{K'}\lambda$ and $\mathcal A_3=\{E>\lambda_*\}$, which is independent from the two BAOAB chains and in particular from $\mathcal A_1$ and $\mathcal A_2$.  Under $\mathcal A_1\cap \mathcal A_2\cap\mathcal A_3$,   the jump rate is bounded by $\lambda_*>0$ along the two coupled BAOAB trajectories and thus $S>2\lfloor 1/\delta\rfloor$, which means there is no jump up to time $\lfloor 1/\delta\rfloor$. Hence, under this event, the two BJAOAJB trajectories coincide with the BAOAB ones and thus have merged at time $\lfloor 1/\delta\rfloor$. As a conclusion,
\[\|\delta_{(x,v)}  Q^{\lfloor 1/\delta\rfloor} - \delta_{(x',v')}  Q^{\lfloor 1/\delta\rfloor}\|_{TV} \leqslant 2 \po 1-\mathbb P\po \mathcal A_1\cap \mathcal A_2\cap\mathcal A_3 \pf \pf \leqslant 2 (1 - e^{-\lambda_*}\varepsilon/2)\,.\]
This concludes the proof of~\eqref{eq:DoeblinHarris}, hence of Theorem~\ref{thm:BJAOAJBergodic}.
\end{proof}

\subsection{Expansion of the invariant measure of the numerical scheme}\label{sec:expansion}
\begin{proof}[Proof of Theorem~\ref{thm:TalayTubaro}]
Let $f\in\mathcal A,\delta_0>0$ such that the result of Theorem~\ref{thm:BJAOAJBergodic} holds and $\delta\in(0,\delta_0]$. The ergodicity of the continuous time process $(Z_t)_{t\geqslant 0}$ (Theorem~\ref{thm:estimates}) and the BJAOAJB chain $(\overline{Z}_n)_{n\geqslant 0}$ (Theorem~\ref{thm:BJAOAJBergodic}) imply, since $\|f\|_{V_b}<\infty$ for some $b\in(0,1/2)$, that
\[
    \mu f = \lim_{n\rightarrow\infty}\frac{1}{n}\sum_{k=1}^n\E_z[f(Z_{k\delta})] \,,\qquad 
    \mu_\delta f = \lim_{n\rightarrow\infty}\frac{1}{n}\sum_{k=1}^n\E_z[f(\overline{Z_k})]\,.
\]
Thanks to  Theorem \ref{thm:finitetimeerror}, for any $n\geqslant 1$,
\begin{align*}
    \left|\frac{1}{n}\sum_{k=1}^n \E_z[f(Z_{k\delta})-f(\overline{Z_k})]\right| &\leqslant \frac{C\delta^2}{n}\sum_{k=1}^n(V_b(z)e^{-nq\delta}+1) \\
    &\leqslant C\delta^2\po \frac{V_b(z)}{n(1-e^{-q\delta})}+1\pf\,. 
\end{align*}
Thus, by taking the limit as $n\rightarrow\infty$ in the previous inequality, we get that
\begin{equation}\label{eq:order2proof}
|\mu f -\mu_\delta f| \leqslant C\delta^2\,.
\end{equation}
Let us now look at the error term of order $\delta^2$. Again, we saw in Theorem~\ref{thm:finitetimeerror} that for all $k\in\N^*$,
\[\left|\E_z[f(Z_{k\delta})-f(\overline{Z_k})]-\delta^3\sum_{p=0}^{k-1}\E_z[S_3P_{(k-p-1)\delta}f(\overline{Z_p})]\right| \leqslant C\delta^3(V_b(z)e^{-qn\delta}+1)\,,\]
thus
\begin{equation}
    \label{eq:order3proof}
\left|\frac{1}{n}\sum_{k=1}^n\po\E_z[f(Z_{k\delta})-f(\overline{Z_k})]-\delta^3\sum_{p=0}^{k-1}\E_z[S_3P_{(k-p-1)\delta}f(\overline{Z_p})]\pf\right|\leqslant C\delta^3 \po \frac{V_b(z)}{n(1-e^{-q\delta})}+1\pf\,.
\end{equation}
Then,
\begin{align*}
    \frac{1}{n}\sum_{k=1}^n\sum_{p=0}^{k-1}\E_z[S_3P_{(k-p-1)\delta}f(\overline{Z_p})] &= \frac{1}{n}\sum_{k=1}^n\sum_{p=0}^{k-1}\E_z[S_3P_{p\delta}f(\overline{Z}_{k-p-1})] \\
    &= \frac{1}{n}\sum_{p=0}^{n-1}\sum_{k=p+1}^{n}\E_z[S_3P_{p\delta}f(\overline{Z}_{k-p-1})] \\
    &= \frac{1}{n}\sum_{p=0}^{n-1}\sum_{k=0}^{n-p-1}\E_z[S_3P_{p\delta}f(\overline{Z}_{k})] \\
    &= \frac{1}{n}\sum_{p=0}^{n-1}\sum_{k=0}^{n}\E_z[S_3P_{p\delta}f(\overline{Z}_{k})]-\frac{1}{n}\sum_{p=0}^{n-1}\sum_{k=n-p}^{n}\E_z[S_3P_{p\delta}f(\overline{Z}_{k})]\,.
\end{align*}
By ergodicity of the numerical scheme, for all $p\in\N$,
\[ \frac{1}{n}\sum_{k=0}^n \E_z[S_3P_{p\delta}f(\overline{Z}_{k})]\underset{n\rightarrow\infty}{\longrightarrow} \mu_\delta S_3P_{p\delta}f\,,\]
and thanks to Lemma~\ref{lem:bornereste} and Proposition~\ref{prop:lyapunovscheme}, there exist $b\in(0,1/2),q>0$ and $C>0$ such that
\begin{align*}
    \frac{1}{n}\left|\sum_{k=0}^n \E_z[S_3P_{p\delta}f(\overline{Z}_{k})]\right| &\leqslant \frac{C}{n}e^{-pq\delta}\sum_{k=0}^n\E_z[V_b(\overline{Z}_k)] \\
   &\leqslant Ce^{-pq\delta}\po \frac{V_b(z)}{n}\sum_{k=0}^n e^{-C'\delta k} + 1\pf \\
   &\leqslant Ce^{-qp\delta}\po \frac{V_b(z)}{n(1-e^{-C'\delta})}+1\pf\,.
\end{align*}
Since the right term of the last inequality is summable in $p$, Lebesgue convergence theorem implies that
\[\frac{1}{n}\sum_{p=0}^{n-1}\sum_{k=0}^{n}\E_z[S_3P_{p\delta}f(\overline{Z}_{k})]\underset{n\rightarrow\infty}{\longrightarrow} \sum_{p=0}^\infty \mu_\delta S_3P_{p\delta}f\,.\]

Similarly, there exist $b\in(0,1/2),q>0$ and $C>0$ such that
\begin{align*}
    \frac{1}{n}\sum_{p=0}^{n-1}\sum_{k=n-p}^n \E_z[|S_3P_{p\delta}f(\overline{Z}_{k})|] &\leqslant \frac{C} {n}\sum_{p=0}^{n-1}e^{-qp\delta}\po V_b(z)\sum_{k=n-p}^n e^{-C'\delta k}+p\pf \\ 
    &\leqslant \frac{C}{n}\sum_{p=0}^{n-1}e^{-qp\delta}\po \frac{V_b(z)}{1-e^{-C'\delta}}+p\pf\underset{n\rightarrow\infty}{\longrightarrow} 0\,,
\end{align*}
which shows that
\[\frac{1}{n}\sum_{k=1}^n\sum_{p=0}^{k-1}\E_z[S_3P_{(k-p-1)\delta}f(\overline{Z_p})]\underset{n\rightarrow\infty}{\longrightarrow} \sum_{p=0}^\infty \mu_\delta S_3P_{p\delta}f\,.\]
We may then let $n\rightarrow \infty$ in \eqref{eq:order3proof} to obtain
\[
\left|\mu f - \mu_\delta f - \delta^3\sum_{p=0}^\infty \mu_\delta S_3P_{p\delta}f\right|\leqslant C\delta^3\,.\]
We have shown that
\[|\mu f - \mu_\delta f - \delta^2\mu_\delta g_\delta|\leqslant C\delta^3\quad\text{ with }\quad g_\delta = \delta\sum_{p=0}^\infty S_3 P_{p\delta} f\,.\]
 Thanks to Theorem~\ref{thm:estimates} and Lemma~\ref{lem:bornereste}, for all $t\geqslant 0$,
\[ |\partial_t S_3 P_t f(z)| = |S_3LP_t f(z)|\leqslant C^{-qt}V_b(z)\,,\]
for a certain $b<1/2$. Therefore, for any $n\in\N^*$,
\begin{align*}
    \left|\int_{0}^{n\delta}S_3P_s f\dd s - \delta\sum_{p=0}^{n-1} S_3P_{p\delta}f \right| &\leqslant \sum_{p=0}^{n-1}\int_{p\delta}^{(p+1)\delta}|S_3P_s f - S_3P_{p\delta} f|\dd s \\
    &\leqslant \sum_{p=0}^{n-1}\sup_{p\delta\leqslant t \leqslant (p+1)\delta}|\partial_t S_3 P_t f|\int_{p\delta}^{(p+1)\delta}(s-p\delta)\dd s \\
    &\leqslant \frac{CV_b(z)\delta^2}{2}\sum_{p=0}^{n-1}e^{-pq\delta} \\
    &\leqslant \frac{CV_b(z)\delta}{2}\,,
\end{align*}
which implies, by taking the limit $n\rightarrow\infty$, that
\[\|g_\delta - g\|_{V_b} \leqslant C\delta \quad \text{ with }\quad g = \int_0^\infty S_3P_sf\dd s\,,\]
and thus
\[|\mu_\delta g_\delta - \mu_\delta g|\leqslant C \delta\,.\]
Again, $g\in\mathcal{A}$. Indeed, if $f\in\mathcal A$, thanks to Lemma~\ref{lem:stability}, we know that all spatial derivatives
of $S_3P_t f$ are bounded by a term of the form $Ce^{-qt}V_b(z)$ for a certain $b<1/2$, which is integrable in $t$ on $\R_{+}$. Therefore, the theorem of differentiation under the integral sign yields that for all $\alpha\in\N^d$,
\[|\partial^\alpha g(z)|\leqslant CV_b(z)\int_{0}^\infty e^{-qs}\dd s \leqslant CV_b(z)\,.\]
This allows to apply~\eqref{eq:order2proof} to the function $g$, that is
\[|\mu_\delta g - \mu g|\leqslant C\delta^2\,.\]
Finally,
\begin{align*}
    |\mu f - \mu_\delta f - \delta^2\mu g| &\leqslant  |\mu f - \mu_\delta f - \delta^2\mu_\delta g_\delta| + \delta^2|\mu_\delta g_\delta - \mu_\delta g| + \delta^2|\mu_\delta g - \mu g| \\
    &\leqslant C\delta^3\,.
\end{align*}
We have therefore explicited the first term in the expansion of the invariant measure:
\[\mu_\delta f = \mu f + \delta^2\int_0^\infty \mu S_3P_sf\dd s+\mathcal{O}(\delta^3)\,,\]
which concludes the proof of the theorem.
\end{proof}

\subsection{Quadratic risk}\label{sec:quadraRisk}

\begin{proof}[Proof of Corollary~\ref{cor:quadraticRisk}]

Denoting $g=f-\mu(f)$ and $\nu_0$ the distribution of $Z_0$,
\begin{eqnarray*}
\mathbb E \po \left|\frac1n\sum_{k=1}^{n}  f(Z_k)-\mu(f)\right|^2\pf   & = & \frac1{n^2} \sum_{j=1}^n \sum_{k=1}^{n} \mathbb E \po   g(Z_k)g(Z_j)\pf \\
&= & \frac1{n^2} \sum_{j=1}^n \nu_0 Q^j(g^2) +  \frac2{n^2} \sum_{j=1}^n \sum_{k=j+1}^{n} \nu_0 Q^j\po g Q^{k-j} g\pf\,. 
\end{eqnarray*}
Bounding $g^2 \leqslant 4\|f\|_{V_b}^2 V_b^2$ and using Proposition~\ref{prop:lyapunovscheme} (applied to $V_b^2=V_{2b}$ since $b<1/4$), we treat the first term as 
\[\frac1{n^2} \sum_{j=1}^n \nu_0 Q^j(g^2) \leqslant \frac{4\|f\|_{V_b}^2}{n^2} \sum_{j=1}^n \nu_0 Q^j(V_b^2) \leqslant \frac{4\|f\|_{V_b}^2}{n^2} \sum_{j=1}^n \co (1-C\delta)^j \nu_0(V_b^2) + \frac{C'}{\delta C}\cf   \,. \]
For the second term, thanks to Theorem~\ref{thm:BJAOAJBergodic}, we bound $\|Q^k g\|_{V_b} \leqslant 2 C e^{-\lambda k\delta}\|f\|_{V_b}$, hence $|g Q^k g| \leqslant 4 C e^{-\lambda k\delta}\|f\|_{V_b}^2V_b^2$, and then
\[
\frac2{n^2} \sum_{j=1}^n \sum_{k=j+1}^{n} \nu_0 Q^j\po g Q^{k-j} g\pf 
\leqslant
\frac{ 8 C \|f\|_{V_b}^2}{n^2} \sum_{j=1}^n \sum_{k=j+1}^{n}  e^{-\lambda (k-j)\delta} \nu_0 Q^j(V_b^2) \\
\leqslant \frac{ 8 C \|f\|_{V_b}^2}{n^2(1-e^{-\lambda \delta})} \sum_{j=1}^n     \nu_0 Q^j(V_b^2) \,,\]
and we bound $\sum_{j=1}^n     \nu_0 Q^j(V_b^2)$ as before.

For the Richardson extrapolation, we simply bound
\begin{multline*}
\mathbb E \po \left|\frac4{3n}\sum_{k=1}^{n}  f(\tilde Z_k)-\frac1{3n}\sum_{k=1}^{n}  f(Z_k)-\mu(f)\right|^2\pf \leqslant 3 \mathbb E \po \left|\frac4{3n}\sum_{k=1}^{n}  f(\tilde Z_k)-\frac43 \mu_{\delta/2}(f)\right|^2\pf \\
+ 3 \mathbb E \po \left|\frac1{3n}\sum_{k=1}^{n}  f( Z_k)-\frac13 \mu_{\delta}(f)\right|^2\pf  +3 \po\frac43 \mu_{\delta/2}(f) - \frac13 \mu_\delta(f) - \mu(f)\pf^2  \,.
\end{multline*}
The two first terms are bounded by $\frac{C}{n\delta}$ for some $C>0$ as before, and conclusion follows from Theorem~\ref{thm:TalayTubaro} (the leading terms of order $\delta^2$ in the asymptotic bias  cancelling out).
\end{proof}

\subsection*{Acknowledgments}

The research of P.M. is supported by the projects SWIDIMS (ANR-20-CE40-0022) and CONVIVIALITY (ANR-23-CE40-0003) of the French National Research Agency. L.J.  is supported by the grant n°200029-21991311 from the Swiss National Science Foundation. N.G. acknowledges a Ph.D. fellowship from Qubit Pharmaceuticals.

\bibliographystyle{abbrv}
\bibliography{biblio} 
\end{document}